\documentclass[12pt]{amsart}
\usepackage[utf8]{inputenc}
\usepackage[english]{babel}
\usepackage{amssymb,amsmath,amsthm,amsfonts,xcolor,enumerate,longtable}
\usepackage{times}
\usepackage{cite}
\usepackage[mathcal]{euscript}
\usepackage{hyperref}

\newtheorem{theorem}{Theorem}
\newtheorem{lemma}[theorem]{Lemma}
\newtheorem{proposition}[theorem]{Proposition}
\newtheorem{remark}[theorem]{Remark}

\newtheorem{definition}[theorem]{Definition}

\begin{document}

\title[Non-split central extensions of naturally graded 2-filiform Leibniz algebras]
{Classification of all non-split central extensions of naturally graded 2-filiform non-Lie Leibniz algebras}

\author{Kobiljon Abdurasulov}
\address{[Kobiljon Abdurasulov] University of Beira Interior, Rua Marquês D'Ávila e Bolama, 6201-001, Covilhã, Portugal}
\email{abdurasulov0505@mail.ru}

\author{Jobir Adashev}
\address{[Jobir Adashev] Institute of Mathematics, Uzbekistan  Academy of  Sciences, 9 University street, 100174, Tashkent,  Uzbekistan}
\email{adashevjq@mail.ru, adashevjq@gmail.com}

\begin{abstract} This paper deals with the naturally graded 2-filiform non-Lie Leibniz algebras over the field of complex numbers. There are two finite series of such algebras, defined in every dimension starting from five, and two further algebras which occur only in dimension five. For each of these algebras we classify, up to isomorphism, all of its non-split central extensions, of every possible dimension. The one-dimensional central extensions are already known from an earlier work, so the main part of the paper is devoted to the extensions of dimension two and higher. As usual, the classification is reduced to the description of the orbits of the automorphism group of the algebra on the subspaces of its second cohomology space with trivial coefficients. For each algebra and each dimension we give the complete list of these orbits together with the multiplication tables of the corresponding algebras, and we point out the extensions whose annihilator coincides with the added central ideal. Among the algebras obtained there are 2-filiform Leibniz algebras which are not naturally graded; this class of algebras has not been systematically studied so far.
\end{abstract}

\subjclass {17A32, 17B30.}
\keywords{Leibniz algebra, 2-filiform algebra, naturally graded, characteristic sequence, 2-cocycles, central extension.}

\maketitle

\section{Introduction}

The idea of a central extension is one of the basic tools of the structure theory of algebras. A central extension of an algebra is a larger algebra which contains a central ideal (an ideal annihilated by all products) such that the quotient by this ideal recovers the original algebra. Central extensions are also important outside algebra: in quantum mechanics a symmetry of a classical system is, in general, realized on the quantum level only projectively, and as a consequence the symmetry group and the symmetry Lie algebra of a quantized system are central extensions of the classical ones. Standard examples are the Heisenberg algebra, which is a central extension of a commutative Lie algebra, and the Virasoro algebra, which is the universal central extension of the Witt algebra \cite{bauerle,cahen}. Such centrally extended algebras appear in conformal field theory and in string theory.

Leibniz algebras are a non-antisymmetric generalization of Lie algebras: the Leibniz identity says that right multiplications act as derivations, and every Lie algebra is a Leibniz algebra. A large part of the structure theory of Leibniz algebras consists of extending classical results on Lie algebras, and in the nilpotent case one of the most fruitful approaches is to study algebras with a large nilindex and their natural gradation. Filiform and 2-filiform Leibniz algebras are defined by means of the characteristic sequence, and the naturally graded algebras of these classes were described in \cite{Omirov1,Comm-2-fil}.

The method of central extensions goes back to Skjelbred and Sund \cite{ss78}, who used it to classify nilpotent Lie algebras of small dimension. The key observation is that every nilpotent algebra with a non-zero annihilator is a central extension of a nilpotent algebra of smaller dimension, and that the isomorphism classes of such extensions correspond to the orbits of the automorphism group of the smaller algebra on the Grassmannians of its second cohomology space. This approach has been developed and applied to many classes of Lie and non-Lie algebras, see \cite{hac16,ha17,zusmanovich,klp20}. In particular, central extensions of filiform and null-filiform Leibniz algebras were studied in \cite{is11,omirov}, of filiform associative algebras in \cite{kkl18}, and of filiform Zinbiel algebras in \cite{Comm-ex-Zin}.

The one-dimensional central extensions of the naturally graded 2-filiform non-Lie Leibniz algebras were studied in \cite{Adashev2023}. In the present paper we go further and classify all non-split central extensions of these algebras, of every dimension. Namely, we consider the two series $\mu_n^1$, $\mu_n^2$ $(n\geq5)$ and the two five-dimensional algebras $\mu_5^3$, $\mu_5^4$, and for each of them we determine, up to isomorphism, the central extensions of all dimensions $s$ for which the second cohomology space admits $s$-dimensional subspaces, that is, $s\leq5$ for $\mu_n^1,\mu_n^2$ and $s\leq4$ for $\mu_5^3,\mu_5^4$. For $s=1$ we recall the results of \cite{Adashev2023}; the cases $s\geq2$ are new. In every case we describe the orbits of the automorphism group on the Grassmannian of the second cohomology space, select the orbits which give extensions without central components and with annihilator equal to the added central ideal, and list the multiplication tables of the resulting algebras in the Appendix.

\section{Preliminaries}

In this section we recall some basic notions and concepts used throughout the paper.

\begin{definition}
A vector space with a bilinear bracket $(\mathcal{L},[\cdot,\cdot])$
is called a Leibniz algebra if for any $x,y,z\in \mathcal{L}$ the so-called
Leibniz identity
\[\big[x,[y,z]\big]=\big[[x,y],z\big]-\big[[x,z],y\big]\]
holds.
\end{definition}

For a Leibniz algebra $\mathcal{L}$ we consider the series
\[
\mathcal{L}^1=\mathcal{L}, \qquad \mathcal{L}^{i+1}=[\mathcal{L}^i, \mathcal{L}], \qquad i\geq 1.
\]

We say that  a Leibniz algebra $\mathcal{L}$ is {\it nilpotent} if $\mathcal{L}^{k}=\{0\}$ for some $k \in \mathbb{N}$. The smallest integer $k$ satisfying $\mathcal{L}^{k}=\{0\}$ is called the  {\it index of nilpotency} or {\it nilindex} of $\mathcal{L}$.

Evidently, for an $n$-dimensional nilpotent Leibniz algebra $\mathcal{L}$ we have $\mathcal{L}^{n+1}=\{0\}.$ An $n$-dimensional Leibniz algebra $\mathcal{L}$ is called null-filiform if the algebra has a maximum nilpotency index. For a nilpotent algebra, the condition of null-filiformity is equivalent to the condition that the algebra is one-generated.

\begin{definition} Given a nilpotent Leibniz algebra $\mathcal{L}$, put
$\mathcal{L}_i=\mathcal{L}^i/\mathcal{L}^{i+1}, \ 1 \leq i\leq k-1$, and ${\rm gr} \mathcal{L} = \mathcal{L}_1 \oplus
\mathcal{L}_2\oplus\dots \oplus \mathcal{L}_{k-1}$. Then $[\mathcal{L}_i,\mathcal{L}_j]\subseteq \mathcal{L}_{i+j}$ and we
obtain the graded algebra ${\rm gr} \mathcal{L}$. If ${\rm gr} \mathcal{L}$ and $\mathcal{L}$ are isomorphic,
denoted by ${\rm gr} \mathcal{L}\cong \mathcal{L}$, we say that the algebra $\mathcal{L}$ is naturally
graded.
\end{definition}

For any element $x$ of $\mathcal{L}$ we define the  right multiplication operator as
\[R_x \colon  \mathcal{L} \rightarrow \mathcal{L},  \quad z \mapsto [z,x], \ z\in\mathcal{L}.\]

Let $x\in\mathcal{L}\setminus\mathcal{L}^2$ and for the nilpotent right multiplication operator $R_x$, let us define the decreasing sequence $C(x)=(n_1,n_2, \dots, n_k)$ that consists of the dimensions of the Jordan blocks of the operator $R_x$. Endow the set of these sequences with the lexicographic order, i.e. $C(x)=(n_1,n_2, \dots, n_k)\leq C(y)=(m_1,m_2, \dots, m_s)$ means that there is an $i\in\mathbb{N}$ such that $n_j=m_j$ for all $j<i$ and $n_i<m_i$.

\begin{definition} The sequence $C(\mathcal{L})=\max\limits_{x\in\mathcal{L}\setminus\mathcal{L}^2}C(x)$ is defined to be the characteristic sequence of the Leibniz algebra $\mathcal{L}$.
\end{definition}

As in the case of Lie algebras, it can be shown that the characteristic
sequence is an isomorphism invariant.

\begin{definition}
A Leibniz algebra $\mathcal{L}$ is called 2-filiform if $C(\mathcal{L})=(n-2,1,1).$
\end{definition}

For a naturally graded 2-filiform nonsplit non-Lie Leibniz algebra the following result holds \cite{Comm-2-fil}.

\begin{theorem} \label{thm2.4} Let $\mathcal{L}$ be an $n$-dimensional ($n\geq 6$) naturally graded 2-filiform nonsplit non-Lie Leibniz
algebra. Then $\mathcal{L}$ is isomorphic to one of the following pairwise
nonisomorphic algebras:
\[\mu_n^1: \quad [e_i,e_1]=e_{i+1}, \  1\leq i \leq {n-3}, \ [e_1,e_{n-1}]=e_{n},\]
\[\mu_n^2:  \quad [e_i,e_1]=e_{i+1}, \ 1\leq i \leq {n-3},\  [e_1,e_{n-1}]=e_2+e_{n}, \
[e_i,e_{n-1}]=e_{i+1}, \  2\leq i \leq {n-3},\]
where all omitted products are equal to zero.
\end{theorem}

\begin{theorem} \label{thm2.5} Let $\mathcal{L}$ be a 5-dimensional naturally graded 2-filiform nonsplit non-Lie Leibniz algebra. Then $\mathcal{L}$ is isomorphic to one of the following pairwise nonisomorphic
algebras:
\[\mu_5^1: \quad [e_1,e_1]=e_{2}, \ [e_2,e_1]=e_{3}, \ [e_1,e_{4}]=e_{5},\]
\[\mu_5^2: \quad [e_1,e_1]=e_{2}, \ [e_2,e_1]=e_{3}, \ [e_1,e_{4}]=e_2+e_{5},\ [e_2,e_{4}]=e_3,\]
\[\mu_5^3: \quad [e_1,e_1]=e_{2}, \ [e_2,e_1]=e_{3}, \ [e_1,e_{4}]=ie_2+e_{5},\ [e_2,e_{4}]=ie_3,\ [e_5,e_{4}]=e_3,\]
\[\mu_5^4: \quad [e_1,e_1]=e_{2}, \ [e_2,e_1]=e_{3}, \ [e_1,e_{4}]=e_{5},\ [e_5,e_{4}]=e_3,\]
where $i$ is the imaginary unit and  all omitted products are equal to zero.
\end{theorem}

Note that $\mu_5^1$ and $\mu_5^2$ are the algebras $\mu_n^1$ and $\mu_n^2$ of Theorem \ref{thm2.4} for $n=5$; therefore below the algebras $\mu_n^1$ and $\mu_n^2$ are considered for all $n\geq 5$.

\subsection{Basic definitions and methods.} We use in this paper the notation and methods described in \cite{hac16,ha17,is11}. Further, in this section we give some important definitions.

Let $(\mathcal{L}, [-,-])$ be a Leibniz algebra and $\mathbb V$ be a vector space. Then the $\mathbb C$-linear space ${\rm Z}^{2}(\mathcal{L},\mathbb{V})$ is defined as the set of all  bilinear maps $\theta\colon\mathcal{L}\times\mathcal{L}\longrightarrow{\mathbb V}$
such that
\[\theta(x, [y,z])=\theta([x,y],z)-\theta([x,z], y).\]
Such maps are said to be {\it cocycles}. Let $f\colon \mathcal{L}\longrightarrow {\mathbb V}$ be a linear map, and define
$\delta f(x, y)=f([x,y])$. Then $\delta f\in {\rm Z}^{2}(\mathcal{L},{\mathbb V})$, i.e. $\delta f$ is a cocycle, called a {\it coboundary}. The set of all
coboundaries is denoted by ${\rm B}^{2}(\mathcal{L},{\mathbb V})=\{\theta =\delta f\ : f\in{\rm Hom}(\mathcal{L},{\mathbb V})\} $. Clearly, ${\rm B}^{2}(\mathcal{L},\mathbb V)$ is a linear subspace of ${\rm Z}^{2}(\mathcal{L},{\mathbb V})$. We call the quotient space, denoted by ${\rm H}^{2}(\mathcal{L},{\mathbb V})={\rm Z}^{2}(\mathcal{L},{\mathbb V})\big/{\rm B}^{2}(\mathcal{L},{\mathbb V}),$ the {\it second cohomology} space of $\mathcal{L}$ by ${\mathbb V}.$

The automorphism group $Aut(\mathcal{L})$ acts on ${\rm Z}^{2}(\mathcal{L},{\mathbb V})$ by $\phi\theta (x,y)
=\theta(\phi(x),\phi(y))$. Also, $\theta\in{\rm B}^{2}(\mathcal{L},{\mathbb V})$ if and only if $\phi\theta\in{\rm B}^{2}(\mathcal{L},{\mathbb V})$, so $Aut(\mathcal{L})$ acts on ${\rm H}^{2}(\mathcal{L},{\mathbb V})$.

Let $\mathcal{L}$ be a Leibniz  algebra of dimension $m<n,$ and ${\mathbb V}$ be a $\mathbb C$-vector
space of dimension $n-m$. For any $\theta \in {\rm Z}^{2}(\mathcal{L},{\mathbb V})$, on the linear space $\mathcal{L}_{\theta } :=\mathcal{L}\oplus{\mathbb V}$ define the bilinear product by $[x+x^{\prime},y+y^{\prime}]_{\mathcal{L}_{\theta }}=
[x,y] +\theta(x,y) $ for all $x,y\in \mathcal{L},x^{\prime },y^{\prime }\in {\mathbb V}$.
The algebra $\mathcal{L}_{\theta }$ is a Leibniz  algebra which is called an $(n-m)$-{\it dimensional central extension} of $\mathcal{L}$ by ${\mathbb V}$. Indeed, $\mathcal{L}_{\theta}$ is a Leibniz algebra if and only if $\theta \in {\rm Z}^{2}(\mathcal{L}, {\mathbb V})$.

Let $\theta$ be a cocycle. Then the annihilator of
$\mathcal{L}_{\theta}$ is equal to $Ann(\mathcal{L}_{\theta}) = (Ann(\theta) \cap Ann(\mathcal{L}) )\oplus \mathbb{V},$ where  $Ann(\theta)=\{ x\in \mathcal{L}:\theta (x,\mathcal{L})+\theta(\mathcal{L},x)=0\}$ is the {\it annihilator} of $\theta $ and $Ann(\mathcal{L}) =\{x\in\mathcal{L}:  [x,\mathcal{L}]+[\mathcal{L}, x] =0\}$ is the {\it annihilator} of the algebra $\mathcal{L}$. Note that if $x\in Ann(\mathcal{L})$, then $\delta f(x,\mathcal{L})=\delta f(\mathcal{L},x)=0$; hence the space $Ann(\theta)\cap Ann(\mathcal{L})$ does not depend on the choice of the representative of the cohomology class of $\theta$.

The result of the following lemma shows that every nilpotent algebra with a non-zero annihilator is a central extension of a smaller-dimensional nilpotent algebra.

\begin{lemma} Let $\mathcal{L}$ be an $n$-dimensional Leibniz algebra such that $\dim(Ann(\mathcal{L}))=m\neq0$. Then there exists, up to isomorphism, a unique $(n-m)$-dimensional Leibniz  algebra $\mathcal{L}'$ and a bilinear map $\theta \in {\rm Z}^{2}(\mathcal{L}', {\mathbb V})$ with $Ann(\mathcal{L}')\cap Ann(\theta)=0$, where $\mathbb{V}$ is a vector space of dimension $m,$ such that $\mathcal{L}\cong {\mathcal{L}'}_{\theta}$ and $\mathcal{L}/Ann(\mathcal{L})\cong \mathcal{L}'$.
\end{lemma}

However, in order to solve the isomorphism problem we need to study the action of $Aut(\mathcal{L})$ on ${\rm H}^{2}(\mathcal{L},{\mathbb V})$. To do this, let us fix $e_{1},\ldots ,e_{s}$ a basis of ${\mathbb V}$, and $\theta \in {\rm Z}^{2}(\mathcal{L},{\mathbb V})$. Then $\theta $ can be uniquely written as $\theta(x,y)=\sum\limits_{i=1}^{s}\theta_{i}(x,y)e_{i}$, where $\theta _{i}\in{\rm Z}^{2}(\mathcal{L},\mathbb{C})$. Moreover, $Ann(\theta)=Ann(\theta_{1})\cap Ann(\theta_{2})\cap\ldots \cap Ann(\theta _{s})$. Further, $\theta \in {\rm B}^{2}(\mathcal{L},{\mathbb V})$ if and only if all $\theta _{i}\in {\rm B}^{2}(\mathcal{L}, \mathbb C)$.

Let $\mathcal{L}=I_1\oplus I_2$ be the direct sum of two ideals. Suppose that $I_2$ is contained in the centre of $\mathcal{L}$. Then $I_2$ is called a central component of $\mathcal{L}$.

A Leibniz algebra $\mathcal{L}=I_1\oplus I_2$ with central component $I_2$ is a trivial central extension of $I_1$ by $I_2,$ i.e. Leibniz algebras with central components are simply obtained by taking direct sums of Leibniz algebras of smaller dimension with trivial Leibniz algebras. Therefore, when constructing Leibniz algebras with a non trivial center as central extensions we want to avoid constructing those with central components. The following lemma will help us to exclude Leibniz algebras with central components.

\begin{lemma}\label{lem-comp} Let $\theta (x,y)=\sum\limits_{i=1}^{s}\theta_{i}(x,y)e_{i}\in {\rm Z}^{2}(\mathcal{L},\mathbb{V})$ and $Ann(\theta)\cap Ann(\mathcal{L}) =0$. Then $\mathcal{L}_{\theta}$ has an annihilator component if and only if $[\theta_{1}],[\theta_{2}],\ldots,[\theta_{s}]$ are linearly dependent in ${\rm H}^{2}(\mathcal{L},\mathbb{C})$.
\end{lemma}

Let $G_{s}({\rm H}^{2}(\mathcal{L},{\mathbb C})) $ be the {\it Grassmannian} of subspaces of dimension $s$ in ${\rm H}^{2}(\mathcal{L},{\mathbb C})$. There is a natural action of $Aut(\mathcal{L})$ on $G_{s}({\rm H}^{2}(\mathcal{L},{\mathbb C})) $ defined by
\[\phi\langle[\theta_{1}],[\theta_{2}],\dots, [\theta_{s}]\rangle=\langle[\phi\theta_{1}], [\phi\theta_{2}], \dots, [\phi\theta_{s}] \rangle\]
for $W=\langle
[\theta_{1}],[\theta_{2}],\dots,[\theta_{s}]\rangle\in G_{s}({\rm H}^{2}(\mathcal{L},{\mathbb C}))$ and $\phi\in Aut(\mathcal{L})$.

Note that if $\{[\theta_{1}],[\theta_{2}],\dots, [\theta_{s}]\}$ is linearly independent, then $\{[\phi\theta_{1}], [\phi\theta_{2}], \dots, [\phi\theta_{s}]\}$ is linearly independent too. So we can introduce the set
\[
T_{s}(\mathcal{L}) =\{W=\langle[\theta_{1}], [\theta_{2}],\dots, [\theta_{s}]\rangle\in
G_{s}({\rm H}^{2}(\mathcal{L},{\mathbb C})) : \bigcap\limits_{i=1}^{s}Ann(\theta _{i})\cap Ann(\mathcal{L})=0\},\]
which is stable under the action of $Aut(\mathcal{L})$.

Now, let ${\mathbb V}$ be an $s$-dimensional linear space and let us denote by $E(\mathcal{L},{\mathbb V})$ the set of all {\it non-split $s$-dimensional central extensions} of $\mathcal{L}$ by ${\mathbb V}$. We can write
\[
E(\mathcal{L},{\mathbb V})=\{\mathcal{L}_{\theta}:\theta(x,y)=\sum_{i=1}^{s}\theta_{i}(x,y)e_{i} \ \ \text{and} \ \ \langle [\theta_{1}], [\theta_{2}],\dots, [\theta_{s}]\rangle\in T_{s}(\mathcal{L})\}. \]

We also have the next result, which is proved in \cite{is11}.

\begin{lemma}\label{lem-iso} Let $\mathcal{L}_{\theta },\mathcal{L}_{\vartheta }\in E(\mathcal{L},{\mathbb V})$. Suppose that $\theta(x,y)=\sum\limits_{i=1}^{s}\theta_{i}(x,y)e_{i}$ and $\vartheta(x,y)=\sum\limits_{i=1}^{s}\vartheta_{i}(x,y)e_{i}$. Then the Leibniz algebras $\mathcal{L}_{\theta}$ and $\mathcal{L}_{\vartheta } $ are isomorphic if and only if \[\operatorname{Orb}\langle[\theta_{1}], [\theta_{2}],\dots,[\theta_{s}]\rangle=\operatorname{Orb}\langle [\vartheta_{1}], [\vartheta_{2}], \dots, [\vartheta_{s}]\rangle .\]
\end{lemma}

In the present paper we determine the $Aut(\mathcal{L})$-orbits on the whole Grassmannian $G_s({\rm H}^{2}(\mathcal{L},{\mathbb C}))$, and we indicate which orbits lie in $T_s(\mathcal{L})$. The meaning of the orbits which do not lie in $T_s(\mathcal{L})$ is explained by the following lemma, which is a slight refinement of Lemma \ref{lem-iso}.

\begin{lemma}\label{lem-orb} Let $\theta,\vartheta\in {\rm Z}^{2}(\mathcal{L},{\mathbb V})$ be such that the classes $[\theta_1],\dots,[\theta_s]$, as well as $[\vartheta_1],\dots,[\vartheta_s]$, are linearly independent, and let $W_\theta=\langle[\theta_1],\dots,[\theta_s]\rangle$, $W_\vartheta=\langle[\vartheta_1],\dots,[\vartheta_s]\rangle$. Then $W_\theta$ and $W_\vartheta$ lie in the same $Aut(\mathcal{L})$-orbit if and only if there exists an isomorphism $\Phi\colon \mathcal{L}_{\theta}\to\mathcal{L}_{\vartheta}$ with $\Phi(\mathbb V)=\mathbb V$. If moreover $W_\theta\in T_s(\mathcal{L})$, then $Ann(\mathcal{L}_\theta)=\mathbb V$, so every isomorphism $\mathcal{L}_{\theta}\to\mathcal{L}_{\vartheta}$ maps $\mathbb V$ onto $\mathbb V$.
\end{lemma}

\begin{proof} Let $\Phi$ be an isomorphism with $\Phi(\mathbb V)=\mathbb V$. Then $\Phi$ induces an automorphism $\phi$ of $\mathcal{L}\cong\mathcal{L}_\theta/\mathbb V\cong \mathcal{L}_\vartheta/\mathbb V$, and we can write $\Phi(x)=\phi(x)+f(x)$ for $x\in\mathcal{L}$ and $\Phi(v)=A v$ for $v\in\mathbb V$, where $f\in {\rm Hom}(\mathcal{L},\mathbb V)$ and $A\in GL(\mathbb V)$. The equality $\Phi([x,y]_{\mathcal{L}_\theta})=[\Phi(x),\Phi(y)]_{\mathcal{L}_\vartheta}$ gives $A\theta(x,y)=\vartheta(\phi(x),\phi(y))-f([x,y])$, i.e. $A\theta=\phi\vartheta-\delta f$. Hence every $[\theta_i]$ is a linear combination of the classes $[\phi\vartheta_j]$, and since both systems are linearly independent we get $W_\theta=\phi W_\vartheta$. Conversely, if $W_\theta=\phi W_\vartheta$, then $A\theta=\phi\vartheta-\delta f$ for some $A\in GL(\mathbb V)$ and $f\in{\rm Hom}(\mathcal{L},\mathbb V)$, and the same formulas define an isomorphism $\Phi$ with $\Phi(\mathbb V)=\mathbb V$. The last assertion follows from $Ann(\mathcal{L}_{\theta}) = (Ann(\theta) \cap Ann(\mathcal{L}) )\oplus \mathbb{V}$.
\end{proof}

Thus there exists a one-to-one correspondence between the set of $Aut(\mathcal{L})$-orbits on $T_{s}(\mathcal{L})$ and the set of isomorphism classes of $E(\mathcal{L},{\mathbb V})$, and the orbits on $G_s({\rm H}^{2}(\mathcal{L},{\mathbb C}))$ classify all $s$-dimensional central extensions $\mathcal{L}_\theta$ without central components, up to isomorphisms preserving $\mathbb V$. For a given Leibniz algebra $\mathcal{L}'$ of dimension $n-s$, we present the following procedure:

\begin{center}
\textsl{The classification procedure:}
\end{center}

\begin{itemize}
\item For a given Leibniz algebra $\mathcal{L}'$ of dimension $n-s$, determine ${\rm Z}^{2}(\mathcal{L}',{\mathbb C}),$  ${\rm B}^{2}(\mathcal{L}',{\mathbb C}),$ ${\rm H}^{2}(\mathcal{L}',{\mathbb C})$ and $Ann(\mathcal{L}')$.

\item Determine the set of $Aut(\mathcal{L}')$-orbits on $G_{s}({\rm H}^{2}(\mathcal{L}',{\mathbb C}))$ and select the orbits lying in $T_{s}(\mathcal{L}')$.

\item For each orbit, construct the Leibniz algebra corresponding to its representative.
\end{itemize}

For $s\geq 3$ we work, exactly as for $s=1,2$, directly with the subspaces $W\in G_s({\rm H})$, where ${\rm H}={\rm H}^{2}(\mathcal{L},{\mathbb C})$, and with bases $\theta_1,\dots,\theta_s$ of $W$. We shall use the following simple observations. Let ${\rm H}_0\subseteq {\rm H}$ be a subspace which is invariant under $Aut(\mathcal{L})$. Then $\phi(W\cap {\rm H}_0)=\phi(W)\cap{\rm H}_0$ and $\phi(W+{\rm H}_0)=\phi(W)+{\rm H}_0$ for every $\phi\in Aut(\mathcal{L})$. Hence $\dim(W\cap{\rm H}_0)$, the orbit of the subspace $V=W\cap {\rm H}_0$ and the orbit of the image of $W$ in ${\rm H}/{\rm H}_0$ do not change when $W$ is replaced by $\phi(W)$. Therefore we may first bring $V$ (or the image of $W$ in ${\rm H}/{\rm H}_0$) to a normal form by a suitable choice of the parameters of the automorphism; after that, two subspaces $W$ and $W'$ for which $V$ has the same normal form lie in the same orbit if and only if they are related by an automorphism which maps this normal form to itself. Only these automorphisms are used in the remaining part of the normalization.

We shall also use the following simple observation. Suppose that, in a suitable basis $\nabla_1,\dots,\nabla_m$ of ${\rm H}$, a subgroup of $Aut(\mathcal{L})$ acts diagonally: $\alpha_i'=\chi_i\alpha_i$, $i=1,\dots,m$. Every $W\in G_s({\rm H})$ has a unique basis in reduced row echelon form (with respect to the order $\nabla_1,\dots,\nabla_m$); the pivot positions $p_1<\dots<p_s$ and the set of non-zero entries of this basis do not change under the diagonal action, and a non-zero entry $x_{rj}$ of the $r$-th basis vector is transformed as $x_{rj}'=\chi_j\chi_{p_r}^{-1}x_{rj}$. Choosing the parameters of the automorphism we normalize as many entries as possible to $1$. If the whole group acts diagonally, the representatives obtained in this way for different pivot positions or different sets of non-zero entries lie in different orbits, and two representatives with the same pattern lie in the same orbit if and only if they are related by an automorphism stabilizing the normalized entries.

Let us introduce some notations. Let $\mathcal{L}$ be a Leibniz algebra with a basis $e_{1},e_{2},\dots,e_{n}$. Then by $\Delta_{i,j}$ we will denote the bilinear form $\Delta_{i,j} \colon \mathcal{L}\times \mathcal{L}\longrightarrow {\mathbb C}$ with $\Delta _{i,j}(e_{l},e_{m})=\delta_{il}\delta_{jm}$. Then the set $\{\Delta_{i,j}:1\leq i, j\leq n\}$ is a basis for the linear space of the bilinear forms on $\mathcal{L}$. Then every $\theta\in{\rm Z}^{2}(\mathcal{L},{\mathbb C})$ can be uniquely written as $\theta=\sum\limits_{1\leq i,j\leq n}c_{ij}\Delta_{{i},{j}}$, where $c_{ij}\in{\mathbb C}$. The $s$-dimensional central extensions of $\mu_n^t$ will be denoted by $\omega^{t,s}_k$ ($k=1,2,\dots$); the central ideal $\mathbb V$ is spanned by $e_{n+1},\dots,e_{n+s}$ (respectively by $e_6,\dots,e_{5+s}$ for $\mu_5^3$ and $\mu_5^4$), and the cocycle $\theta=\sum_{k}\theta_k e_{n+k}$ corresponds to a basis $[\theta_1],\dots,[\theta_s]$ of the subspace $W$. A representative $W$ which lies in $T_s$ is marked by an asterisk.

\section{Central extensions of naturally graded 2-filiform Leibniz algebras}

\begin{proposition}\label{prop10} Let $\mu_n^1,\mu_n^2$ $(n\geq 5)$, $\mu_5^3$ and $\mu_5^4$ be the 2-filiform Leibniz algebras defined in Theorems \ref{thm2.4} and \ref{thm2.5}. Then:
\begin{itemize}
\item A basis of ${\rm Z}^{2}(\mu_k^t,\mathbb{C})$ is formed by the following cocycles
\[\begin{array}{l}
{\rm Z}^{2}(\mu_n^1,\mathbb{C})=\langle\Delta_{1,1},\Delta_{1,n-1},\Delta_{i,1}\ (2\leq i\leq n-1), \Delta_{n-1,n-1}, \Delta_{2,n-1}+\Delta_{n,1}, \Delta_{n,n-1}\rangle, \\[1mm]
{\rm Z}^{2}(\mu_n^2,\mathbb{C})=\langle\Delta_{1,1},\Delta_{1,n-1},\Delta_{i,1}+\Delta_{i,n-1}\ (2\leq i\leq n-2), \Delta_{n-1,1}, \Delta_{n-1,n-1}, -\Delta_{2,1}+\Delta_{n,1}, \Delta_{n,n-1}\rangle, \\[1mm]
{\rm Z}^{2}(\mu_5^3,\mathbb{C})=\langle\Delta_{1,1},\Delta_{1,4},-i\Delta_{2,1}+\Delta_{2,4}, \Delta_{4,1},\Delta_{4,4}, i\Delta_{2,1}+\Delta_{5,1}, \Delta_{5,4}\rangle, \\[1mm]
{\rm Z}^{2}(\mu_5^4,\mathbb{C})=\langle\Delta_{1,1},\Delta_{1,4},\Delta_{2,1}, \Delta_{4,1},\Delta_{4,4}, \Delta_{2,4}+\Delta_{5,1}, \Delta_{5,4}\rangle.
\end{array}\]
\item A basis of ${\rm B}^2(\mu_k^t,\mathbb{C})$  is formed by the following coboundaries
\[\begin{array}{l}
{\rm B}^{2}(\mu_n^1,\mathbb{C})=\langle\Delta_{i,1}\ (1\leq i\leq n-3), \Delta_{1,n-1}\rangle, \\[1mm]
{\rm B}^{2}(\mu_n^2,\mathbb{C})=\langle\Delta_{i,1}+\Delta_{i,n-1}\ (1\leq i\leq n-3), \Delta_{1,n-1}\rangle, \\[1mm]
{\rm B}^{2}(\mu_5^3,\mathbb{C})=\langle\Delta_{1,1}+i\Delta_{1,4}, \Delta_{2,1}+i\Delta_{2,4}+\Delta_{5,4}, \Delta_{1,4}\rangle, \\[1mm]
{\rm B}^{2}(\mu_5^4,\mathbb{C})=\langle\Delta_{1,1}, \Delta_{2,1}+\Delta_{5,4}, \Delta_{1,4}\rangle.
\end{array}\]
\item A basis of ${\rm H}^2(\mu_k^t,\mathbb{C})$ is formed by the following classes
\[\begin{array}{l}
{\rm H}^2(\mu_n^1,\mathbb{C})=\langle[\Delta_{n-2,1}], [\Delta_{n-1,1}], [\Delta_{n-1,n-1}], [\Delta_{2,n-1}+\Delta_{n,1}], [\Delta_{n,n-1}] \rangle,\\[1mm]
{\rm H}^2(\mu_n^2,\mathbb{C})=\langle[\Delta_{n-2,1}+\Delta_{n-2,n-1}], [\Delta_{n-1,1}], [\Delta_{n-1,n-1}], [-\Delta_{2,1}+\Delta_{n,1}], [\Delta_{n,n-1}] \rangle,\\[1mm]
{\rm H}^{2}(\mu_5^3,\mathbb{C})=\langle [-i\Delta_{2,1}+\Delta_{2,4}], [\Delta_{4,1}], [\Delta_{4,4}], [i\Delta_{2,1}+ \Delta_{5,1}]\rangle, \\[1mm]
{\rm H}^{2}(\mu_5^4,\mathbb{C})=\langle[\Delta_{2,1}], [\Delta_{4,1}], [\Delta_{4,4}], [\Delta_{2,4}+\Delta_{5,1}]\rangle.
\end{array}\]
\end{itemize}
In particular, $\dim {\rm H}^2(\mu_n^1,\mathbb{C})=\dim {\rm H}^2(\mu_n^2,\mathbb{C})=5$ and $\dim {\rm H}^2(\mu_5^3,\mathbb{C})=\dim {\rm H}^2(\mu_5^4,\mathbb{C})=4$. Moreover,
$Ann(\mu_n^1)=Ann(\mu_n^2)=\langle e_{n-2},e_n\rangle$ and $Ann(\mu_5^3)=Ann(\mu_5^4)=\langle e_3\rangle$.
\end{proposition}

\begin{remark}\label{rem-34} No cocycle of $\mu_5^3$ or of $\mu_5^4$ involves the index $3$, i.e. $e_3\in Ann(\theta)$ for every $\theta\in {\rm Z}^{2}(\mu_5^t,\mathbb{C})$, $t=3,4$. Since $Ann(\mu_5^3)=Ann(\mu_5^4)=\langle e_3\rangle$, we get $T_s(\mu_5^3)=T_s(\mu_5^4)=\emptyset$ for all $s$: every central extension $(\mu_5^t)_\theta$, $t=3,4$, has the annihilator $\langle e_3\rangle\oplus\mathbb V$ of dimension $s+1$. Therefore the $s$-dimensional central extensions of $\mu_5^3$ and $\mu_5^4$ can also be obtained as $(s+1)$-dimensional central extensions of the four-dimensional algebra $\mu_5^3/\langle e_3\rangle\cong\mu_5^4/\langle e_3\rangle$, which is given by $[e_1,e_1]=e_2,\ [e_1,e_4]=e_5$ (for $\mu_5^3$ one replaces $e_5$ by $ie_2+e_5$).
\end{remark}

\begin{proposition}\label{prop-aut} The automorphism groups of the algebras $\mu_n^1,\mu_n^2,\mu_5^3,\mu_5^4$ consist of the following linear maps $\phi_k^n$ (the $j$-th column of a matrix is formed by the coordinates of $\phi(e_j)$):
{\small
\[\phi_{1}^n= \begin{pmatrix}
a_{1} &0  & 0 &\dots  &  0&  0& 0 \\
a_{2} &a_{1}^2  & 0 & \dots &  0& 0 & 0 \\
a_{3}& a_1a_2 & a_{1}^3 &\dots  &  0&  0& 0 \\
\vdots & \vdots & \vdots &\ddots  & \vdots &  \vdots& \vdots \\
a_{n-2} &a_1a_{n-3}  &a_1^2a_{n-4}  &\dots  & a_{1}^{n-2} &  b_{n-2}& 0 \\
a_{n-1} &0  &0  &\dots  & 0 &  b_{n-1}& 0 \\
a_{n} & a_1a_{n-1} & 0 & \dots &  0& b_n & a_1b_{n-1} \\
\end{pmatrix},\]
\[\phi_{2}^n=\begin{pmatrix}
a_{1} &0  & 0 &\dots  &  0&  0& 0 \\
a_{2} &a_{1}b_{n-1}  & 0 & \dots &  0& 0 & 0 \\
a_{3}& a_2b_{n-1} & a_{1}b_{n-1}^2 &\dots  &  0&  0& 0 \\
\vdots & \vdots & \vdots &\ddots  & \vdots &  \vdots& \vdots \\
a_{n-2} &a_{n-3}b_{n-1}  &a_{n-4}b_{n-1}^2  &\dots  & a_{1}b_{n-1}^{n-3} &  b_{n-2}& 0 \\
a_{n-1} &0  &0  &\dots  & 0 &  b_{n-1}& 0 \\
a_{n} & a_1a_{n-1} & 0 & \dots &  0& b_n & a_1^2 \\
\end{pmatrix},\quad b_{n-1}=a_1+a_{n-1},\]
\[\phi_{3}^5=
\begin{pmatrix}
a_{1} &0  & 0 &  0& 0 \\
a_{2} &a_{1}(a_1+ia_4)  & 0 & 0 & -a_1a_4 \\
a_{3}& a_{2}(a_1+ia_4)+a_4a_5 &a_{1}^2b_4  &  b_3& a_1a_5-a_2a_4+ia_4a_5 \\
a_{4} &0  &0  &  b_{4} & 0 \\
a_{5} & a_1a_4 & 0 & 0 & a_{1}(a_1+ia_4) \\
\end{pmatrix}, \quad b_{4}=a_1+2ia_{4},\]
\[\phi_{4}^5=
\begin{pmatrix}
a_{1} &0  & 0 &  0& 0 \\
a_{2} &a_{1}^2  & 0 & 0 & 0 \\
a_{3}& a_{1}a_2 &a_{1}^3  &  b_3& \varepsilon a_1a_5 \\
0 &0  &0  &  \varepsilon a_{1} & 0 \\
a_{5} & 0 & 0 & 0 & \varepsilon a_{1}^{2} \\
\end{pmatrix}, \quad \varepsilon=\pm1,\]}
where $a_1\neq0$, $b_{n-1}\neq0$, $b_4\neq0$ and the remaining parameters are arbitrary. We have $\det(\phi_1^n)=a_1^{\frac{n^2-3n+4}{2}}b_{n-1}^2$, $\det(\phi_2^n)=a_1^{n}b_{n-1}^{\frac{n^2-5n+8}{2}}$, $\det(\phi_3^5)=a_1^6(a_1+2ia_4)^3$ and $\det(\phi_4^5)=a_1^{9}$.
\end{proposition}

\subsection{Central extensions of $\mu_n^1$} Let us use the following notations:
\[\nabla_1=[\Delta_{n-2,1}], \ \nabla_2=[\Delta_{n-1,1}], \ \nabla_3=[\Delta_{n-1,n-1}], \ \nabla_4= [\Delta_{2,n-1}+\Delta_{n,1}], \ \nabla_5= [\Delta_{n,n-1}].\]
Take $\theta=\alpha_1\nabla_1+\alpha_2\nabla_2+\alpha_3\nabla_3+\alpha_4\nabla_4+\alpha_5\nabla_5\in {\rm H}^{2}(\mu_n^1,\mathbb{C}).$
Computing $\phi_1^n\theta(x,y)=\theta(\phi_1^n(x),\phi_1^n(y))$ and reducing modulo ${\rm B}^{2}(\mu_n^1,\mathbb{C})$
we have that the action of ${\rm Aut} (\mu_n^1)$ on the subspace
$\langle \alpha_1\nabla_1+\alpha_2\nabla_2+\alpha_3\nabla_3+\alpha_4\nabla_4+\alpha_5\nabla_5  \rangle$
is given by
$\langle \alpha_1^\prime\nabla_1+\alpha_2^\prime\nabla_2+\alpha_3^\prime\nabla_3+\alpha_4^\prime\nabla_4+\alpha_5^\prime\nabla_5\rangle,$
where
\[\begin{array}{l}
\alpha_1^\prime=a_1^{n-1}\alpha_1, \ \alpha_2^\prime=a_1(b_{n-2}\alpha_1+b_{n-1}\alpha_2+b_n\alpha_4)+a_{n-1}(b_{n-1}\alpha_3+b_n\alpha_5), \\
\alpha_3^\prime=b_{n-1}(b_{n-1}\alpha_3+b_n\alpha_5), \ \alpha_4^\prime=a_1b_{n-1}(a_1\alpha_4+a_{n-1}\alpha_5),\ \alpha_5^\prime=a_1b_{n-1}^2\alpha_5.
\end{array}\]
Note that only the parameters $a_1,a_{n-1},b_{n-2},b_{n-1},b_n$ appear; $a_1$ and $b_{n-1}$ are arbitrary non-zero numbers and $a_{n-1},b_{n-2},b_n$ are arbitrary. For $a_1=b_{n-1}=1$ the formulas above take the form
\begin{equation}\label{dagger} \alpha_2\mapsto\alpha_2+b_{n-2}\alpha_1+a_{n-1}\alpha_3+b_n\alpha_4+a_{n-1}b_n\alpha_5,\  \alpha_3\mapsto\alpha_3+b_n\alpha_5,\ \alpha_4\mapsto\alpha_4+a_{n-1}\alpha_5,\end{equation}
the coefficients $\alpha_1$, $\alpha_5$ being unchanged, and for $a_{n-1}=b_{n-2}=b_n=0$ they take the diagonal form
\begin{equation}\label{ddagger}\qquad \alpha_1'=a_1^{n-1}\alpha_1,\  \alpha_2'=a_1b_{n-1}\alpha_2,\  \alpha_3'=b_{n-1}^2\alpha_3,\  \alpha_4'=a_1^2b_{n-1}\alpha_4,\  \alpha_5'=a_1b_{n-1}^2\alpha_5.\end{equation}
Every automorphism is a composition of automorphisms of these two types. In particular, the line $\langle\nabla_2\rangle$ and the hyperplane ${\rm H}_0=\langle\nabla_1,\nabla_2,\nabla_3,\nabla_4\rangle=\{\theta:\alpha_5=0\}$ are invariant subspaces, and on the quotient ${\rm H}_0/\langle\nabla_2\rangle$ the action is diagonal: $\alpha_1'=a_1^{n-1}\alpha_1$, $\alpha_3'=b_{n-1}^2\alpha_3$, $\alpha_4'=a_1^2b_{n-1}\alpha_4$.

Let $x=\lambda e_{n-2}+\mu e_n\in Ann(\mu_n^1)$. Then $\theta(x,e_1)=\lambda\alpha_1+\mu\alpha_4$, $\theta(x,e_{n-1})=\mu\alpha_5$ and all the other values $\theta(x,e_j)$, $\theta(e_j,x)$ vanish. Hence a subspace $W$ lies in $T_s(\mu_n^1)$ if and only if the system $\lambda\alpha_1+\mu\alpha_4=0$, $\mu\alpha_5=0$, $\theta\in W$, has only the trivial solution.

\subsubsection{$1$-dimensional central extensions of $\mu_n^1$}
The one-dimensional central extensions of $\mu_n^1$ were classified in \cite{Adashev2023}. In the notation of the present paper the result is as follows: there are $11$ orbits (for $n=5$ the orbit $\langle\nabla_1+\nabla_3+\nabla_4\rangle$ is replaced by the family $\langle \alpha\nabla_1+\nabla_3+\nabla_4 \rangle$):
\begingroup\small
\begin{longtable}{ll}
$\omega^{1,1}_{1}$: & $\langle \nabla_{1}+\nabla_{5}\rangle^{\ast}$\\
$\omega^{1,1}_{2}$: & $\langle \nabla_{5}\rangle$\\
$\omega^{1,1}_{3}$: & $\langle \nabla_{2}+\nabla_{5}\rangle$\\
$\omega^{1,1}_{4}$: & $\langle \nabla_{1}\rangle$\\
$\omega^{1,1}_{5}$: & $\langle \nabla_{3}\rangle$\\
$\omega^{1,1}_{6}$: & $\langle \nabla_{4}\rangle$\\
$\omega^{1,1}_{7}$: & $\langle \nabla_{1}+\nabla_{3}\rangle$\\
$\omega^{1,1}_{8}$: & $\langle \nabla_{1}+\nabla_{4}\rangle$\\
$\omega^{1,1}_{9}$: & $\langle \nabla_{3}+\nabla_{4}\rangle$\\
$\omega^{1,1}_{10}$: & $\langle \nabla_{1}+\nabla_{3}+\nabla_{4}\rangle$ \ ($n\neq 5$); $\langle \alpha \nabla_{1}+\nabla_{3}+\nabla_{4}\rangle$ \ ($n=5,\ \alpha\neq0$)\\
$\omega^{1,1}_{11}$: & $\langle \nabla_{2}\rangle$\\
\end{longtable}
\endgroup

By the criterion above, only the orbit $\langle \nabla_1+\nabla_5\rangle$ lies in $T_1(\mu_n^1)$; the other one-dimensional extensions have the annihilator of dimension $2$ or $3$.

\subsubsection{$2$-dimensional central extensions of $\mu_n^1$}
Let $W\in G_2({\rm H}^{2}(\mu_n^1,\mathbb{C}))$.

\textbf{Case A.} Suppose that $W$ contains an element with $\alpha_5\neq0$. Then we may assume that $W$ is generated by
\[\theta_1  = \alpha_1 \nabla_1+ \alpha_2 \nabla_2+\alpha_3 \nabla_3+\alpha_4 \nabla_4+\alpha_5\nabla_5,\
\theta_2  = \beta_1 \nabla_1+ \beta_2 \nabla_2+\beta_3\nabla_3+\beta_4\nabla_4,\  \alpha_5\neq0.\]
The line $\langle\theta_2\rangle=W\cap\{\alpha_5=0\}$ is uniquely determined by $W$. Putting $a_{n-1}=-\frac{a_1\alpha_4}{\alpha_5}$, $b_n=-\frac{b_{n-1}\alpha_3}{\alpha_5}$ we obtain $\alpha_3'=\alpha_4'=0$, so we may assume $\alpha_3=\alpha_4=0$. The automorphisms with $a_{n-1}=b_n=0$ act by
\[\begin{array}{l}
\alpha_1^\prime=a_1^{n-1}\alpha_1, \ \alpha_2^\prime=a_1(b_{n-2}\alpha_1+b_{n-1}\alpha_2), \ \alpha_5^\prime=a_1b_{n-1}^2\alpha_5,\\
\beta_1^\prime=a_1^{n-1}\beta_1, \ \beta_2^\prime=a_1(b_{n-2}\beta_1+b_{n-1}\beta_2), \ \beta_3^\prime=b_{n-1}^2\beta_3, \ \beta_4^\prime=a_1^2b_{n-1}\beta_4.
\end{array}\]
Moreover, we can replace $\theta_1$ by $\theta_1+t\theta_2$ and then apply the automorphism with $a_1=b_{n-1}=1$, $b_{n-2}=0$, $a_{n-1}=-\frac{t\beta_4}{\alpha_5}$, $b_n=-\frac{t\beta_3}{\alpha_5}$, which restores $\alpha_3=\alpha_4=0$. The result of this transformation is
\begin{equation}\label{ast} \alpha_1\mapsto \alpha_1+t\beta_1,\  \alpha_2\mapsto \alpha_2+t\beta_2-\frac{t^2\beta_3\beta_4}{\alpha_5},\  \beta_2\mapsto\beta_2-\frac{2t\beta_3\beta_4}{\alpha_5},\end{equation}
the other coefficients being unchanged.

\textbf{Case A.1.} Let $\beta_3\beta_4\neq0$. Choosing $t=\frac{\alpha_5\beta_2}{2\beta_3\beta_4}$ in $(\ref{ast})$ we obtain $\beta_2=0$.

If $\beta_1\neq0$, then applying $(\ref{ast})$ with $t=-\frac{\alpha_1}{\beta_1}$ and then choosing $b_{n-2}=-\frac{b_{n-1}\beta_2}{\beta_1}$ we obtain $\alpha_1=\beta_2=0$, i.e. $\theta_1=\alpha_2\nabla_2+\alpha_5\nabla_5$, $\theta_2=\beta_1\nabla_1+\beta_3\nabla_3+\beta_4\nabla_4$.
\begin{itemize}
\item If $\alpha_2\neq0$, then by choosing $b_{n-1}=\frac{\alpha_2}{\alpha_5}$, $a_1=\sqrt{\frac{b_{n-1}\beta_3}{\beta_4}}$ we have the representatives $\langle \nabla_2+\nabla_5,\ \alpha\nabla_1+\nabla_3+\nabla_4\rangle$, $\alpha\neq0$. Since $a_1$ is determined up to sign, $\alpha$ and $(-1)^{n-3}\alpha$ give the same orbit.
\item If $\alpha_2=0$ and $n\neq5$, then by choosing $a_1=\sqrt[n-5]{\frac{\beta_4^2}{\beta_1\beta_3}}$, $b_{n-1}=\frac{a_1^2\beta_4}{\beta_3}$ we have the representative $\langle \nabla_5,\ \nabla_1+\nabla_3+\nabla_4\rangle$. If $n=5$, the number $\frac{\beta_1\beta_3}{\beta_4^2}$ is invariant and we have the representatives $\langle \nabla_5,\ \alpha\nabla_1+\nabla_3+\nabla_4\rangle$, $\alpha\neq0$.
\end{itemize}

If $\beta_1=0$, then $\theta_2=\beta_3\nabla_3+\beta_4\nabla_4$; if moreover $\alpha_1\neq0$, then choosing $b_{n-2}=-\frac{b_{n-1}\alpha_2}{\alpha_1}$ we obtain $\alpha_2=0$.
\begin{itemize}
\item If $\alpha_1=\alpha_2=0$, then by choosing $a_1=1$, $b_{n-1}=\frac{\beta_4}{\beta_3}$ we have the representative $\langle \nabla_5,\ \nabla_3+\nabla_4\rangle$.
\item If $\alpha_1=0$, $\alpha_2\neq0$, then by choosing $b_{n-1}=\frac{\alpha_2}{\alpha_5}$, $a_1=\sqrt{\frac{b_{n-1}\beta_3}{\beta_4}}$ we have the representative $\langle \nabla_2+\nabla_5,\ \nabla_3+\nabla_4\rangle$.
\item If $\alpha_1\neq0$ and $n\neq6$, then by choosing $a_1=\sqrt[n-6]{\frac{\alpha_5\beta_4^2}{\alpha_1\beta_3^2}}$, $b_{n-1}=\frac{a_1^2\beta_4}{\beta_3}$ we have the representative $\langle \nabla_1+\nabla_5,\ \nabla_3+\nabla_4\rangle$. If $n=6$, the number $\frac{\alpha_1\beta_3^2}{\alpha_5\beta_4^2}$ is invariant and we have the representatives $\langle \alpha\nabla_1+\nabla_5,\ \nabla_3+\nabla_4\rangle$, $\alpha\neq0$.
\end{itemize}

\textbf{Case A.2.} Let $\beta_4\neq0$, $\beta_3=0$. Then $(\ref{ast})$ becomes $\alpha_1\mapsto\alpha_1+t\beta_1$, $\alpha_2\mapsto\alpha_2+t\beta_2$.
\begin{itemize}
\item If $\beta_1\neq0$, then choosing $b_{n-2}=-\frac{b_{n-1}\beta_2}{\beta_1}$ and $t=-\frac{\alpha_1}{\beta_1}$ we obtain $\beta_2=\alpha_1=0$. If $\alpha_2=0$, then by choosing $a_1=1$, $b_{n-1}=\frac{\beta_1}{\beta_4}$ we have the representative $\langle \nabla_5,\ \nabla_1+\nabla_4\rangle$; if $\alpha_2\neq0$, then by choosing $b_{n-1}=\frac{\alpha_2}{\alpha_5}$, $a_1=\sqrt[n-3]{\frac{b_{n-1}\beta_4}{\beta_1}}$ we have the representative $\langle \nabla_2+\nabla_5,\ \nabla_1+\nabla_4\rangle$.
\item If $\beta_1=0$, $\beta_2\neq0$, then choosing $t=-\frac{\alpha_2}{\beta_2}$ we obtain $\alpha_2=0$. If $\alpha_1=0$, then by choosing $a_1=\frac{\beta_2}{\beta_4}$, $b_{n-1}=1$ we have the representative $\langle \nabla_5,\ \nabla_2+\nabla_4\rangle$; if $\alpha_1\neq0$, then by choosing $a_1=\frac{\beta_2}{\beta_4}$, $b_{n-1}=\sqrt{\frac{a_1^{n-2}\alpha_1}{\alpha_5}}$ we have the representative $\langle \nabla_1+\nabla_5,\ \nabla_2+\nabla_4\rangle$.
\item If $\beta_1=\beta_2=0$, then $\theta_2=\nabla_4$. If $\alpha_1\neq0$, then by choosing $b_{n-2}=-\frac{b_{n-1}\alpha_2}{\alpha_1}$, $a_1=1$, $b_{n-1}=\sqrt{\frac{\alpha_1}{\alpha_5}}$ we have the representative $\langle \nabla_1+\nabla_5,\ \nabla_4\rangle$. If $\alpha_1=0$, we have the representatives $\langle \nabla_5,\ \nabla_4\rangle$ and (by choosing $b_{n-1}=\frac{\alpha_2}{\alpha_5}$) $\langle \nabla_2+\nabla_5,\ \nabla_4\rangle$, depending on whether $\alpha_2=0$ or not.
\end{itemize}

\textbf{Case A.3.} Let $\beta_3\neq0$, $\beta_4=0$. Then again $(\ref{ast})$ becomes $\alpha_1\mapsto\alpha_1+t\beta_1$, $\alpha_2\mapsto\alpha_2+t\beta_2$, and exactly as in Case A.2 we obtain:
\begin{itemize}
\item for $\beta_1\neq0$: $\langle \nabla_5,\ \nabla_1+\nabla_3\rangle$ (by choosing $a_1=1$, $b_{n-1}=\sqrt{\frac{\beta_1}{\beta_3}}$) and $\langle \nabla_2+\nabla_5,\ \nabla_1+\nabla_3\rangle$ (by choosing $b_{n-1}=\frac{\alpha_2}{\alpha_5}$, $a_1=\sqrt[n-1]{\frac{b_{n-1}^2\beta_3}{\beta_1}}$);
\item for $\beta_1=0$, $\beta_2\neq0$: $\langle \nabla_5,\ \nabla_2+\nabla_3\rangle$ (by choosing $b_{n-1}=1$, $a_1=\frac{\beta_3}{\beta_2}$) and $\langle \nabla_1+\nabla_5,\ \nabla_2+\nabla_3\rangle$ (by choosing $a_1=\sqrt[n-4]{\frac{\alpha_5\beta_2^2}{\alpha_1\beta_3^2}}$, $b_{n-1}=\frac{a_1\beta_2}{\beta_3}$);
\item for $\beta_1=\beta_2=0$: $\langle \nabla_1+\nabla_5,\ \nabla_3\rangle$, $\langle \nabla_5,\ \nabla_3\rangle$, $\langle \nabla_2+\nabla_5,\ \nabla_3\rangle$.
\end{itemize}

\textbf{Case A.4.} Let $\beta_3=\beta_4=0$, $\beta_1\neq0$. Then choosing $b_{n-2}=-\frac{b_{n-1}\beta_2}{\beta_1}$ and $t=-\frac{\alpha_1}{\beta_1}$ we obtain $\theta_2=\nabla_1$, $\theta_1=\alpha_2\nabla_2+\alpha_5\nabla_5$, and we have the representatives $\langle \nabla_5,\ \nabla_1\rangle$, $\langle \nabla_2+\nabla_5,\ \nabla_1\rangle$.

\textbf{Case A.5.} Let $\beta_1=\beta_3=\beta_4=0$. Then $\theta_2=\nabla_2$, choosing $t=-\frac{\alpha_2}{\beta_2}$ we obtain $\alpha_2=0$, and we have the representatives $\langle \nabla_5,\ \nabla_2\rangle$, $\langle \nabla_1+\nabla_5,\ \nabla_2\rangle$ (by choosing $a_1=1$, $b_{n-1}=\sqrt{\frac{\alpha_1}{\alpha_5}}$).

\textbf{Case B.} Suppose that $\alpha_5=0$ for all elements of $W$. Then the action is
\[\alpha_1^\prime=a_1^{n-1}\alpha_1, \ \alpha_2^\prime=a_1(b_{n-2}\alpha_1+b_{n-1}\alpha_2+b_n\alpha_4)+a_{n-1}b_{n-1}\alpha_3, \ \alpha_3^\prime=b_{n-1}^2\alpha_3, \ \alpha_4^\prime=a_1^2b_{n-1}\alpha_4.\]

\textbf{Case B.1.} Let $\nabla_2\in W$. Then $W=\langle\nabla_2,\theta\rangle$ with $\theta=\alpha_1\nabla_1+\alpha_3\nabla_3+\alpha_4\nabla_4\neq0$, and $\theta$ is determined modulo $\nabla_2$, where the action is diagonal: $\alpha_1'=a_1^{n-1}\alpha_1$, $\alpha_3'=b_{n-1}^2\alpha_3$, $\alpha_4'=a_1^2b_{n-1}\alpha_4$. If only one of $\alpha_1,\alpha_3,\alpha_4$ is non-zero, there is nothing to normalize. If $\alpha_1\alpha_3\neq0$, $\alpha_4=0$, we choose $a_1=1$, $b_{n-1}=\sqrt{\frac{\alpha_1}{\alpha_3}}$; if $\alpha_1\alpha_4\neq0$, $\alpha_3=0$, we choose $a_1=1$, $b_{n-1}=\frac{\alpha_1}{\alpha_4}$; if $\alpha_3\alpha_4\neq0$, $\alpha_1=0$, we choose $a_1=1$, $b_{n-1}=\frac{\alpha_4}{\alpha_3}$. If $\alpha_1\alpha_3\alpha_4\neq0$ and $n\neq5$, we choose $a_1=\sqrt[n-5]{\frac{\alpha_4^2}{\alpha_1\alpha_3}}$, $b_{n-1}=\frac{a_1^2\alpha_4}{\alpha_3}$, while for $n=5$ the number $\frac{\alpha_1\alpha_3}{\alpha_4^2}$ is invariant and we choose $a_1=1$, $b_4=\frac{\alpha_4}{\alpha_3}$. Thus we have the representatives
$\langle \nabla_2,\nabla_1\rangle$, $\langle \nabla_2,\nabla_3\rangle$, $\langle \nabla_2,\nabla_4\rangle$, $\langle \nabla_2,\nabla_1+\nabla_3\rangle$, $\langle \nabla_2,\nabla_1+\nabla_4\rangle$, $\langle \nabla_2,\nabla_3+\nabla_4\rangle$ and $\langle \nabla_2,\nabla_1+\nabla_3+\nabla_4\rangle$ (for $n\neq5$), $\langle \nabla_2,\alpha\nabla_1+\nabla_3+\nabla_4\rangle$, $\alpha\neq0$ (for $n=5$).

\textbf{Case B.2.} Let $\nabla_2\notin W$, $W=\langle\theta_1,\theta_2\rangle$. Then the vectors $(\alpha_1,\alpha_3,\alpha_4)$ and $(\beta_1,\beta_3,\beta_4)$ are linearly independent, and the linear system
\[a_1b_{n-2}\alpha_1+a_{n-1}b_{n-1}\alpha_3+a_1b_n\alpha_4=-a_1b_{n-1}\alpha_2,\ a_1b_{n-2}\beta_1+a_{n-1}b_{n-1}\beta_3+a_1b_n\beta_4=-a_1b_{n-1}\beta_2\]
has a solution $(b_{n-2},a_{n-1},b_n)$. Hence we may assume $\alpha_2=\beta_2=0$, i.e. $W\subseteq\langle\nabla_1,\nabla_3,\nabla_4\rangle$, where the action is diagonal: $\alpha_1'=a_1^{n-1}\alpha_1$, $\alpha_3'=b_{n-1}^2\alpha_3$, $\alpha_4'=a_1^2b_{n-1}\alpha_4$. The two-dimensional subspaces of $\langle\nabla_1,\nabla_3,\nabla_4\rangle$ give the representatives
$\langle \nabla_3,\nabla_4\rangle$, $\langle \nabla_1,\nabla_4\rangle$, $\langle \nabla_1,\nabla_3\rangle$, $\langle \nabla_1+\nabla_3,\nabla_4\rangle$, $\langle \nabla_1+\nabla_4,\nabla_3\rangle$, $\langle \nabla_1,\nabla_3+\nabla_4\rangle$, and, for $W=\langle \nabla_1+x\nabla_4,\nabla_3+y\nabla_4\rangle$ with $xy\neq0$: $\langle \nabla_1+\nabla_4,\nabla_3+\nabla_4\rangle$ for $n\neq5$ (by choosing $a_1=\sqrt[n-5]{xy}$, $b_{n-1}=a_1^2y$) and $\langle \nabla_1+\nabla_4,\nabla_3+\alpha\nabla_4\rangle$, $\alpha\neq0$, for $n=5$ (here $xy$ is invariant).

The orbits obtained are pairwise distinct. Indeed, the property ``$W\subseteq\{\alpha_5=0\}$'' is invariant; in Case A the line $W\cap\{\alpha_5=0\}=\langle\theta_2\rangle$ is invariant and, since $\theta_2$ is transformed by $\beta_1'=a_1^{n-1}\beta_1$, $\beta_3'=b_{n-1}^2\beta_3$, $\beta_4'=a_1^2b_{n-1}\beta_4$, the vanishing of $\beta_1$, $\beta_3$, $\beta_4$ is invariant; the remaining distinctions are given by the vanishing of $\alpha_1$, $\alpha_2$ or $\beta_2$ in the normal forms (which cannot be changed by the stabilizer of the normal form) and by the invariants indicated above. Thus we obtain the following $37$ orbits:
\begingroup\small
\begin{longtable}{ll}
$\omega^{1,2}_{1}$: & $\langle \nabla_{2}+\nabla_{5},\ \alpha \nabla_{1}+\nabla_{3}+\nabla_{4}\rangle^{\ast}$ \ ($\alpha\neq0;\ \alpha\sim-\alpha \text{ if } n \text{ is even}$)\\
$\omega^{1,2}_{2}$: & $\langle \nabla_{5},\ \nabla_{1}+\nabla_{3}+\nabla_{4}\rangle^{\ast}$ \ ($n\neq 5$); $\langle \nabla_{5},\ \alpha \nabla_{1}+\nabla_{3}+\nabla_{4}\rangle^{\ast}$ \ ($n=5,\ \alpha\neq0$)\\
$\omega^{1,2}_{3}$: & $\langle \nabla_{5},\ \nabla_{3}+\nabla_{4}\rangle$\\
$\omega^{1,2}_{4}$: & $\langle \nabla_{2}+\nabla_{5},\ \nabla_{3}+\nabla_{4}\rangle$\\
$\omega^{1,2}_{5}$: & $\langle \nabla_{1}+\nabla_{5},\ \nabla_{3}+\nabla_{4}\rangle^{\ast}$ \ ($n\neq 6$); $\langle \alpha \nabla_{1}+\nabla_{5},\ \nabla_{3}+\nabla_{4}\rangle^{\ast}$ \ ($n=6,\ \alpha\neq0$)\\
$\omega^{1,2}_{6}$: & $\langle \nabla_{5},\ \nabla_{1}+\nabla_{4}\rangle^{\ast}$\\
$\omega^{1,2}_{7}$: & $\langle \nabla_{2}+\nabla_{5},\ \nabla_{1}+\nabla_{4}\rangle^{\ast}$\\
$\omega^{1,2}_{8}$: & $\langle \nabla_{5},\ \nabla_{2}+\nabla_{4}\rangle$\\
$\omega^{1,2}_{9}$: & $\langle \nabla_{1}+\nabla_{5},\ \nabla_{2}+\nabla_{4}\rangle^{\ast}$\\
$\omega^{1,2}_{10}$: & $\langle \nabla_{1}+\nabla_{5},\ \nabla_{4}\rangle^{\ast}$\\
$\omega^{1,2}_{11}$: & $\langle \nabla_{5},\ \nabla_{4}\rangle$\\
$\omega^{1,2}_{12}$: & $\langle \nabla_{2}+\nabla_{5},\ \nabla_{4}\rangle$\\
$\omega^{1,2}_{13}$: & $\langle \nabla_{5},\ \nabla_{1}+\nabla_{3}\rangle^{\ast}$\\
$\omega^{1,2}_{14}$: & $\langle \nabla_{2}+\nabla_{5},\ \nabla_{1}+\nabla_{3}\rangle^{\ast}$\\
$\omega^{1,2}_{15}$: & $\langle \nabla_{5},\ \nabla_{2}+\nabla_{3}\rangle$\\
$\omega^{1,2}_{16}$: & $\langle \nabla_{1}+\nabla_{5},\ \nabla_{2}+\nabla_{3}\rangle^{\ast}$\\
$\omega^{1,2}_{17}$: & $\langle \nabla_{1}+\nabla_{5},\ \nabla_{3}\rangle^{\ast}$\\
$\omega^{1,2}_{18}$: & $\langle \nabla_{5},\ \nabla_{3}\rangle$\\
$\omega^{1,2}_{19}$: & $\langle \nabla_{2}+\nabla_{5},\ \nabla_{3}\rangle$\\
$\omega^{1,2}_{20}$: & $\langle \nabla_{5},\ \nabla_{1}\rangle^{\ast}$\\
$\omega^{1,2}_{21}$: & $\langle \nabla_{2}+\nabla_{5},\ \nabla_{1}\rangle^{\ast}$\\
$\omega^{1,2}_{22}$: & $\langle \nabla_{5},\ \nabla_{2}\rangle$\\
$\omega^{1,2}_{23}$: & $\langle \nabla_{1}+\nabla_{5},\ \nabla_{2}\rangle^{\ast}$\\
$\omega^{1,2}_{24}$: & $\langle \nabla_{2},\ \nabla_{1}\rangle$\\
$\omega^{1,2}_{25}$: & $\langle \nabla_{2},\ \nabla_{3}\rangle$\\
$\omega^{1,2}_{26}$: & $\langle \nabla_{2},\ \nabla_{4}\rangle$\\
$\omega^{1,2}_{27}$: & $\langle \nabla_{2},\ \nabla_{1}+\nabla_{3}\rangle$\\
$\omega^{1,2}_{28}$: & $\langle \nabla_{2},\ \nabla_{1}+\nabla_{4}\rangle$\\
$\omega^{1,2}_{29}$: & $\langle \nabla_{2},\ \nabla_{3}+\nabla_{4}\rangle$\\
$\omega^{1,2}_{30}$: & $\langle \nabla_{2},\ \nabla_{1}+\nabla_{3}+\nabla_{4}\rangle$ \ ($n\neq 5$); $\langle \nabla_{2},\ \alpha \nabla_{1}+\nabla_{3}+\nabla_{4}\rangle$ \ ($n=5,\ \alpha\neq0$)\\
$\omega^{1,2}_{31}$: & $\langle \nabla_{3},\ \nabla_{4}\rangle$\\
$\omega^{1,2}_{32}$: & $\langle \nabla_{1},\ \nabla_{4}\rangle^{\ast}$\\
$\omega^{1,2}_{33}$: & $\langle \nabla_{1},\ \nabla_{3}\rangle$\\
$\omega^{1,2}_{34}$: & $\langle \nabla_{1}+\nabla_{3},\ \nabla_{4}\rangle^{\ast}$\\
$\omega^{1,2}_{35}$: & $\langle \nabla_{1}+\nabla_{4},\ \nabla_{3}\rangle$\\
$\omega^{1,2}_{36}$: & $\langle \nabla_{1},\ \nabla_{3}+\nabla_{4}\rangle^{\ast}$\\
$\omega^{1,2}_{37}$: & $\langle \nabla_{1}+\nabla_{4},\ \nabla_{3}+\nabla_{4}\rangle^{\ast}$ \ ($n\neq 5$); $\langle \nabla_{1}+\nabla_{4},\ \nabla_{3}+\alpha \nabla_{4}\rangle^{\ast}$ \ ($n=5,\ \alpha\neq0$)\\
\end{longtable}
\endgroup

\subsubsection{$3$-dimensional central extensions of $\mu_n^1$}
Let $W\in G_3({\rm H}^{2}(\mu_n^1,\mathbb{C}))$. We use the invariant subspaces $\langle\nabla_2\rangle$ and ${\rm H}_0=\{\theta:\alpha_5=0\}$; the parameters $a_{n-1}$, $b_{n-2}$, $b_n$ are used for removing coefficients by means of the formulas $(\ref{dagger})$, and afterwards the parameters $a_1$, $b_{n-1}$ are used for the normalization by means of the formulas $(\ref{ddagger})$.

\textbf{Case A.} Suppose that $W\not\subseteq{\rm H}_0$, i.e. $W$ contains an element with $\alpha_5\neq0$. Then $V=W\cap{\rm H}_0$ is two-dimensional and $W=V\oplus\langle\theta_1\rangle$, where $\theta_1=\alpha_1\nabla_1+\alpha_2\nabla_2+\alpha_3\nabla_3+\alpha_4\nabla_4+\nabla_5$; the element $\theta_1$ is determined by $W$ up to adding elements of $V$.

\textbf{Case A.1.} Let $\nabla_2\in V$. Then $V=\langle\nabla_2,\kappa\rangle$ with $0\neq\kappa=\kappa_1\nabla_1+\kappa_3\nabla_3+\kappa_4\nabla_4$, and modulo $\nabla_2$ the vector $\kappa$ is transformed by $\kappa_1'=a_1^{n-1}\kappa_1$, $\kappa_3'=b_{n-1}^2\kappa_3$, $\kappa_4'=a_1^2b_{n-1}\kappa_4$; in particular, the vanishing of $\kappa_1,\kappa_3,\kappa_4$ is invariant.
\begin{itemize}
\item Let $\kappa_1\neq0$, $\kappa=\nabla_1+x\nabla_3+y\nabla_4$. Subtracting from $\theta_1$ multiples of $\nabla_2$ and $\kappa$ we may assume $\theta_1=\alpha_3\nabla_3+\alpha_4\nabla_4+\nabla_5$. By $(\ref{dagger})$, choosing $a_{n-1}=-\alpha_4$, $b_n=-\alpha_3$, $b_{n-2}=0$ we obtain $\alpha_3=\alpha_4=0$, while $\kappa$ changes only by a multiple of $\nabla_2$. Hence $W=\langle\nabla_2,\nabla_5,\nabla_1+x\nabla_3+y\nabla_4\rangle$, and by $(\ref{ddagger})$ we have $x'=a_1^{1-n}b_{n-1}^2x$, $y'=a_1^{3-n}b_{n-1}y$. We obtain the representatives $\langle\nabla_2,\nabla_5,\nabla_1\rangle$, $\langle\nabla_2,\nabla_5,\nabla_1+\nabla_3\rangle$ (by choosing $a_1=1$, $b_{n-1}=\frac{1}{\sqrt{x}}$), $\langle\nabla_2,\nabla_5,\nabla_1+\nabla_4\rangle$ (by choosing $a_1=1$, $b_{n-1}=\frac{1}{y}$) and, if $xy\neq0$: $\langle\nabla_2,\nabla_5,\nabla_1+\nabla_3+\nabla_4\rangle$ for $n\neq5$ (by choosing $a_1=\sqrt[n-5]{\frac{y^2}{x}}$, $b_{n-1}=\frac{a_1^{n-3}}{y}$); for $n=5$ the number $\frac{x}{y^2}$ is invariant and we have the representatives $\langle\nabla_2,\nabla_5,\nabla_1+\alpha\nabla_3+\nabla_4\rangle$, $\alpha\neq0$.
\item Let $\kappa_1=0$, $\kappa_3\neq0$, $\kappa=\nabla_3+y\nabla_4$. Subtracting from $\theta_1$ multiples of $\nabla_2$ and $\kappa$ we may assume $\theta_1=\alpha_1\nabla_1+\alpha_4\nabla_4+\nabla_5$, and choosing $a_{n-1}=-\alpha_4$ we obtain $\alpha_4=0$ ($\kappa$ changes only by a multiple of $\nabla_2$). Since all elements of $V$ have $\alpha_1=0$, the coefficient $\alpha_1$ of $\theta_1$ is determined by $W$, and by $(\ref{ddagger})$ it is transformed by $\alpha_1'=a_1^{n-2}b_{n-1}^{-2}\alpha_1$ (after the normalization $\alpha_5'=1$), while $y'=a_1^2b_{n-1}^{-1}y$. We obtain the representatives $\langle\nabla_2,\nabla_3,\nabla_5\rangle$, $\langle\nabla_2,\nabla_3,\nabla_1+\nabla_5\rangle$ (by choosing $a_1=1$, $b_{n-1}=\sqrt{\alpha_1}$), $\langle\nabla_2,\nabla_3+\nabla_4,\nabla_5\rangle$ (by choosing $a_1=1$, $b_{n-1}=y$) and, if $\alpha_1y\neq0$: $\langle\nabla_2,\nabla_3+\nabla_4,\nabla_1+\nabla_5\rangle$ for $n\neq6$ (by choosing $a_1=\sqrt[n-6]{\frac{y^2}{\alpha_1}}$, $b_{n-1}=a_1^2y$); for $n=6$ the number $\frac{\alpha_1}{y^2}$ is invariant and we have the representatives $\langle\nabla_2,\nabla_3+\nabla_4,\alpha\nabla_1+\nabla_5\rangle$, $\alpha\neq0$.
\item Let $\kappa=\nabla_4$. Then we may assume $\theta_1=\alpha_1\nabla_1+\alpha_3\nabla_3+\nabla_5$, and choosing $b_n=-\alpha_3$ we obtain $\alpha_3=0$. As above, $\alpha_1'=a_1^{n-2}b_{n-1}^{-2}\alpha_1$, and we have the representatives $\langle\nabla_2,\nabla_4,\nabla_5\rangle$ and $\langle\nabla_2,\nabla_4,\nabla_1+\nabla_5\rangle$ (by choosing $a_1=1$, $b_{n-1}=\sqrt{\alpha_1}$).
\end{itemize}

\textbf{Case A.2.} Let $\nabla_2\notin V$. Then the projection of $V$ to $\langle\nabla_1,\nabla_3,\nabla_4\rangle$ along $\langle\nabla_2\rangle$ is two-dimensional. Since the action on ${\rm H}_0/\langle\nabla_2\rangle$ is diagonal, the positions of the pivots of the reduced row echelon basis of this projection are invariant, and we have the following three cases:
\[{\rm (P1)}\ V=\langle\nabla_1+x\nabla_4+p\nabla_2,\ \nabla_3+y\nabla_4+q\nabla_2\rangle,\ {\rm (P2)}\ V=\langle\nabla_1+x\nabla_3+p\nabla_2,\ \nabla_4+q\nabla_2\rangle,\] \[{\rm (P3)}\ V=\langle\nabla_3+p\nabla_2,\ \nabla_4+q\nabla_2\rangle.\]

\emph{Case {\rm (P1)}.} By $(\ref{dagger})$ the automorphism with parameters $a_{n-1},b_{n-2},b_n$ maps the basis of $V$ to $\nabla_1+x\nabla_4+(p+b_{n-2}+b_nx)\nabla_2$ and $\nabla_3+y\nabla_4+(q+a_{n-1}+b_ny)\nabla_2$. Choosing $b_{n-2}=-p$, $a_{n-1}=-q$, $b_n=0$ we obtain $V=\langle\nabla_1+x\nabla_4,\ \nabla_3+y\nabla_4\rangle$, and this form of $V$ is preserved exactly by the automorphisms with $b_{n-2}=-b_nx$, $a_{n-1}=-b_ny$ (and arbitrary $a_1,b_{n-1}$; by $(\ref{ddagger})$ $x'=a_1^{3-n}b_{n-1}x$, $y'=a_1^2b_{n-1}^{-1}y$). Subtracting from $\theta_1$ elements of $V$ we may assume $\theta_1=\alpha_2\nabla_2+\alpha_4\nabla_4+\nabla_5$. Applying the automorphism with $b_{n-2}=-b_nx$, $a_{n-1}=-b_ny$ and subtracting $b_n(\nabla_3+y\nabla_4)$ we obtain
\begin{equation}\label{astast} \alpha_2\mapsto\alpha_2+b_n\alpha_4-b_n^2y,\quad \alpha_4\mapsto\alpha_4-2b_ny.\end{equation}
\begin{itemize}
\item Let $y\neq0$. Choosing $b_n=\frac{\alpha_4}{2y}$ in $(\ref{astast})$ we obtain $\alpha_4=0$, i.e. $W=\langle\alpha_2\nabla_2+\nabla_5,\ \nabla_1+x\nabla_4,\ \nabla_3+y\nabla_4\rangle$; this form is preserved only by $b_n=0$, and by $(\ref{ddagger})$ we have $\alpha_2'=b_{n-1}^{-1}\alpha_2$. If $\alpha_2x\neq0$, then by choosing $b_{n-1}=\alpha_2$, $a_1=\sqrt{\frac{\alpha_2}{y}}$ we have the representatives $\langle\nabla_2+\nabla_5,\ \nabla_1+\alpha\nabla_4,\ \nabla_3+\nabla_4\rangle$, $\alpha\neq0$; since the stabilizer of the normalized coefficients is $b_{n-1}=1$, $a_1=\pm1$, the numbers $\alpha$ and $(-1)^{n-3}\alpha$ give the same orbit. If $\alpha_2\neq0$, $x=0$, then in the same way we have the representative $\langle\nabla_2+\nabla_5,\ \nabla_1,\ \nabla_3+\nabla_4\rangle$. If $\alpha_2=0$, $x\neq0$, then for $n\neq5$ by choosing $a_1=\sqrt[n-5]{xy}$, $b_{n-1}=\frac{a_1^{n-3}}{x}$ we have the representative $\langle\nabla_5,\ \nabla_1+\nabla_4,\ \nabla_3+\nabla_4\rangle$, and for $n=5$ the number $xy$ is invariant and we have the representatives $\langle\nabla_5,\ \nabla_1+\nabla_4,\ \nabla_3+\alpha\nabla_4\rangle$, $\alpha\neq0$. If $\alpha_2=x=0$, then by choosing $a_1=1$, $b_{n-1}=y$ we have the representative $\langle\nabla_5,\ \nabla_1,\ \nabla_3+\nabla_4\rangle$.
\item Let $y=0$. Then $(\ref{astast})$ becomes $\alpha_2\mapsto\alpha_2+b_n\alpha_4$, and $\alpha_4$ does not change. If $\alpha_4\neq0$, then choosing $b_n=-\frac{\alpha_2}{\alpha_4}$ we obtain $\alpha_2=0$, i.e. $W=\langle\alpha_4\nabla_4+\nabla_5,\ \nabla_1+x\nabla_4,\ \nabla_3\rangle$, where $\alpha_4'=a_1b_{n-1}^{-1}\alpha_4$; we have the representatives $\langle\nabla_4+\nabla_5,\ \nabla_1,\ \nabla_3\rangle$ (by choosing $a_1=1$, $b_{n-1}=\alpha_4$) and $\langle\nabla_4+\nabla_5,\ \nabla_1+\nabla_4,\ \nabla_3\rangle$ (by choosing $a_1=\sqrt[n-4]{\alpha_4x}$, $b_{n-1}=a_1\alpha_4$). If $\alpha_4=0$, then $W=\langle\alpha_2\nabla_2+\nabla_5,\ \nabla_1+x\nabla_4,\ \nabla_3\rangle$ and we have the representatives $\langle\nabla_5,\ \nabla_1,\ \nabla_3\rangle$, $\langle\nabla_2+\nabla_5,\ \nabla_1,\ \nabla_3\rangle$, $\langle\nabla_5,\ \nabla_1+\nabla_4,\ \nabla_3\rangle$, $\langle\nabla_2+\nabla_5,\ \nabla_1+\nabla_4,\ \nabla_3\rangle$ (by choosing $b_{n-1}=\alpha_2$ if $\alpha_2\neq0$, and $a_1=\sqrt[n-3]{b_{n-1}x}$ if $x\neq0$).
\end{itemize}

\emph{Case {\rm (P2)}.} By $(\ref{dagger})$ the basis of $V$ is mapped to $\nabla_1+x\nabla_3+(p+b_{n-2}+a_{n-1}x)\nabla_2$ and $\nabla_4+(q+b_n)\nabla_2$. Choosing $b_{n-2}=-p$, $b_n=-q$, $a_{n-1}=0$ we obtain $V=\langle\nabla_1+x\nabla_3,\ \nabla_4\rangle$, where $x'=a_1^{1-n}b_{n-1}^2x$, and this form is preserved exactly by the automorphisms with $b_n=0$, $b_{n-2}=-a_{n-1}x$. Subtracting from $\theta_1$ elements of $V$ we may assume $\theta_1=\alpha_2\nabla_2+\alpha_3\nabla_3+\nabla_5$. Applying the automorphism with $b_n=0$, $b_{n-2}=-a_{n-1}x$ and subtracting $a_{n-1}\nabla_4$ we obtain $\alpha_2\mapsto\alpha_2+a_{n-1}\alpha_3$, $\alpha_3\mapsto\alpha_3$. If $\alpha_3\neq0$, then choosing $a_{n-1}=-\frac{\alpha_2}{\alpha_3}$ we obtain $\alpha_2=0$; since $\alpha_3'=a_1^{-1}\alpha_3$, we have the representatives $\langle\nabla_3+\nabla_5,\ \nabla_1,\ \nabla_4\rangle$ (by choosing $a_1=\alpha_3$, $b_{n-1}=1$) and $\langle\nabla_3+\nabla_5,\ \nabla_1+\nabla_3,\ \nabla_4\rangle$ (by choosing $a_1=\alpha_3$, $b_{n-1}=\sqrt{\frac{a_1^{n-1}}{x}}$). If $\alpha_3=0$, then $\alpha_2'=b_{n-1}^{-1}\alpha_2$ and we have the representatives $\langle\nabla_5,\ \nabla_1,\ \nabla_4\rangle$, $\langle\nabla_2+\nabla_5,\ \nabla_1,\ \nabla_4\rangle$, $\langle\nabla_5,\ \nabla_1+\nabla_3,\ \nabla_4\rangle$, $\langle\nabla_2+\nabla_5,\ \nabla_1+\nabla_3,\ \nabla_4\rangle$ (by choosing $b_{n-1}=\alpha_2$ if $\alpha_2\neq0$, and $a_1=\sqrt[n-1]{b_{n-1}^2x}$ if $x\neq0$).

\emph{Case {\rm (P3)}.} Choosing $a_{n-1}=-p$, $b_n=-q$ we obtain $V=\langle\nabla_3,\nabla_4\rangle$; this form is preserved exactly by the automorphisms with $a_{n-1}=b_n=0$. We may assume $\theta_1=\alpha_1\nabla_1+\alpha_2\nabla_2+\nabla_5$, and the parameter $b_{n-2}$ gives $\alpha_2\mapsto\alpha_2+b_{n-2}\alpha_1$. If $\alpha_1\neq0$, then by choosing $b_{n-2}=-\frac{\alpha_2}{\alpha_1}$, $a_1=1$, $b_{n-1}=\sqrt{\alpha_1}$ we have the representative $\langle\nabla_1+\nabla_5,\ \nabla_3,\ \nabla_4\rangle$. If $\alpha_1=0$, we have the representatives $\langle\nabla_5,\ \nabla_3,\ \nabla_4\rangle$ and $\langle\nabla_2+\nabla_5,\ \nabla_3,\ \nabla_4\rangle$ (by choosing $b_{n-1}=\alpha_2$).

\textbf{Case B.} Suppose that $W\subseteq{\rm H}_0$.

\textbf{Case B.1.} Let $\nabla_2\in W$. Then $W=\langle\nabla_2\rangle\oplus P$, where $P$ is the two-dimensional subspace of $\langle\nabla_1,\nabla_3,\nabla_4\rangle$ formed by the elements of $W$ with $\alpha_2=0$ (equivalently, the projection of $W$ along $\langle\nabla_2\rangle$), and $P$ is transformed diagonally: $\alpha_1'=a_1^{n-1}\alpha_1$, $\alpha_3'=b_{n-1}^2\alpha_3$, $\alpha_4'=a_1^2b_{n-1}\alpha_4$. Exactly as in Case B.2 of the two-dimensional extensions we have the representatives $\langle\nabla_2,\nabla_3,\nabla_4\rangle$, $\langle\nabla_2,\nabla_1,\nabla_4\rangle$, $\langle\nabla_2,\nabla_1,\nabla_3\rangle$, $\langle\nabla_2,\nabla_1+\nabla_3,\nabla_4\rangle$, $\langle\nabla_2,\nabla_1+\nabla_4,\nabla_3\rangle$, $\langle\nabla_2,\nabla_1,\nabla_3+\nabla_4\rangle$, and $\langle\nabla_2,\nabla_1+\nabla_4,\nabla_3+\nabla_4\rangle$ for $n\neq5$, $\langle\nabla_2,\nabla_1+\nabla_4,\nabla_3+\alpha\nabla_4\rangle$, $\alpha\neq0$, for $n=5$.

\textbf{Case B.2.} Let $\nabla_2\notin W$. Then $W=\langle\nabla_1+x\nabla_2,\ \nabla_3+y\nabla_2,\ \nabla_4+z\nabla_2\rangle$, and by $(\ref{dagger})$, choosing $b_{n-2}=-x$, $a_{n-1}=-y$, $b_n=-z$, we have the representative $\langle\nabla_1,\nabla_3,\nabla_4\rangle$.

The orbits obtained are pairwise distinct. Indeed, the subspaces ${\rm H}_0$ and $\langle\nabla_2\rangle$ are invariant; hence the property $W\subseteq{\rm H}_0$, the conditions $\nabla_2\in W\cap{\rm H}_0$, the orbit of the line $\langle\kappa\rangle$ modulo $\nabla_2$ in Case A.1 and the case {\rm (P1)}, {\rm (P2)}, {\rm (P3)} in Case A.2 are invariants. After $V$ is brought to its normal form, only the automorphisms preserving this normal form can be used; their action on the remaining coefficients of $\theta_1$ (which is determined modulo $V$) is written above, the vanishing of these coefficients is invariant, and the parameters $\alpha$ are invariant up to the indicated finite stabilizers. Thus we obtain the following $37$ orbits (the corresponding algebras are given in the Appendix):
\begingroup\small
\begin{longtable}{ll}
$\omega^{1,3}_{1}$: & $\langle \nabla_{2},\ \nabla_{5},\ \nabla_{1}\rangle^{\ast}$\\
$\omega^{1,3}_{2}$: & $\langle \nabla_{2},\ \nabla_{5},\ \nabla_{1}+\nabla_{3}\rangle^{\ast}$\\
$\omega^{1,3}_{3}$: & $\langle \nabla_{2},\ \nabla_{5},\ \nabla_{1}+\nabla_{4}\rangle^{\ast}$\\
$\omega^{1,3}_{4}$: & $\langle \nabla_{2},\ \nabla_{5},\ \nabla_{1}+\nabla_{3}+\nabla_{4}\rangle^{\ast}$ \ ($n\neq 5$); $\langle \nabla_{2},\ \nabla_{5},\ \nabla_{1}+\alpha \nabla_{3}+\nabla_{4}\rangle^{\ast}$ \ ($n=5,\ \alpha\neq0$)\\
$\omega^{1,3}_{5}$: & $\langle \nabla_{2},\ \nabla_{3},\ \nabla_{5}\rangle$\\
$\omega^{1,3}_{6}$: & $\langle \nabla_{2},\ \nabla_{3},\ \nabla_{1}+\nabla_{5}\rangle^{\ast}$\\
$\omega^{1,3}_{7}$: & $\langle \nabla_{2},\ \nabla_{3}+\nabla_{4},\ \nabla_{5}\rangle$\\
$\omega^{1,3}_{8}$: & $\langle \nabla_{2},\ \nabla_{3}+\nabla_{4},\ \nabla_{1}+\nabla_{5}\rangle^{\ast}$ \ ($n\neq 6$); $\langle \nabla_{2},\ \nabla_{3}+\nabla_{4},\ \alpha \nabla_{1}+\nabla_{5}\rangle^{\ast}$ \ ($n=6,\ \alpha\neq0$)\\
$\omega^{1,3}_{9}$: & $\langle \nabla_{2},\ \nabla_{4},\ \nabla_{5}\rangle$\\
$\omega^{1,3}_{10}$: & $\langle \nabla_{2},\ \nabla_{4},\ \nabla_{1}+\nabla_{5}\rangle^{\ast}$\\
$\omega^{1,3}_{11}$: & $\langle \nabla_{2}+\nabla_{5},\ \nabla_{1}+\alpha \nabla_{4},\ \nabla_{3}+\nabla_{4}\rangle^{\ast}$ \ ($\alpha\neq0;\ \alpha\sim-\alpha \text{ if } n \text{ is even}$)\\
$\omega^{1,3}_{12}$: & $\langle \nabla_{2}+\nabla_{5},\ \nabla_{1},\ \nabla_{3}+\nabla_{4}\rangle^{\ast}$\\
$\omega^{1,3}_{13}$: & $\langle \nabla_{5},\ \nabla_{1}+\nabla_{4},\ \nabla_{3}+\nabla_{4}\rangle^{\ast}$ \ ($n\neq 5$); $\langle \nabla_{5},\ \nabla_{1}+\nabla_{4},\ \nabla_{3}+\alpha \nabla_{4}\rangle^{\ast}$ \ ($n=5,\ \alpha\neq0$)\\
$\omega^{1,3}_{14}$: & $\langle \nabla_{5},\ \nabla_{1},\ \nabla_{3}+\nabla_{4}\rangle^{\ast}$\\
$\omega^{1,3}_{15}$: & $\langle \nabla_{4}+\nabla_{5},\ \nabla_{1},\ \nabla_{3}\rangle^{\ast}$\\
$\omega^{1,3}_{16}$: & $\langle \nabla_{4}+\nabla_{5},\ \nabla_{1}+\nabla_{4},\ \nabla_{3}\rangle^{\ast}$\\
$\omega^{1,3}_{17}$: & $\langle \nabla_{5},\ \nabla_{1},\ \nabla_{3}\rangle^{\ast}$\\
$\omega^{1,3}_{18}$: & $\langle \nabla_{2}+\nabla_{5},\ \nabla_{1},\ \nabla_{3}\rangle^{\ast}$\\
$\omega^{1,3}_{19}$: & $\langle \nabla_{5},\ \nabla_{1}+\nabla_{4},\ \nabla_{3}\rangle^{\ast}$\\
$\omega^{1,3}_{20}$: & $\langle \nabla_{2}+\nabla_{5},\ \nabla_{1}+\nabla_{4},\ \nabla_{3}\rangle^{\ast}$\\
$\omega^{1,3}_{21}$: & $\langle \nabla_{3}+\nabla_{5},\ \nabla_{1},\ \nabla_{4}\rangle^{\ast}$\\
$\omega^{1,3}_{22}$: & $\langle \nabla_{3}+\nabla_{5},\ \nabla_{1}+\nabla_{3},\ \nabla_{4}\rangle^{\ast}$\\
$\omega^{1,3}_{23}$: & $\langle \nabla_{5},\ \nabla_{1},\ \nabla_{4}\rangle^{\ast}$\\
$\omega^{1,3}_{24}$: & $\langle \nabla_{2}+\nabla_{5},\ \nabla_{1},\ \nabla_{4}\rangle^{\ast}$\\
$\omega^{1,3}_{25}$: & $\langle \nabla_{5},\ \nabla_{1}+\nabla_{3},\ \nabla_{4}\rangle^{\ast}$\\
$\omega^{1,3}_{26}$: & $\langle \nabla_{2}+\nabla_{5},\ \nabla_{1}+\nabla_{3},\ \nabla_{4}\rangle^{\ast}$\\
$\omega^{1,3}_{27}$: & $\langle \nabla_{1}+\nabla_{5},\ \nabla_{3},\ \nabla_{4}\rangle^{\ast}$\\
$\omega^{1,3}_{28}$: & $\langle \nabla_{5},\ \nabla_{3},\ \nabla_{4}\rangle$\\
$\omega^{1,3}_{29}$: & $\langle \nabla_{2}+\nabla_{5},\ \nabla_{3},\ \nabla_{4}\rangle$\\
$\omega^{1,3}_{30}$: & $\langle \nabla_{2},\ \nabla_{3},\ \nabla_{4}\rangle$\\
$\omega^{1,3}_{31}$: & $\langle \nabla_{2},\ \nabla_{1},\ \nabla_{4}\rangle^{\ast}$\\
$\omega^{1,3}_{32}$: & $\langle \nabla_{2},\ \nabla_{1},\ \nabla_{3}\rangle$\\
$\omega^{1,3}_{33}$: & $\langle \nabla_{2},\ \nabla_{1}+\nabla_{3},\ \nabla_{4}\rangle^{\ast}$\\
$\omega^{1,3}_{34}$: & $\langle \nabla_{2},\ \nabla_{1}+\nabla_{4},\ \nabla_{3}\rangle$\\
$\omega^{1,3}_{35}$: & $\langle \nabla_{2},\ \nabla_{1},\ \nabla_{3}+\nabla_{4}\rangle^{\ast}$\\
$\omega^{1,3}_{36}$: & $\langle \nabla_{2},\ \nabla_{1}+\nabla_{4},\ \nabla_{3}+\nabla_{4}\rangle^{\ast}$ \ ($n\neq 5$); $\langle \nabla_{2},\ \nabla_{1}+\nabla_{4},\ \nabla_{3}+\alpha \nabla_{4}\rangle^{\ast}$ \ ($n=5,\ \alpha\neq0$)\\
$\omega^{1,3}_{37}$: & $\langle \nabla_{1},\ \nabla_{3},\ \nabla_{4}\rangle^{\ast}$\\
\end{longtable}
\endgroup

\subsubsection{$4$-dimensional central extensions of $\mu_n^1$}
Let $W\in G_4({\rm H}^{2}(\mu_n^1,\mathbb{C}))$. If $W\subseteq{\rm H}_0$, then $W={\rm H}_0=\langle\nabla_1,\nabla_2,\nabla_3,\nabla_4\rangle$. Otherwise $V=W\cap{\rm H}_0$ is three-dimensional and $W=V\oplus\langle\theta_1\rangle$ with $\theta_1=\alpha_1\nabla_1+\alpha_2\nabla_2+\alpha_3\nabla_3+\alpha_4\nabla_4+\nabla_5$, the element $\theta_1$ being determined modulo $V$.

\textbf{Case 1.} Let $\nabla_2\in V$. Then $V=\langle\nabla_2\rangle\oplus P$, where $P$ is a two-dimensional subspace of $\langle\nabla_1,\nabla_3,\nabla_4\rangle$, and modulo $\nabla_2$ the subspace $P$ is transformed diagonally; hence the pivots of its reduced row echelon basis are invariant.
\begin{itemize}
\item Let $P=\langle\nabla_1+x\nabla_4,\ \nabla_3+y\nabla_4\rangle$. Subtracting from $\theta_1$ elements of $V$ we may assume $\theta_1=\alpha_4\nabla_4+\nabla_5$, and by $(\ref{dagger})$ the choice $a_{n-1}=-\alpha_4$ gives $\alpha_4=0$, while the elements of $V$ change only by multiples of $\nabla_2$. Thus $W=\langle\nabla_1+x\nabla_4,\nabla_2,\nabla_3+y\nabla_4,\nabla_5\rangle$, where $x'=a_1^{3-n}b_{n-1}x$, $y'=a_1^2b_{n-1}^{-1}y$, and we have the representatives $\langle\nabla_1,\nabla_2,\nabla_3,\nabla_5\rangle$, $\langle\nabla_1+\nabla_4,\nabla_2,\nabla_3,\nabla_5\rangle$ (by choosing $a_1=1$, $b_{n-1}=\frac1x$), $\langle\nabla_1,\nabla_2,\nabla_3+\nabla_4,\nabla_5\rangle$ (by choosing $a_1=1$, $b_{n-1}=y$) and, for $xy\neq0$, $\langle\nabla_1+\nabla_4,\nabla_2,\nabla_3+\nabla_4,\nabla_5\rangle$ if $n\neq5$ (by choosing $a_1=\sqrt[n-5]{xy}$, $b_{n-1}=\frac{a_1^{n-3}}{x}$), $\langle\nabla_1+\nabla_4,\nabla_2,\nabla_3+\alpha\nabla_4,\nabla_5\rangle$, $\alpha\neq0$, if $n=5$ (here $xy$ is invariant).
\item Let $P=\langle\nabla_1+x\nabla_3,\ \nabla_4\rangle$. We may assume $\theta_1=\alpha_3\nabla_3+\nabla_5$, and the choice $b_n=-\alpha_3$ gives $\alpha_3=0$. Thus $W=\langle\nabla_1+x\nabla_3,\nabla_2,\nabla_4,\nabla_5\rangle$ with $x'=a_1^{1-n}b_{n-1}^2x$, and we have the representatives $\langle\nabla_1,\nabla_2,\nabla_4,\nabla_5\rangle$ and $\langle\nabla_1+\nabla_3,\nabla_2,\nabla_4,\nabla_5\rangle$ (by choosing $a_1=1$, $b_{n-1}=\frac{1}{\sqrt{x}}$).
\item Let $P=\langle\nabla_3,\nabla_4\rangle$. We may assume $\theta_1=\alpha_1\nabla_1+\nabla_5$; the coefficient $\alpha_1$ is determined by $W$ and $\alpha_1'=a_1^{n-2}b_{n-1}^{-2}\alpha_1$. We have the representatives $\langle\nabla_2,\nabla_3,\nabla_4,\nabla_5\rangle$ and $\langle\nabla_1+\nabla_5,\nabla_2,\nabla_3,\nabla_4\rangle$ (by choosing $a_1=1$, $b_{n-1}=\sqrt{\alpha_1}$).
\end{itemize}

\textbf{Case 2.} Let $\nabla_2\notin V$. Then $V=\langle\nabla_1+p\nabla_2,\ \nabla_3+q\nabla_2,\ \nabla_4+r\nabla_2\rangle$, and by $(\ref{dagger})$, choosing $b_{n-2}=-p$, $a_{n-1}=-q$, $b_n=-r$, we obtain $V=\langle\nabla_1,\nabla_3,\nabla_4\rangle$; this form is preserved only by the automorphisms with $a_{n-1}=b_{n-2}=b_n=0$. Now we may assume $\theta_1=\alpha_2\nabla_2+\nabla_5$ with $\alpha_2'=b_{n-1}^{-1}\alpha_2$, and we have the representatives $\langle\nabla_1,\nabla_3,\nabla_4,\nabla_5\rangle$ and $\langle\nabla_1,\nabla_3,\nabla_4,\nabla_2+\nabla_5\rangle$ (by choosing $b_{n-1}=\alpha_2$).

As above, these orbits are pairwise distinct ($W\subseteq{\rm H}_0$, $\nabla_2\in W\cap{\rm H}_0$ and the orbit of $V$ are invariants, and the remaining coefficients are transformed as indicated). Thus we obtain the following $11$ orbits:
\begingroup\small
\begin{longtable}{ll}
$\omega^{1,4}_{1}$: & $\langle \nabla_{1},\ \nabla_{2},\ \nabla_{3},\ \nabla_{5}\rangle^{\ast}$\\
$\omega^{1,4}_{2}$: & $\langle \nabla_{1}+\nabla_{4},\ \nabla_{2},\ \nabla_{3},\ \nabla_{5}\rangle^{\ast}$\\
$\omega^{1,4}_{3}$: & $\langle \nabla_{1},\ \nabla_{2},\ \nabla_{3}+\nabla_{4},\ \nabla_{5}\rangle^{\ast}$\\
$\omega^{1,4}_{4}$: & $\langle \nabla_{1}+\nabla_{4},\ \nabla_{2},\ \nabla_{3}+\nabla_{4},\ \nabla_{5}\rangle^{\ast}$ \ ($n\neq 5$); $\langle \nabla_{1}+\nabla_{4},\ \nabla_{2},\ \nabla_{3}+\alpha \nabla_{4},\ \nabla_{5}\rangle^{\ast}$ \ ($n=5,\ \alpha\neq0$)\\
$\omega^{1,4}_{5}$: & $\langle \nabla_{1},\ \nabla_{2},\ \nabla_{4},\ \nabla_{5}\rangle^{\ast}$\\
$\omega^{1,4}_{6}$: & $\langle \nabla_{1}+\nabla_{3},\ \nabla_{2},\ \nabla_{4},\ \nabla_{5}\rangle^{\ast}$\\
$\omega^{1,4}_{7}$: & $\langle \nabla_{2},\ \nabla_{3},\ \nabla_{4},\ \nabla_{5}\rangle$\\
$\omega^{1,4}_{8}$: & $\langle \nabla_{1}+\nabla_{5},\ \nabla_{2},\ \nabla_{3},\ \nabla_{4}\rangle^{\ast}$\\
$\omega^{1,4}_{9}$: & $\langle \nabla_{1},\ \nabla_{3},\ \nabla_{4},\ \nabla_{5}\rangle^{\ast}$\\
$\omega^{1,4}_{10}$: & $\langle \nabla_{1},\ \nabla_{3},\ \nabla_{4},\ \nabla_{2}+\nabla_{5}\rangle^{\ast}$\\
$\omega^{1,4}_{11}$: & $\langle \nabla_{1},\ \nabla_{2},\ \nabla_{3},\ \nabla_{4}\rangle^{\ast}$\\
\end{longtable}
\endgroup

\subsubsection{$5$-dimensional central extensions of $\mu_n^1$}
Since $\dim {\rm H}^2(\mu_n^1,\mathbb{C})=5$, there is only one five-dimensional subspace, $W={\rm H}^2(\mu_n^1,\mathbb{C})$, and it lies in $T_5(\mu_n^1)$. We obtain the unique algebra $\omega^{1,5}_{1}$ (see the Appendix).

\subsection{Central extensions of $\mu_n^2$} Let us use the following notations:
\[\nabla_1=[\Delta_{n-2,1}+\Delta_{n-2,n-1}],\ \nabla_2=[\Delta_{n-1,1}], \ \nabla_3=[\Delta_{n-1,n-1}], \ \nabla_4=[-\Delta_{2,1}+\Delta_{n,1}], \ \nabla_5= [\Delta_{n,n-1}].\]
Take $\theta=\alpha_1\nabla_1+\alpha_2\nabla_2+\alpha_3\nabla_3+\alpha_4\nabla_4+\alpha_5\nabla_5\in {\rm H}^{2}(\mu_n^2,\mathbb{C}).$
Computing $\phi_2^n\theta(x,y)=\theta(\phi_2^n(x),\phi_2^n(y))$ and reducing modulo ${\rm B}^{2}(\mu_n^2,\mathbb{C})$
we have that the action of ${\rm Aut} (\mu_n^2)$ on the subspace
$\langle \alpha_1\nabla_1+\alpha_2\nabla_2+\alpha_3\nabla_3+\alpha_4\nabla_4+\alpha_5\nabla_5  \rangle$
is given by
$\langle \alpha_1^\prime\nabla_1+\alpha_2^\prime\nabla_2+\alpha_3^\prime\nabla_3+\alpha_4^\prime\nabla_4+\alpha_5^\prime\nabla_5\rangle,$
where
\[\begin{array}{l}
\alpha_1^\prime=a_1(a_1+a_{n-1})^{n-2}\alpha_1, \\
\alpha_2^\prime=(a_1+a_{n-1})(b_{n-2}\alpha_1 +a_1\alpha_2+a_{n-1}\alpha_3)+a_1b_n\alpha_4+a_{n-1}b_n\alpha_5, \\
\alpha_3^\prime=(a_1+a_{n-1})(b_{n-2}\alpha_1+(a_1+a_{n-1})\alpha_3+b_n\alpha_5), \\
\alpha_4^\prime=a_1^2(a_1\alpha_4+a_{n-1}\alpha_5),\\
\alpha_5^\prime=a_1^2(a_1+a_{n-1})\alpha_5.
\end{array}\]

Let $x=\lambda e_{n-2}+\mu e_n\in Ann(\mu_n^2)$. Then $\theta(x,e_1)=\lambda\alpha_1+\mu\alpha_4$, $\theta(x,e_{n-1})=\lambda\alpha_1+\mu\alpha_5$, and all the other values $\theta(x,e_j)$, $\theta(e_j,x)$ vanish. Hence $W\in T_s(\mu_n^2)$ if and only if the system $\lambda\alpha_1+\mu\alpha_4=0$, $\lambda\alpha_1+\mu\alpha_5=0$, $\theta\in W$, has only the trivial solution.

\subsubsection{$1$-dimensional central extensions of $\mu_n^2$}
The one-dimensional central extensions of $\mu_n^2$ were classified in \cite{Adashev2023}. In the notation of the present paper the result is as follows: there are $15$ orbits:
\begingroup\small
\begin{longtable}{ll}
$\omega^{2,1}_{1}$: & $\langle \nabla_{2}\rangle$\\
$\omega^{2,1}_{2}$: & $\langle \nabla_{3}\rangle$\\
$\omega^{2,1}_{3}$: & $\langle \nabla_{2}+\nabla_{3}\rangle$\\
$\omega^{2,1}_{4}$: & $\langle \nabla_{1}\rangle$\\
$\omega^{2,1}_{5}$: & $\langle \nabla_{1}+\nabla_{2}\rangle$\\
$\omega^{2,1}_{6}$: & $\langle \nabla_{4}\rangle$\\
$\omega^{2,1}_{7}$: & $\langle \nabla_{3}+\nabla_{4}\rangle$\\
$\omega^{2,1}_{8}$: & $\langle \nabla_{1}+\nabla_{4}\rangle^{\ast}$\\
$\omega^{2,1}_{9}$: & $\langle \nabla_{4}+\nabla_{5}\rangle$\\
$\omega^{2,1}_{10}$: & $\langle \nabla_{2}+\nabla_{4}+\nabla_{5}\rangle$\\
$\omega^{2,1}_{11}$: & $\langle \nabla_{5}\rangle$\\
$\omega^{2,1}_{12}$: & $\langle \nabla_{2}+\nabla_{5}\rangle$\\
$\omega^{2,1}_{13}$: & $\langle \nabla_{1}+\nabla_{4}+\nabla_{5}\rangle$\\
$\omega^{2,1}_{14}$: & $\langle \nabla_{1}+\nabla_{2}+\nabla_{4}+\nabla_{5}\rangle$\\
$\omega^{2,1}_{15}$: & $\langle \nabla_{1}+\nabla_{5}\rangle^{\ast}$\\
\end{longtable}
\endgroup

By the criterion above, exactly the orbits $\langle \nabla_1+\nabla_4\rangle$ and $\langle\nabla_1+\nabla_5\rangle$ lie in $T_1(\mu_n^2)$ (here $\alpha_1(\alpha_5-\alpha_4)\neq0$).

\subsubsection{An adapted basis of ${\rm H}^2(\mu_n^2,\mathbb{C})$} For the extensions of higher dimension it is convenient to use the basis
\[\Omega_1=\nabla_1,\quad \Omega_2=\nabla_2,\quad \Omega_3=\nabla_2+\nabla_3,\quad \Omega_4=\nabla_4,\quad \Omega_5=\nabla_4+\nabla_5.\]
If $\theta=\sum_{i=1}^5\beta_i\Omega_i$, i.e. $\alpha_1=\beta_1$, $\alpha_2=\beta_2+\beta_3$, $\alpha_3=\beta_3$, $\alpha_4=\beta_4+\beta_5$, $\alpha_5=\beta_5$, then, putting $b_{n-1}=a_1+a_{n-1}$ (an arbitrary non-zero number), the formulas above take the form
\[\beta_1'=a_1b_{n-1}^{n-2}\beta_1,\  \beta_2'=a_1(b_{n-1}\beta_2+b_n\beta_4),\  \beta_3'=b_{n-1}(b_{n-2}\beta_1+b_{n-1}\beta_3+b_n\beta_5),\  \beta_4'=a_1^3\beta_4,\ \beta_5'=a_1^2b_{n-1}\beta_5.\]
For $a_1=b_{n-1}=1$ these formulas take the form
\begin{equation}\label{dagger'} \beta_2\mapsto\beta_2+b_n\beta_4,\quad \beta_3\mapsto\beta_3+b_{n-2}\beta_1+b_n\beta_5\end{equation}
($\beta_1,\beta_4,\beta_5$ being unchanged), and for $b_{n-2}=b_n=0$ they are diagonal:
\begin{equation}\label{ddagger'} \beta_1'=a_1b_{n-1}^{n-2}\beta_1,\quad \beta_2'=a_1b_{n-1}\beta_2,\quad \beta_3'=b_{n-1}^2\beta_3,\quad \beta_4'=a_1^3\beta_4,\quad \beta_5'=a_1^2b_{n-1}\beta_5;\end{equation}
every automorphism is a composition of automorphisms of these two types.
In terms of the coordinates $\beta_i$ the criterion for $T_s(\mu_n^2)$ reads: the system $\mu\beta_4=0$, $\lambda\beta_1+\mu\beta_5=0$, $\theta\in W$, has only the trivial solution. In particular, $T_s$ depends only on the projection of $W$ to $\langle\Omega_1,\Omega_4,\Omega_5\rangle$ along $\langle\Omega_2,\Omega_3\rangle$. Note also that $\langle\Omega_2,\Omega_3\rangle$ is an invariant subspace, and that $\beta_1,\beta_4,\beta_5$ are transformed diagonally.

\subsubsection{$2$-dimensional central extensions of $\mu_n^2$}
Let $W\in G_2({\rm H}^{2}(\mu_n^2,\mathbb{C}))$ and let $\pi(W)$ be the projection of $W$ to $\langle\Omega_1,\Omega_4,\Omega_5\rangle$ along $\langle\Omega_2,\Omega_3\rangle$. The number $r=\dim\pi(W)$ is invariant. Below, when we use the parameters $b_{n-2}$, $b_n$ for removing coefficients, we write the corresponding formulas for $a_1=b_{n-1}=1$; the torus $(a_1,b_{n-1})$ is used afterwards for the normalization.

\textbf{Case $r=0$.} Then $W=\langle\Omega_2,\Omega_3\rangle$.

\textbf{Case $r=1$.} Then $W=\langle\kappa,\theta\rangle$ with $0\neq\kappa=\kappa_2\Omega_2+\kappa_3\Omega_3$ and $\theta=\sum\beta_i\Omega_i$, $(\beta_1,\beta_4,\beta_5)\neq0$. We have $\kappa_2'=a_1b_{n-1}\kappa_2$, $\kappa_3'=b_{n-1}^2\kappa_3$.

\emph{Case $r=1$, $\kappa_2=0$.} Then $\kappa=\Omega_3$ and we may assume $\beta_3=0$. If $\beta_4\neq0$, then choosing $b_n=-\frac{b_{n-1}\beta_2}{\beta_4}$ we obtain $\beta_2=0$, and the diagonal action on $(\beta_1,\beta_4,\beta_5)$ gives the representatives $\langle\Omega_3,\Omega_4\rangle$, $\langle\Omega_3,\Omega_1+\Omega_4\rangle$, $\langle\Omega_3,\Omega_4+\Omega_5\rangle$, $\langle\Omega_3,\Omega_1+\Omega_4+\Omega_5\rangle$ (in the last case $a_1,b_{n-1}$ are determined by $\frac{\beta_1'}{\beta_4'}=\frac{b_{n-1}^{n-2}\beta_1}{a_1^2\beta_4}=1$, $\frac{\beta_5'}{\beta_4'}=\frac{b_{n-1}\beta_5}{a_1\beta_4}=1$, which is possible since $n\neq4$). If $\beta_4=0$, then $\beta_2$ cannot be removed, and we have the representatives $\langle\Omega_3,\Omega_1\rangle$, $\langle\Omega_3,\Omega_1+\Omega_2\rangle$, $\langle\Omega_3,\Omega_5\rangle$, $\langle\Omega_3,\Omega_2+\Omega_5\rangle$, $\langle\Omega_3,\Omega_1+\Omega_5\rangle$, $\langle\Omega_3,\Omega_1+\Omega_2+\Omega_5\rangle$.

\emph{Case $r=1$, $\kappa_2\neq0$.} Then $\kappa=\Omega_2+k\Omega_3$ with $k'=\frac{b_{n-1}}{a_1}k$, and we may assume $\beta_2=0$. Applying the automorphism with parameter $b_n$ and subtracting the multiple $b_n\beta_4\kappa$ we obtain $\beta_3\mapsto\beta_3+b_n(\beta_5-k\beta_4)$; the parameter $b_{n-2}$ gives $\beta_3\mapsto\beta_3+b_{n-2}\beta_1$.
\begin{itemize}
\item Let $k=0$. If $\beta_1=\beta_5=0$, we have the representatives $\langle\Omega_2,\Omega_4\rangle$, $\langle\Omega_2,\Omega_3+\Omega_4\rangle$. Otherwise we obtain $\beta_3=0$ and the representatives $\langle\Omega_2,\Omega_1\rangle$, $\langle\Omega_2,\Omega_5\rangle$, $\langle\Omega_2,\Omega_1+\Omega_4\rangle$, $\langle\Omega_2,\Omega_1+\Omega_5\rangle$, $\langle\Omega_2,\Omega_4+\Omega_5\rangle$, $\langle\Omega_2,\Omega_1+\Omega_4+\Omega_5\rangle$.
\item Let $k\neq0$, $\beta_1\neq0$. Choosing $b_{n-2}$ we obtain $\beta_3=0$, and normalizing $k$ and the non-zero coefficients among $\beta_4,\beta_5$ (the characters of $k$, $\frac{\beta_4}{\beta_1}$, $\frac{\beta_5}{\beta_1}$ are $a_1^{-1}b_{n-1}$, $a_1^{2}b_{n-1}^{2-n}$, $a_1b_{n-1}^{3-n}$) we have the representatives $\langle\Omega_2+\Omega_3,\Omega_1\rangle$, $\langle\Omega_2+\Omega_3,\Omega_1+\Omega_4\rangle$, $\langle\Omega_2+\Omega_3,\Omega_1+\Omega_5\rangle$ and $\langle\Omega_2+\alpha\Omega_3,\Omega_1+\Omega_4+\Omega_5\rangle$, $\alpha\neq0$ (in the last case the stabilizer of the normalized coefficients satisfies $b_{n-1}=a_1$, so $\alpha$ is an invariant).
\item Let $k\neq0$, $\beta_1=0$, $\beta_5\neq k\beta_4$. Choosing $b_n$ we obtain $\beta_3=0$, and we have the representatives $\langle\Omega_2+\Omega_3,\Omega_4\rangle$, $\langle\Omega_2+\Omega_3,\Omega_5\rangle$ and, if $\beta_4\beta_5\neq0$, $\langle\Omega_2+\alpha\Omega_3,\Omega_4+\Omega_5\rangle$ with $\alpha=\frac{k\beta_4}{\beta_5}\neq0,1$ (this number is invariant since $k$ and $\frac{\beta_5}{\beta_4}$ have the same character $a_1^{-1}b_{n-1}$).
\item Let $k\neq0$, $\beta_1=0$, $\beta_5=k\beta_4$. Then $\beta_3$ cannot be removed; normalizing $k=1$ (so that $b_{n-1}=a_1$ afterwards) and $\frac{\beta_3'}{\beta_4'}=\frac{b_{n-1}^2\beta_3}{a_1^3\beta_4}$ we have the representatives $\langle\Omega_2+\Omega_3,\Omega_4+\Omega_5\rangle$ and $\langle\Omega_2+\Omega_3,\Omega_3+\Omega_4+\Omega_5\rangle$.
\end{itemize}
The representatives $\langle\Omega_2+\alpha\Omega_3,\Omega_4+\Omega_5\rangle$ (respectively $\langle\Omega_2+\alpha\Omega_3,\Omega_1+\Omega_4+\Omega_5\rangle$) obtained for $\alpha=0$, $\alpha=1$ and $\alpha\neq0,1$ (respectively for $\alpha=0$ and $\alpha\neq0$) are united into one family with $\alpha\in\mathbb{C}$.

\textbf{Case $r=2$.} The two-dimensional subspaces $P=\pi(W)$ of $\langle\Omega_1,\Omega_4,\Omega_5\rangle$ are, up to the diagonal action, $\langle\Omega_4,\Omega_5\rangle$, $\langle\Omega_1,\Omega_5\rangle$, $\langle\Omega_1,\Omega_4\rangle$, $\langle\Omega_1+p\Omega_4,\Omega_5\rangle$, $\langle\Omega_1+q\Omega_5,\Omega_4\rangle$, $\langle\Omega_1,\Omega_4+m\Omega_5\rangle$ ($p,q,m\neq0$) and $\langle\Omega_1+\Omega_5,\Omega_4+\Omega_5\rangle$ (the generic case; here the two ratios of the coefficients have the characters $a_1b_{n-1}^{3-n}$ and $a_1^{-1}b_{n-1}$, which are independent since $n\neq4$). Accordingly, $W=\langle\theta,\theta'\rangle$, where $\theta,\theta'$ project to the indicated basis of $P$, and we write $\theta=\dots+x_2\Omega_2+x_3\Omega_3$, $\theta'=\dots+y_2\Omega_2+y_3\Omega_3$.

\emph{$P=\langle\Omega_4,\Omega_5\rangle$.} The parameter $b_n$ gives $x_2\mapsto x_2+b_n$, $y_3\mapsto y_3+b_n$, so we may assume $x_2=0$. Then $x_3$, $y_2$, $y_3$ are transformed with the characters $a_1^{-3}b_{n-1}^2$, $a_1^{-1}$, $a_1^{-2}b_{n-1}$; any two of these characters are independent, so we have the representatives $\langle\Omega_4,\Omega_5\rangle$, $\langle\Omega_3+\Omega_4,\Omega_5\rangle$, $\langle\Omega_4,\Omega_2+\Omega_5\rangle$, $\langle\Omega_4,\Omega_3+\Omega_5\rangle$, $\langle\Omega_3+\Omega_4,\Omega_2+\Omega_5\rangle$, $\langle\Omega_4,\Omega_2+\Omega_3+\Omega_5\rangle$, and $\langle\Omega_3+\Omega_4,\alpha\Omega_2+\Omega_3+\Omega_5\rangle$, $\alpha\in\mathbb{C}$ (for $x_3y_3\neq0$ we normalize $x_3=y_3=1$, which determines $a_1,b_{n-1}$ uniquely).

\emph{$P=\langle\Omega_1,\Omega_5\rangle$.} Choosing $b_{n-2}$ and $b_n$ we obtain $x_3=y_3=0$; $x_2$ and $y_2$ have the independent characters $b_{n-1}^{3-n}$, $a_1^{-1}$, and we have the representatives $\langle\Omega_1,\Omega_5\rangle$, $\langle\Omega_1+\Omega_2,\Omega_5\rangle$, $\langle\Omega_1,\Omega_2+\Omega_5\rangle$, $\langle\Omega_1+\Omega_2,\Omega_2+\Omega_5\rangle$.

\emph{$P=\langle\Omega_1,\Omega_4\rangle$.} Choosing $b_{n-2}$ and $b_n$ we obtain $x_3=y_2=0$; $x_2$ and $y_3$ have the independent characters $b_{n-1}^{3-n}$, $a_1^{-3}b_{n-1}^2$, and we have the representatives $\langle\Omega_1,\Omega_4\rangle$, $\langle\Omega_1+\Omega_2,\Omega_4\rangle$, $\langle\Omega_1,\Omega_3+\Omega_4\rangle$, $\langle\Omega_1+\Omega_2,\Omega_3+\Omega_4\rangle$.

\emph{$P=\langle\Omega_1+p\Omega_4,\Omega_5\rangle$.} The parameter $b_{n-2}$ gives $x_3\mapsto x_3+b_{n-2}$ and $b_n$ gives $x_2\mapsto x_2+pb_n$, $y_3\mapsto y_3+b_n$. So we may assume $x_2=x_3=0$. The numbers $p$, $y_2$, $y_3$ have the characters $a_1^2b_{n-1}^{2-n}$, $a_1^{-1}$, $a_1^{-2}b_{n-1}$, and we have the representatives $\langle\Omega_1+\Omega_4,\Omega_5\rangle$, $\langle\Omega_1+\Omega_4,\Omega_2+\Omega_5\rangle$, $\langle\Omega_1+\Omega_4,\Omega_3+\Omega_5\rangle$ (here we use $n\neq3$) and $\langle\Omega_1+\alpha\Omega_4,\Omega_2+\Omega_3+\Omega_5\rangle$, $\alpha\neq0$ (normalizing $y_2=y_3=1$ determines $a_1,b_{n-1}$ uniquely).

\emph{$P=\langle\Omega_1+q\Omega_5,\Omega_4\rangle$.} The parameter $b_n$ gives $x_3\mapsto x_3+qb_n$, $y_2\mapsto y_2+b_n$, and $b_{n-2}$ gives $x_3\mapsto x_3+b_{n-2}$. So we may assume $y_2=x_3=0$. The numbers $q$, $x_2$, $y_3$ have the characters $a_1b_{n-1}^{3-n}$, $b_{n-1}^{3-n}$, $a_1^{-3}b_{n-1}^2$, and we have the representatives $\langle\Omega_1+\Omega_5,\Omega_4\rangle$, $\langle\Omega_1+\Omega_2+\Omega_5,\Omega_4\rangle$, $\langle\Omega_1+\Omega_5,\Omega_3+\Omega_4\rangle$ (here we use $3n\neq11$) and $\langle\Omega_1+\Omega_2+\Omega_5,\alpha\Omega_3+\Omega_4\rangle$, $\alpha\neq0$. In the last case the stabilizer of the normalized coefficients is $a_1=1$, $b_{n-1}^{n-3}=1$, so $\alpha$ and $\zeta^2\alpha$ ($\zeta^{n-3}=1$) give the same orbit.

\emph{$P=\langle\Omega_1,\Omega_4+m\Omega_5\rangle$.} Choosing $b_{n-2}$ and $b_n$ we obtain $x_3=y_2=0$. Normalizing $m=1$ we get $b_{n-1}=a_1$, and then $x_2$, $y_3$ have the characters $a_1^{3-n}$, $a_1^{-1}$. We have the representatives $\langle\Omega_1,\Omega_4+\Omega_5\rangle$, $\langle\Omega_1+\Omega_2,\Omega_4+\Omega_5\rangle$, $\langle\Omega_1+\alpha\Omega_2,\Omega_3+\Omega_4+\Omega_5\rangle$, $\alpha\in\mathbb{C}$.

\emph{$P=\langle\Omega_1+\Omega_5,\Omega_4+\Omega_5\rangle$.} The parameter $b_n$ gives $x_3\mapsto x_3+b_n$, $y_2\mapsto y_2+b_n$, $y_3\mapsto y_3+b_n$, and $b_{n-2}$ gives $x_3\mapsto x_3+b_{n-2}$. So we may assume $y_2=x_3=0$. The stabilizer of $P$ in the torus is $a_1=b_{n-1}=\zeta$, $\zeta^{n-4}=1$, and it multiplies $x_2$ and $y_3$ by $\zeta^{-1}$. We have the representatives $\langle\Omega_1+\alpha\Omega_2+\Omega_5,\beta\Omega_3+\Omega_4+\Omega_5\rangle$, $\alpha,\beta\in\mathbb{C}$, where $(\alpha,\beta)$ and $(\zeta\alpha,\zeta\beta)$, $\zeta^{n-4}=1$, give the same orbit.

The orbits obtained are pairwise distinct: $r$, the orbit of $\pi(W)$ and, for $r=1$, the orbit of the line $W\cap\langle\Omega_2,\Omega_3\rangle$ are invariants; inside each case the vanishing of the remaining coefficients of the normal forms is invariant under the stabilizer of the normal form, and the identifications of the parameters are given by the finite stabilizers indicated above. Thus we obtain the following $52$ items ($59$ series of orbits if the families are counted according to the different values of the parameter considered in the proof):
\begingroup\small
\begin{longtable}{ll}
$\omega^{2,2}_{1}$: & $\langle \Omega_{2},\ \Omega_{3}\rangle$\\
$\omega^{2,2}_{2}$: & $\langle \Omega_{3},\ \Omega_{4}\rangle$\\
$\omega^{2,2}_{3}$: & $\langle \Omega_{3},\ \Omega_{1}+\Omega_{4}\rangle^{\ast}$\\
$\omega^{2,2}_{4}$: & $\langle \Omega_{3},\ \Omega_{4}+\Omega_{5}\rangle$\\
$\omega^{2,2}_{5}$: & $\langle \Omega_{3},\ \Omega_{1}+\Omega_{4}+\Omega_{5}\rangle^{\ast}$\\
$\omega^{2,2}_{6}$: & $\langle \Omega_{3},\ \Omega_{1}\rangle$\\
$\omega^{2,2}_{7}$: & $\langle \Omega_{3},\ \Omega_{1}+\Omega_{2}\rangle$\\
$\omega^{2,2}_{8}$: & $\langle \Omega_{3},\ \Omega_{5}\rangle$\\
$\omega^{2,2}_{9}$: & $\langle \Omega_{3},\ \Omega_{2}+\Omega_{5}\rangle$\\
$\omega^{2,2}_{10}$: & $\langle \Omega_{3},\ \Omega_{1}+\Omega_{5}\rangle$\\
$\omega^{2,2}_{11}$: & $\langle \Omega_{3},\ \Omega_{1}+\Omega_{2}+\Omega_{5}\rangle$\\
$\omega^{2,2}_{12}$: & $\langle \Omega_{2},\ \Omega_{4}\rangle$\\
$\omega^{2,2}_{13}$: & $\langle \Omega_{2},\ \Omega_{3}+\Omega_{4}\rangle$\\
$\omega^{2,2}_{14}$: & $\langle \Omega_{2},\ \Omega_{1}\rangle$\\
$\omega^{2,2}_{15}$: & $\langle \Omega_{2},\ \Omega_{5}\rangle$\\
$\omega^{2,2}_{16}$: & $\langle \Omega_{2},\ \Omega_{1}+\Omega_{4}\rangle^{\ast}$\\
$\omega^{2,2}_{17}$: & $\langle \Omega_{2},\ \Omega_{1}+\Omega_{5}\rangle$\\
$\omega^{2,2}_{18}$: & $\langle \Omega_{2}+\Omega_{3},\ \Omega_{4}\rangle$\\
$\omega^{2,2}_{19}$: & $\langle \Omega_{2}+\Omega_{3},\ \Omega_{5}\rangle$\\
$\omega^{2,2}_{20}$: & $\langle \Omega_{2}+\alpha \Omega_{3},\ \Omega_{4}+\Omega_{5}\rangle$ \ ($\alpha\in\mathbb{C}$)\\
$\omega^{2,2}_{21}$: & $\langle \Omega_{2}+\Omega_{3},\ \Omega_{3}+\Omega_{4}+\Omega_{5}\rangle$\\
$\omega^{2,2}_{22}$: & $\langle \Omega_{2}+\Omega_{3},\ \Omega_{1}\rangle$\\
$\omega^{2,2}_{23}$: & $\langle \Omega_{2}+\Omega_{3},\ \Omega_{1}+\Omega_{4}\rangle^{\ast}$\\
$\omega^{2,2}_{24}$: & $\langle \Omega_{2}+\Omega_{3},\ \Omega_{1}+\Omega_{5}\rangle$\\
$\omega^{2,2}_{25}$: & $\langle \Omega_{2}+\alpha \Omega_{3},\ \Omega_{1}+\Omega_{4}+\Omega_{5}\rangle^{\ast}$ \ ($\alpha\in\mathbb{C}$)\\
$\omega^{2,2}_{26}$: & $\langle \Omega_{4},\ \Omega_{5}\rangle$\\
$\omega^{2,2}_{27}$: & $\langle \Omega_{3}+\Omega_{4},\ \Omega_{5}\rangle$\\
$\omega^{2,2}_{28}$: & $\langle \Omega_{4},\ \Omega_{2}+\Omega_{5}\rangle$\\
$\omega^{2,2}_{29}$: & $\langle \Omega_{4},\ \Omega_{3}+\Omega_{5}\rangle$\\
$\omega^{2,2}_{30}$: & $\langle \Omega_{3}+\Omega_{4},\ \Omega_{2}+\Omega_{5}\rangle$\\
$\omega^{2,2}_{31}$: & $\langle \Omega_{4},\ \Omega_{2}+\Omega_{3}+\Omega_{5}\rangle$\\
$\omega^{2,2}_{32}$: & $\langle \Omega_{3}+\Omega_{4},\ \alpha \Omega_{2}+\Omega_{3}+\Omega_{5}\rangle$ \ ($\alpha\in\mathbb{C}$)\\
$\omega^{2,2}_{33}$: & $\langle \Omega_{1},\ \Omega_{5}\rangle^{\ast}$\\
$\omega^{2,2}_{34}$: & $\langle \Omega_{1}+\Omega_{2},\ \Omega_{5}\rangle^{\ast}$\\
$\omega^{2,2}_{35}$: & $\langle \Omega_{1},\ \Omega_{2}+\Omega_{5}\rangle^{\ast}$\\
$\omega^{2,2}_{36}$: & $\langle \Omega_{1}+\Omega_{2},\ \Omega_{2}+\Omega_{5}\rangle^{\ast}$\\
$\omega^{2,2}_{37}$: & $\langle \Omega_{1},\ \Omega_{4}\rangle^{\ast}$\\
$\omega^{2,2}_{38}$: & $\langle \Omega_{1}+\Omega_{2},\ \Omega_{4}\rangle^{\ast}$\\
$\omega^{2,2}_{39}$: & $\langle \Omega_{1},\ \Omega_{3}+\Omega_{4}\rangle^{\ast}$\\
$\omega^{2,2}_{40}$: & $\langle \Omega_{1}+\Omega_{2},\ \Omega_{3}+\Omega_{4}\rangle^{\ast}$\\
$\omega^{2,2}_{41}$: & $\langle \Omega_{1}+\Omega_{4},\ \Omega_{5}\rangle^{\ast}$\\
$\omega^{2,2}_{42}$: & $\langle \Omega_{1}+\Omega_{4},\ \Omega_{2}+\Omega_{5}\rangle^{\ast}$\\
$\omega^{2,2}_{43}$: & $\langle \Omega_{1}+\Omega_{4},\ \Omega_{3}+\Omega_{5}\rangle^{\ast}$\\
$\omega^{2,2}_{44}$: & $\langle \Omega_{1}+\alpha \Omega_{4},\ \Omega_{2}+\Omega_{3}+\Omega_{5}\rangle^{\ast}$ \ ($\alpha\neq0$)\\
$\omega^{2,2}_{45}$: & $\langle \Omega_{1}+\Omega_{5},\ \Omega_{4}\rangle^{\ast}$\\
$\omega^{2,2}_{46}$: & $\langle \Omega_{1}+\Omega_{2}+\Omega_{5},\ \Omega_{4}\rangle^{\ast}$\\
$\omega^{2,2}_{47}$: & $\langle \Omega_{1}+\Omega_{5},\ \Omega_{3}+\Omega_{4}\rangle^{\ast}$\\
$\omega^{2,2}_{48}$: & $\langle \Omega_{1}+\Omega_{2}+\Omega_{5},\ \alpha \Omega_{3}+\Omega_{4}\rangle^{\ast}$ \ ($\alpha\neq0;\ \alpha\sim\zeta^2\alpha,\ \zeta^{n-3}=1$)\\
$\omega^{2,2}_{49}$: & $\langle \Omega_{1},\ \Omega_{4}+\Omega_{5}\rangle^{\ast}$\\
$\omega^{2,2}_{50}$: & $\langle \Omega_{1}+\Omega_{2},\ \Omega_{4}+\Omega_{5}\rangle^{\ast}$\\
$\omega^{2,2}_{51}$: & $\langle \Omega_{1}+\alpha \Omega_{2},\ \Omega_{3}+\Omega_{4}+\Omega_{5}\rangle^{\ast}$ \ ($\alpha\in\mathbb{C}$)\\
$\omega^{2,2}_{52}$: & $\langle \Omega_{1}+\alpha \Omega_{2}+\Omega_{5},\ \beta \Omega_{3}+\Omega_{4}+\Omega_{5}\rangle^{\ast}$ \ ($\alpha\in\mathbb{C},\ \beta\in\mathbb{C};\ (\alpha,\beta)\sim(\zeta\alpha,\zeta\beta),\ \zeta^{n-4}=1$)\\
\end{longtable}
\endgroup

\subsubsection{$3$-dimensional central extensions of $\mu_n^2$}
Let $W\in G_3({\rm H}^{2}(\mu_n^2,\mathbb{C}))$, let $\pi(W)$ be the projection of $W$ to $\langle\Omega_1,\Omega_4,\Omega_5\rangle$ along $\langle\Omega_2,\Omega_3\rangle$ and let $K=W\cap\langle\Omega_2,\Omega_3\rangle$. Since $\langle\Omega_2,\Omega_3\rangle$ is invariant, the number $r=\dim\pi(W)=3-\dim K$, the orbit of $K$ and the orbit of $\pi(W)$ (with respect to the diagonal action of $\beta_1,\beta_4,\beta_5$) are invariants. Note that $r\geq1$. As before, the parameters $b_{n-2},b_n$ are used by means of $(\ref{dagger'})$, and then $a_1,b_{n-1}$ are used by means of $(\ref{ddagger'})$.

\textbf{Case $r=1$.} Then $K=\langle\Omega_2,\Omega_3\rangle$ and $W=\langle\Omega_2,\Omega_3,\theta\rangle$, where $0\neq\theta=\beta_1\Omega_1+\beta_4\Omega_4+\beta_5\Omega_5$ is transformed diagonally. We have the representatives $\langle\Omega_2,\Omega_3,\Omega_1\rangle$, $\langle\Omega_2,\Omega_3,\Omega_4\rangle$, $\langle\Omega_2,\Omega_3,\Omega_5\rangle$, $\langle\Omega_2,\Omega_3,\Omega_1+\Omega_4\rangle$, $\langle\Omega_2,\Omega_3,\Omega_1+\Omega_5\rangle$, $\langle\Omega_2,\Omega_3,\Omega_4+\Omega_5\rangle$ and $\langle\Omega_2,\Omega_3,\Omega_1+\Omega_4+\Omega_5\rangle$ (in the last case $\frac{\beta_4}{\beta_1}$ and $\frac{\beta_5}{\beta_1}$ have the characters $a_1^2b_{n-1}^{2-n}$ and $a_1b_{n-1}^{3-n}$, which are independent since $n\neq4$).

\textbf{Case $r=3$.} Then $K=0$ and $W=\langle\Omega_1+x_2\Omega_2+x_3\Omega_3,\ \Omega_4+y_2\Omega_2+y_3\Omega_3,\ \Omega_5+z_2\Omega_2+z_3\Omega_3\rangle$. By $(\ref{dagger'})$ we have $x_3\mapsto x_3+b_{n-2}$, $y_2\mapsto y_2+b_n$, $z_3\mapsto z_3+b_n$, the other coefficients being unchanged. Choosing $b_{n-2}=-x_3$, $b_n=-y_2$ we obtain $x_3=y_2=0$; this form is preserved only by $b_{n-2}=b_n=0$. By $(\ref{ddagger'})$ the remaining coefficients are transformed by
\[x_2'=b_{n-1}^{3-n}x_2,\qquad y_3'=a_1^{-3}b_{n-1}^2y_3,\qquad z_2'=a_1^{-1}z_2,\qquad z_3'=a_1^{-2}b_{n-1}z_3,\]
and any two of these characters are independent (here we use $n\neq3$). Hence, if at most two of the coefficients $x_2,y_3,z_2,z_3$ are non-zero, they can be normalized to $1$, and if $y_3\neq0$, $z_3\neq0$ (respectively $z_2\neq0$, $z_3\neq0$; $z_2\neq0$, $y_3\neq0$, $z_3=0$) we normalize $y_3=z_3=1$, which gives $a_1=b_{n-1}=1$ (respectively $z_2=z_3=1$, which gives $a_1=b_{n-1}=1$; $z_2=y_3=1$, which gives $a_1=1$, $b_{n-1}=\pm1$). We obtain the representatives $\langle\Omega_1,\Omega_4,\Omega_5\rangle$, $\langle\Omega_1+\Omega_2,\Omega_4,\Omega_5\rangle$, $\langle\Omega_1,\Omega_4,\Omega_2+\Omega_5\rangle$, $\langle\Omega_1,\Omega_3+\Omega_4,\Omega_5\rangle$, $\langle\Omega_1,\Omega_4,\Omega_3+\Omega_5\rangle$, $\langle\Omega_1+\Omega_2,\Omega_4,\Omega_2+\Omega_5\rangle$, $\langle\Omega_1+\Omega_2,\Omega_3+\Omega_4,\Omega_5\rangle$, $\langle\Omega_1+\Omega_2,\Omega_4,\Omega_3+\Omega_5\rangle$, and the families $\langle\Omega_1+\alpha\Omega_2,\Omega_3+\Omega_4,\Omega_2+\Omega_5\rangle$ ($\alpha\in\mathbb{C}$; $\alpha$ and $(-1)^{n-3}\alpha$ give the same orbit), $\langle\Omega_1+\alpha\Omega_2,\Omega_4,\Omega_2+\Omega_3+\Omega_5\rangle$ ($\alpha\in\mathbb{C}$) and $\langle\Omega_1+\alpha\Omega_2,\Omega_3+\Omega_4,\beta\Omega_2+\Omega_3+\Omega_5\rangle$ ($\alpha,\beta\in\mathbb{C}$). In these families the values $\alpha=0$ and $\beta=0$ are allowed; they correspond to the cases $x_2=0$ and $z_2=0$, respectively.

\textbf{Case $r=2$.} Then $K$ is a line in $\langle\Omega_2,\Omega_3\rangle$, on which the action is diagonal ($\beta_2'=a_1b_{n-1}\beta_2$, $\beta_3'=b_{n-1}^2\beta_3$). Hence $K=\langle\Omega_3\rangle$, or $K=\langle\Omega_2\rangle$, or $K=\langle\Omega_2+k\Omega_3\rangle$ with $k\neq0$; in the last case $k'=a_1^{-1}b_{n-1}k$, and choosing $b_{n-1}=\frac{a_1}{k}$ we obtain $K=\langle\Omega_2+\Omega_3\rangle$, which is preserved exactly by the automorphisms with $b_{n-1}=a_1$. The two-dimensional subspace $P=\pi(W)$ of $\langle\Omega_1,\Omega_4,\Omega_5\rangle$ is transformed diagonally; by the reduced row echelon form, it is one of
\[P=\langle\Omega_1+u\Omega_5,\ \Omega_4+v\Omega_5\rangle,\qquad P=\langle\Omega_1+p\Omega_4,\ \Omega_5\rangle,\qquad P=\langle\Omega_4,\Omega_5\rangle,\]
where $u'=a_1b_{n-1}^{3-n}u$, $v'=a_1^{-1}b_{n-1}v$, $p'=a_1^2b_{n-1}^{2-n}p$, and the vanishing of $u,v,p$ is invariant. Then $W=K\oplus\langle\theta,\theta'\rangle$, where $\theta,\theta'$ are the elements of $W$ projecting to the indicated basis of $P$; they are determined modulo $K$, so each of them has one more coefficient, which we write in the form $x\Omega_3$, $y\Omega_3$ (if $K=\langle\Omega_2\rangle$ or $K=\langle\Omega_2+\Omega_3\rangle$) or $x\Omega_2$, $y\Omega_2$ (if $K=\langle\Omega_3\rangle$).

\emph{$K=\langle\Omega_2\rangle$.} Modulo $\Omega_2$ the formula $(\ref{dagger'})$ is $\beta_3\mapsto\beta_3+b_{n-2}\beta_1+b_n\beta_5$, and by $(\ref{ddagger'})$ the coefficient of $\Omega_3$ in $\theta$ (respectively $\theta'$) has the character $b_{n-1}^2$ divided by the character of the leading coefficient of $\theta$ (respectively $\theta'$).
\begin{itemize}
\item $P=\langle\Omega_1+u\Omega_5,\Omega_4+v\Omega_5\rangle$: $x\mapsto x+b_{n-2}+b_nu$, $y\mapsto y+b_nv$. If $v\neq0$, we choose $b_n=-\frac{y}{v}$ and $b_{n-2}=-x-b_nu$, and obtain $x=y=0$; we have the representatives $\langle\Omega_2,\Omega_1+\Omega_5,\Omega_4+\Omega_5\rangle$ ($uv\neq0$; by choosing $b_{n-1}=\sqrt[4-n]{\frac{1}{uv}}$, $a_1=b_{n-1}v$) and $\langle\Omega_2,\Omega_1,\Omega_4+\Omega_5\rangle$ ($u=0$; by choosing $a_1=v$, $b_{n-1}=1$). If $v=0$, we choose $b_{n-2}=-x-b_nu$ and obtain $x=0$; then $y$ does not change and $y'=a_1^{-3}b_{n-1}^2y$. We have the representatives $\langle\Omega_2,\Omega_1+\Omega_5,\Omega_4\rangle$, $\langle\Omega_2,\Omega_1+\Omega_5,\Omega_3+\Omega_4\rangle$ (by choosing $b_{n-1}=\sqrt[3n-11]{yu^3}$, $a_1=\frac{b_{n-1}^{n-3}}{u}$; note that $3n\neq11$), $\langle\Omega_2,\Omega_1,\Omega_4\rangle$ and $\langle\Omega_2,\Omega_1,\Omega_3+\Omega_4\rangle$ (by choosing $a_1=1$, $b_{n-1}=\frac{1}{\sqrt{y}}$).
\item $P=\langle\Omega_1+p\Omega_4,\Omega_5\rangle$: $x\mapsto x+b_{n-2}$, $y\mapsto y+b_n$. Hence $x=y=0$, and we have the representatives $\langle\Omega_2,\Omega_1+\Omega_4,\Omega_5\rangle$ (by choosing $a_1=\frac{1}{\sqrt{p}}$, $b_{n-1}=1$) and $\langle\Omega_2,\Omega_1,\Omega_5\rangle$.
\item $P=\langle\Omega_4,\Omega_5\rangle$: $x$ does not change, $y\mapsto y+b_n$. Hence $y=0$, $x'=a_1^{-3}b_{n-1}^2x$, and we have the representatives $\langle\Omega_2,\Omega_4,\Omega_5\rangle$ and $\langle\Omega_2,\Omega_3+\Omega_4,\Omega_5\rangle$.
\end{itemize}

\emph{$K=\langle\Omega_3\rangle$.} Modulo $\Omega_3$ the formula $(\ref{dagger'})$ is $\beta_2\mapsto\beta_2+b_n\beta_4$, and the parameter $b_{n-2}$ does not act.
\begin{itemize}
\item $P=\langle\Omega_1+u\Omega_5,\Omega_4+v\Omega_5\rangle$: $x$ does not change, $y\mapsto y+b_n$. Hence $y=0$, and $x'=b_{n-1}^{3-n}x$. If $uv\neq0$, then normalizing $u=v=1$ we get $b_{n-1}=a_1$, $a_1^{n-4}=1$, and $x$ is multiplied by $a_1^{3-n}=a_1^{-1}$; we have the representatives $\langle\Omega_3,\Omega_1+\Omega_5,\Omega_4+\Omega_5\rangle$ and $\langle\Omega_3,\Omega_1+\alpha\Omega_2+\Omega_5,\Omega_4+\Omega_5\rangle$, $\alpha\neq0$, where $\alpha$ and $\zeta\alpha$ ($\zeta^{n-4}=1$) give the same orbit. If $u\neq0$, $v=0$, then the characters of $x$ and $u$ are independent, and we have the representatives $\langle\Omega_3,\Omega_1+\Omega_5,\Omega_4\rangle$, $\langle\Omega_3,\Omega_1+\Omega_2+\Omega_5,\Omega_4\rangle$; similarly, for $u=0$, $v\neq0$ we have $\langle\Omega_3,\Omega_1,\Omega_4+\Omega_5\rangle$, $\langle\Omega_3,\Omega_1+\Omega_2,\Omega_4+\Omega_5\rangle$, and for $u=v=0$ we have $\langle\Omega_3,\Omega_1,\Omega_4\rangle$, $\langle\Omega_3,\Omega_1+\Omega_2,\Omega_4\rangle$.
\item $P=\langle\Omega_1+p\Omega_4,\Omega_5\rangle$: $x\mapsto x+b_np$, and $y$ does not change; $y'=a_1^{-1}y$. If $p\neq0$, we choose $b_n=-\frac{x}{p}$ and obtain $x=0$; we have the representatives $\langle\Omega_3,\Omega_1+\Omega_4,\Omega_5\rangle$ and $\langle\Omega_3,\Omega_1+\Omega_4,\Omega_2+\Omega_5\rangle$ (by choosing $a_1=y$, $b_{n-1}=\sqrt[n-2]{a_1^2p}$). If $p=0$, then $x'=b_{n-1}^{3-n}x$, $y'=a_1^{-1}y$, and we have the representatives $\langle\Omega_3,\Omega_1,\Omega_5\rangle$, $\langle\Omega_3,\Omega_1+\Omega_2,\Omega_5\rangle$, $\langle\Omega_3,\Omega_1,\Omega_2+\Omega_5\rangle$, $\langle\Omega_3,\Omega_1+\Omega_2,\Omega_2+\Omega_5\rangle$.
\item $P=\langle\Omega_4,\Omega_5\rangle$: $x\mapsto x+b_n$, so $x=0$; $y'=a_1^{-1}y$, and we have the representatives $\langle\Omega_3,\Omega_4,\Omega_5\rangle$ and $\langle\Omega_3,\Omega_4,\Omega_2+\Omega_5\rangle$.
\end{itemize}

\emph{$K=\langle\Omega_2+\Omega_3\rangle$.} Here $b_{n-1}=a_1$, and modulo $K$ an element $y_2\Omega_2+y_3\Omega_3$ of $\langle\Omega_2,\Omega_3\rangle$ is equal to $(y_3-y_2)\Omega_3$. By $(\ref{dagger'})$ and $(\ref{ddagger'})$, for $\theta=\beta_1\Omega_1+\beta_4\Omega_4+\beta_5\Omega_5+x\Omega_3$ we have $x\mapsto x+b_{n-2}\beta_1+b_n(\beta_5-\beta_4)$, and the characters of $\beta_1,\beta_4,\beta_5,x$ are $a_1^{n-1},a_1^3,a_1^3,a_1^2$. In particular $u'=a_1^{4-n}u$, $p'=a_1^{4-n}p$ and $v'=v$, i.e. $v$ is an invariant.
\begin{itemize}
\item $P=\langle\Omega_1+u\Omega_5,\Omega_4+v\Omega_5\rangle$: $x\mapsto x+b_{n-2}+b_nu$, $y\mapsto y+b_n(v-1)$. If $v\neq1$, then choosing $b_n=-\frac{y}{v-1}$, $b_{n-2}=-x-b_nu$ we obtain $x=y=0$, and we have the representatives $\langle\Omega_2+\Omega_3,\Omega_1+\Omega_5,\Omega_4+\alpha\Omega_5\rangle$, $\alpha\neq1$ (for $u\neq0$, by choosing $a_1=b_{n-1}=\sqrt[n-4]{u}$) and $\langle\Omega_2+\Omega_3,\Omega_1,\Omega_4+\alpha\Omega_5\rangle$, $\alpha\neq1$ (for $u=0$). If $v=1$, then choosing $b_{n-2}=-x-b_nu$ we obtain $x=0$, and $y$ does not change; $y'=a_1^{-1}y$. For $u\neq0$ we normalize $u=1$, so that $a_1^{n-4}=1$, and we have the representatives $\langle\Omega_2+\Omega_3,\Omega_1+\Omega_5,\alpha\Omega_3+\Omega_4+\Omega_5\rangle$, $\alpha\in\mathbb{C}$, where $\alpha$ and $\zeta\alpha$ ($\zeta^{n-4}=1$) give the same orbit; for $u=0$ we have the representatives $\langle\Omega_2+\Omega_3,\Omega_1,\Omega_4+\Omega_5\rangle$ and $\langle\Omega_2+\Omega_3,\Omega_1,\Omega_3+\Omega_4+\Omega_5\rangle$ (by choosing $a_1=b_{n-1}=y$).
\item $P=\langle\Omega_1+p\Omega_4,\Omega_5\rangle$: $x\mapsto x+b_{n-2}-b_np$, $y\mapsto y+b_n$. Hence $x=y=0$, and we have the representatives $\langle\Omega_2+\Omega_3,\Omega_1+\Omega_4,\Omega_5\rangle$ (by choosing $a_1=b_{n-1}=\sqrt[n-4]{p}$) and $\langle\Omega_2+\Omega_3,\Omega_1,\Omega_5\rangle$.
\item $P=\langle\Omega_4,\Omega_5\rangle$: $x\mapsto x-b_n$, $y\mapsto y+b_n$. Choosing $b_n=x$ we obtain $x=0$; then $y'=a_1^{-1}y$, and we have the representatives $\langle\Omega_2+\Omega_3,\Omega_4,\Omega_5\rangle$ and $\langle\Omega_2+\Omega_3,\Omega_4,\Omega_3+\Omega_5\rangle$.
\end{itemize}

The orbits obtained are pairwise distinct: $r$, the orbit of $K$, the orbit of $\pi(W)$ (for $K=\langle\Omega_2+\Omega_3\rangle$ including the value of $v$) are invariants; after $K$ and $\pi(W)$ are normalized, only the automorphisms preserving the normal forms can be used, their action on the remaining coefficients is written above, the vanishing of these coefficients is invariant, and the identifications of the parameters are given by the finite stabilizers indicated above. Thus we obtain the following $53$ items:
\begingroup\small
\begin{longtable}{ll}
$\omega^{2,3}_{1}$: & $\langle \Omega_{2},\ \Omega_{3},\ \Omega_{1}\rangle$\\
$\omega^{2,3}_{2}$: & $\langle \Omega_{2},\ \Omega_{3},\ \Omega_{4}\rangle$\\
$\omega^{2,3}_{3}$: & $\langle \Omega_{2},\ \Omega_{3},\ \Omega_{5}\rangle$\\
$\omega^{2,3}_{4}$: & $\langle \Omega_{2},\ \Omega_{3},\ \Omega_{1}+\Omega_{4}\rangle^{\ast}$\\
$\omega^{2,3}_{5}$: & $\langle \Omega_{2},\ \Omega_{3},\ \Omega_{1}+\Omega_{5}\rangle$\\
$\omega^{2,3}_{6}$: & $\langle \Omega_{2},\ \Omega_{3},\ \Omega_{4}+\Omega_{5}\rangle$\\
$\omega^{2,3}_{7}$: & $\langle \Omega_{2},\ \Omega_{3},\ \Omega_{1}+\Omega_{4}+\Omega_{5}\rangle^{\ast}$\\
$\omega^{2,3}_{8}$: & $\langle \Omega_{2},\ \Omega_{1}+\Omega_{5},\ \Omega_{4}+\Omega_{5}\rangle^{\ast}$\\
$\omega^{2,3}_{9}$: & $\langle \Omega_{2},\ \Omega_{1},\ \Omega_{4}+\Omega_{5}\rangle^{\ast}$\\
$\omega^{2,3}_{10}$: & $\langle \Omega_{2},\ \Omega_{1}+\Omega_{5},\ \Omega_{4}\rangle^{\ast}$\\
$\omega^{2,3}_{11}$: & $\langle \Omega_{2},\ \Omega_{1}+\Omega_{5},\ \Omega_{3}+\Omega_{4}\rangle^{\ast}$\\
$\omega^{2,3}_{12}$: & $\langle \Omega_{2},\ \Omega_{1},\ \Omega_{4}\rangle^{\ast}$\\
$\omega^{2,3}_{13}$: & $\langle \Omega_{2},\ \Omega_{1},\ \Omega_{3}+\Omega_{4}\rangle^{\ast}$\\
$\omega^{2,3}_{14}$: & $\langle \Omega_{2},\ \Omega_{1}+\Omega_{4},\ \Omega_{5}\rangle^{\ast}$\\
$\omega^{2,3}_{15}$: & $\langle \Omega_{2},\ \Omega_{1},\ \Omega_{5}\rangle^{\ast}$\\
$\omega^{2,3}_{16}$: & $\langle \Omega_{2},\ \Omega_{4},\ \Omega_{5}\rangle$\\
$\omega^{2,3}_{17}$: & $\langle \Omega_{2},\ \Omega_{3}+\Omega_{4},\ \Omega_{5}\rangle$\\
$\omega^{2,3}_{18}$: & $\langle \Omega_{3},\ \Omega_{1}+\Omega_{5},\ \Omega_{4}+\Omega_{5}\rangle^{\ast}$\\
$\omega^{2,3}_{19}$: & $\langle \Omega_{3},\ \Omega_{1}+\alpha \Omega_{2}+\Omega_{5},\ \Omega_{4}+\Omega_{5}\rangle^{\ast}$ \ ($\alpha\neq0;\ \alpha\sim\zeta\alpha,\ \zeta^{n-4}=1$)\\
$\omega^{2,3}_{20}$: & $\langle \Omega_{3},\ \Omega_{1}+\Omega_{5},\ \Omega_{4}\rangle^{\ast}$\\
$\omega^{2,3}_{21}$: & $\langle \Omega_{3},\ \Omega_{1}+\Omega_{2}+\Omega_{5},\ \Omega_{4}\rangle^{\ast}$\\
$\omega^{2,3}_{22}$: & $\langle \Omega_{3},\ \Omega_{1},\ \Omega_{4}+\Omega_{5}\rangle^{\ast}$\\
$\omega^{2,3}_{23}$: & $\langle \Omega_{3},\ \Omega_{1}+\Omega_{2},\ \Omega_{4}+\Omega_{5}\rangle^{\ast}$\\
$\omega^{2,3}_{24}$: & $\langle \Omega_{3},\ \Omega_{1},\ \Omega_{4}\rangle^{\ast}$\\
$\omega^{2,3}_{25}$: & $\langle \Omega_{3},\ \Omega_{1}+\Omega_{2},\ \Omega_{4}\rangle^{\ast}$\\
$\omega^{2,3}_{26}$: & $\langle \Omega_{3},\ \Omega_{1}+\Omega_{4},\ \Omega_{5}\rangle^{\ast}$\\
$\omega^{2,3}_{27}$: & $\langle \Omega_{3},\ \Omega_{1}+\Omega_{4},\ \Omega_{2}+\Omega_{5}\rangle^{\ast}$\\
$\omega^{2,3}_{28}$: & $\langle \Omega_{3},\ \Omega_{1},\ \Omega_{5}\rangle^{\ast}$\\
$\omega^{2,3}_{29}$: & $\langle \Omega_{3},\ \Omega_{1}+\Omega_{2},\ \Omega_{5}\rangle^{\ast}$\\
$\omega^{2,3}_{30}$: & $\langle \Omega_{3},\ \Omega_{1},\ \Omega_{2}+\Omega_{5}\rangle^{\ast}$\\
$\omega^{2,3}_{31}$: & $\langle \Omega_{3},\ \Omega_{1}+\Omega_{2},\ \Omega_{2}+\Omega_{5}\rangle^{\ast}$\\
$\omega^{2,3}_{32}$: & $\langle \Omega_{3},\ \Omega_{4},\ \Omega_{5}\rangle$\\
$\omega^{2,3}_{33}$: & $\langle \Omega_{3},\ \Omega_{4},\ \Omega_{2}+\Omega_{5}\rangle$\\
$\omega^{2,3}_{34}$: & $\langle \Omega_{2}+\Omega_{3},\ \Omega_{1}+\Omega_{5},\ \Omega_{4}+\alpha \Omega_{5}\rangle^{\ast}$ \ ($\alpha\neq1$)\\
$\omega^{2,3}_{35}$: & $\langle \Omega_{2}+\Omega_{3},\ \Omega_{1},\ \Omega_{4}+\alpha \Omega_{5}\rangle^{\ast}$ \ ($\alpha\neq1$)\\
$\omega^{2,3}_{36}$: & $\langle \Omega_{2}+\Omega_{3},\ \Omega_{1}+\Omega_{5},\ \alpha \Omega_{3}+\Omega_{4}+\Omega_{5}\rangle^{\ast}$ \ ($\alpha\in\mathbb{C};\ \alpha\sim\zeta\alpha,\ \zeta^{n-4}=1$)\\
$\omega^{2,3}_{37}$: & $\langle \Omega_{2}+\Omega_{3},\ \Omega_{1},\ \Omega_{4}+\Omega_{5}\rangle^{\ast}$\\
$\omega^{2,3}_{38}$: & $\langle \Omega_{2}+\Omega_{3},\ \Omega_{1},\ \Omega_{3}+\Omega_{4}+\Omega_{5}\rangle^{\ast}$\\
$\omega^{2,3}_{39}$: & $\langle \Omega_{2}+\Omega_{3},\ \Omega_{1}+\Omega_{4},\ \Omega_{5}\rangle^{\ast}$\\
$\omega^{2,3}_{40}$: & $\langle \Omega_{2}+\Omega_{3},\ \Omega_{1},\ \Omega_{5}\rangle^{\ast}$\\
$\omega^{2,3}_{41}$: & $\langle \Omega_{2}+\Omega_{3},\ \Omega_{4},\ \Omega_{5}\rangle$\\
$\omega^{2,3}_{42}$: & $\langle \Omega_{2}+\Omega_{3},\ \Omega_{4},\ \Omega_{3}+\Omega_{5}\rangle$\\
$\omega^{2,3}_{43}$: & $\langle \Omega_{1},\ \Omega_{4},\ \Omega_{5}\rangle^{\ast}$\\
$\omega^{2,3}_{44}$: & $\langle \Omega_{1}+\Omega_{2},\ \Omega_{4},\ \Omega_{5}\rangle^{\ast}$\\
$\omega^{2,3}_{45}$: & $\langle \Omega_{1},\ \Omega_{4},\ \Omega_{2}+\Omega_{5}\rangle^{\ast}$\\
$\omega^{2,3}_{46}$: & $\langle \Omega_{1},\ \Omega_{3}+\Omega_{4},\ \Omega_{5}\rangle^{\ast}$\\
$\omega^{2,3}_{47}$: & $\langle \Omega_{1},\ \Omega_{4},\ \Omega_{3}+\Omega_{5}\rangle^{\ast}$\\
$\omega^{2,3}_{48}$: & $\langle \Omega_{1}+\Omega_{2},\ \Omega_{4},\ \Omega_{2}+\Omega_{5}\rangle^{\ast}$\\
$\omega^{2,3}_{49}$: & $\langle \Omega_{1}+\Omega_{2},\ \Omega_{3}+\Omega_{4},\ \Omega_{5}\rangle^{\ast}$\\
$\omega^{2,3}_{50}$: & $\langle \Omega_{1}+\Omega_{2},\ \Omega_{4},\ \Omega_{3}+\Omega_{5}\rangle^{\ast}$\\
$\omega^{2,3}_{51}$: & $\langle \Omega_{1}+\alpha \Omega_{2},\ \Omega_{3}+\Omega_{4},\ \Omega_{2}+\Omega_{5}\rangle^{\ast}$ \ ($\alpha\in\mathbb{C};\ \alpha\sim-\alpha \text{ if } n \text{ is even}$)\\
$\omega^{2,3}_{52}$: & $\langle \Omega_{1}+\alpha \Omega_{2},\ \Omega_{4},\ \Omega_{2}+\Omega_{3}+\Omega_{5}\rangle^{\ast}$ \ ($\alpha\in\mathbb{C}$)\\
$\omega^{2,3}_{53}$: & $\langle \Omega_{1}+\alpha \Omega_{2},\ \Omega_{3}+\Omega_{4},\ \beta \Omega_{2}+\Omega_{3}+\Omega_{5}\rangle^{\ast}$ \ ($\alpha,\beta\in\mathbb{C}$)\\
\end{longtable}
\endgroup

\subsubsection{$4$-dimensional central extensions of $\mu_n^2$}
Let $W\in G_4({\rm H}^{2}(\mu_n^2,\mathbb{C}))$, and let $\pi(W)$, $K$ and $r=\dim\pi(W)=4-\dim K$ be as above; now $r\in\{2,3\}$.

\textbf{Case $r=2$.} Then $W=\langle\Omega_2,\Omega_3\rangle\oplus P$, where $P$ is a two-dimensional subspace of $\langle\Omega_1,\Omega_4,\Omega_5\rangle$ transformed diagonally, and we have the representatives $\langle\Omega_2,\Omega_3,\Omega_4,\Omega_5\rangle$, $\langle\Omega_2,\Omega_3,\Omega_1,\Omega_5\rangle$, $\langle\Omega_2,\Omega_3,\Omega_1,\Omega_4\rangle$, $\langle\Omega_2,\Omega_3,\Omega_1+\Omega_4,\Omega_5\rangle$, $\langle\Omega_2,\Omega_3,\Omega_1+\Omega_5,\Omega_4\rangle$, $\langle\Omega_2,\Omega_3,\Omega_1,\Omega_4+\Omega_5\rangle$ and $\langle\Omega_2,\Omega_3,\Omega_1+\Omega_5,\Omega_4+\Omega_5\rangle$ (in the last case the characters $a_1b_{n-1}^{3-n}$ and $a_1^{-1}b_{n-1}$ of the two coefficients are independent since $n\neq4$).

\textbf{Case $r=3$.} Then $K$ is a line, and as above we may assume $K=\langle\Omega_2\rangle$, $K=\langle\Omega_3\rangle$ or $K=\langle\Omega_2+\Omega_3\rangle$ (in the last case $b_{n-1}=a_1$ afterwards). Further, $W=K\oplus\langle\Omega_1+x\Omega_\ast,\ \Omega_4+y\Omega_\ast,\ \Omega_5+z\Omega_\ast\rangle$, where $\Omega_\ast=\Omega_3$ for $K=\langle\Omega_2\rangle$ or $K=\langle\Omega_2+\Omega_3\rangle$, and $\Omega_\ast=\Omega_2$ for $K=\langle\Omega_3\rangle$.
\begin{itemize}
\item $K=\langle\Omega_2\rangle$: by $(\ref{dagger'})$ $x\mapsto x+b_{n-2}$, $z\mapsto z+b_n$, and $y$ does not change. Choosing $b_{n-2}=-x$, $b_n=-z$ we obtain $x=z=0$, and since $y'=a_1^{-3}b_{n-1}^2y$ we have the representatives $\langle\Omega_2,\Omega_1,\Omega_4,\Omega_5\rangle$ and $\langle\Omega_2,\Omega_1,\Omega_3+\Omega_4,\Omega_5\rangle$.
\item $K=\langle\Omega_3\rangle$: $y\mapsto y+b_n$, and $x$, $z$ do not change. Choosing $b_n=-y$ we obtain $y=0$; since $x'=b_{n-1}^{3-n}x$ and $z'=a_1^{-1}z$ have independent characters, we have the representatives $\langle\Omega_3,\Omega_1,\Omega_4,\Omega_5\rangle$, $\langle\Omega_3,\Omega_1+\Omega_2,\Omega_4,\Omega_5\rangle$, $\langle\Omega_3,\Omega_1,\Omega_4,\Omega_2+\Omega_5\rangle$, $\langle\Omega_3,\Omega_1+\Omega_2,\Omega_4,\Omega_2+\Omega_5\rangle$.
\item $K=\langle\Omega_2+\Omega_3\rangle$: $x\mapsto x+b_{n-2}$, $y\mapsto y-b_n$, $z\mapsto z+b_n$. Choosing $b_{n-2}=-x$, $b_n=y$ we obtain $x=y=0$; since $z'=a_1^{-1}z$, we have the representatives $\langle\Omega_2+\Omega_3,\Omega_1,\Omega_4,\Omega_5\rangle$ and $\langle\Omega_2+\Omega_3,\Omega_1,\Omega_4,\Omega_3+\Omega_5\rangle$.
\end{itemize}
These orbits are pairwise distinct ($r$, the orbit of $K$ and the vanishing of the remaining coefficient(s) are invariant). Thus we obtain the following $15$ orbits:
\begingroup\small
\begin{longtable}{ll}
$\omega^{2,4}_{1}$: & $\langle \Omega_{2},\ \Omega_{3},\ \Omega_{4},\ \Omega_{5}\rangle$\\
$\omega^{2,4}_{2}$: & $\langle \Omega_{2},\ \Omega_{3},\ \Omega_{1},\ \Omega_{5}\rangle^{\ast}$\\
$\omega^{2,4}_{3}$: & $\langle \Omega_{2},\ \Omega_{3},\ \Omega_{1},\ \Omega_{4}\rangle^{\ast}$\\
$\omega^{2,4}_{4}$: & $\langle \Omega_{2},\ \Omega_{3},\ \Omega_{1}+\Omega_{4},\ \Omega_{5}\rangle^{\ast}$\\
$\omega^{2,4}_{5}$: & $\langle \Omega_{2},\ \Omega_{3},\ \Omega_{1}+\Omega_{5},\ \Omega_{4}\rangle^{\ast}$\\
$\omega^{2,4}_{6}$: & $\langle \Omega_{2},\ \Omega_{3},\ \Omega_{1},\ \Omega_{4}+\Omega_{5}\rangle^{\ast}$\\
$\omega^{2,4}_{7}$: & $\langle \Omega_{2},\ \Omega_{3},\ \Omega_{1}+\Omega_{5},\ \Omega_{4}+\Omega_{5}\rangle^{\ast}$\\
$\omega^{2,4}_{8}$: & $\langle \Omega_{2},\ \Omega_{1},\ \Omega_{4},\ \Omega_{5}\rangle^{\ast}$\\
$\omega^{2,4}_{9}$: & $\langle \Omega_{2},\ \Omega_{1},\ \Omega_{3}+\Omega_{4},\ \Omega_{5}\rangle^{\ast}$\\
$\omega^{2,4}_{10}$: & $\langle \Omega_{3},\ \Omega_{1},\ \Omega_{4},\ \Omega_{5}\rangle^{\ast}$\\
$\omega^{2,4}_{11}$: & $\langle \Omega_{3},\ \Omega_{1}+\Omega_{2},\ \Omega_{4},\ \Omega_{5}\rangle^{\ast}$\\
$\omega^{2,4}_{12}$: & $\langle \Omega_{3},\ \Omega_{1},\ \Omega_{4},\ \Omega_{2}+\Omega_{5}\rangle^{\ast}$\\
$\omega^{2,4}_{13}$: & $\langle \Omega_{3},\ \Omega_{1}+\Omega_{2},\ \Omega_{4},\ \Omega_{2}+\Omega_{5}\rangle^{\ast}$\\
$\omega^{2,4}_{14}$: & $\langle \Omega_{2}+\Omega_{3},\ \Omega_{1},\ \Omega_{4},\ \Omega_{5}\rangle^{\ast}$\\
$\omega^{2,4}_{15}$: & $\langle \Omega_{2}+\Omega_{3},\ \Omega_{1},\ \Omega_{4},\ \Omega_{3}+\Omega_{5}\rangle^{\ast}$\\
\end{longtable}
\endgroup

\subsubsection{$5$-dimensional central extensions of $\mu_n^2$}
There is only one five-dimensional subspace, $W={\rm H}^2(\mu_n^2,\mathbb{C})$; it lies in $T_5(\mu_n^2)$, and we obtain the unique algebra $\omega^{2,5}_{1}$.

\subsection{Central extensions of $\mu_5^3$} Let us use the following notations:
\[\nabla_1=[-i\Delta_{2,1}+\Delta_{2,4}], \ \nabla_2=[\Delta_{4,1}], \  \nabla_3=[\Delta_{4,4}], \ \nabla_4=[i\Delta_{2,1}+ \Delta_{5,1}].\]
Take $\theta=\alpha_1\nabla_1+\alpha_2\nabla_2+\alpha_3\nabla_3+\alpha_4\nabla_4\in {\rm H}^{2}(\mu_5^3,\mathbb{C}).$
Computing $\phi_3^5\theta(x,y)=\theta(\phi_3^5(x),\phi_3^5(y))$ and reducing modulo ${\rm B}^{2}(\mu_5^3,\mathbb{C})$
we have that the action of ${\rm Aut} (\mu_5^3)$ on the subspace
$\langle \alpha_1\nabla_1+\alpha_2\nabla_2+\alpha_3\nabla_3+\alpha_4\nabla_4  \rangle$
is given by
$\langle \alpha_1^\prime\nabla_1+\alpha_2^\prime\nabla_2+\alpha_3^\prime\nabla_3+\alpha_4^\prime\nabla_4\rangle,$
where
\[\begin{array}{ll}
\alpha_1^\prime=a_1(a_1 + 2i a_4)^2\alpha_1, &
\alpha_2^\prime=(a_1+2ia_4)(a_1\alpha_2+a_4\alpha_3), \\
\alpha_3^\prime=(a_1+2ia_4)^2\alpha_3, &
\alpha_4^\prime=ia_1a_4(a_1+ia_4)\alpha_1+a_1^3\alpha_4.
\end{array}\]
In particular, $\alpha_1'=0$ if and only if $\alpha_1=0$; hence the coefficient $\alpha_1$ cannot be removed (a substitution of the type ``$a_4=ia_4$'' does not lead to $\alpha_1'=0$, since $a_1\neq0$ and $a_1+2ia_4\neq0$ for an automorphism).

Put $c=a_1+2ia_4$; since $a_4=\frac{c-a_1}{2i}$, the numbers $a_1$ and $c$ are arbitrary non-zero numbers. Consider the basis
\[\Omega_1=4\nabla_1+\nabla_4=[\Delta_{2,4}+\Delta_{5,1}+3i\Delta_{5,4}],\ \Omega_2=\nabla_2=[\Delta_{4,1}],\]  \[\Omega_3=2\nabla_3-i\nabla_2=[2\Delta_{4,4}-i\Delta_{4,1}],\ \Omega_4=\nabla_4=[i\Delta_{2,1}+\Delta_{5,1}]\]
(we used $\Delta_{2,1}+i\Delta_{2,4}+\Delta_{5,4}\in{\rm B}^2(\mu_5^3,\mathbb{C})$). If $\theta=\sum_{i=1}^4\beta_i\Omega_i$, then a direct computation with the formulas above gives
\[\beta_1'=a_1c^2\beta_1,\qquad \beta_2'=a_1c\beta_2,\qquad \beta_3'=c^2\beta_3,\qquad \beta_4'=a_1^3\beta_4.\]
Thus the action of ${\rm Aut}(\mu_5^3)$ on ${\rm H}^2(\mu_5^3,\mathbb{C})$ is diagonal in the basis $\Omega_1,\dots,\Omega_4$, with two independent parameters $a_1,c$, and the orbits are described by the observation at the end of Section 2: the orbit of $W$ is determined by the pivot positions and the set of non-zero entries of the reduced row echelon basis of $W$, together with the values of the entries which cannot be normalized. By Remark \ref{rem-34}, none of the orbits lies in $T_s(\mu_5^3)$.

\subsubsection{$1$-dimensional central extensions of $\mu_5^3$}
The one-dimensional central extensions of $\mu_5^3$ were classified in \cite{Adashev2023}. In the notation of the present paper (with respect to the basis $\Omega_1,\dots,\Omega_4$) the result is as follows: there are $15$ orbits:
\begingroup\small
\begin{longtable}{ll}
$\omega^{3,1}_{1}$: & $\langle \Omega_{1}\rangle$\\
$\omega^{3,1}_{2}$: & $\langle \Omega_{1}+\Omega_{2}\rangle$\\
$\omega^{3,1}_{3}$: & $\langle \Omega_{1}+\Omega_{3}\rangle$\\
$\omega^{3,1}_{4}$: & $\langle \Omega_{1}+\Omega_{4}\rangle$\\
$\omega^{3,1}_{5}$: & $\langle \Omega_{1}+\Omega_{2}+\Omega_{3}\rangle$\\
$\omega^{3,1}_{6}$: & $\langle \Omega_{1}+\Omega_{2}+\Omega_{4}\rangle$\\
$\omega^{3,1}_{7}$: & $\langle \Omega_{1}+\Omega_{3}+\Omega_{4}\rangle$\\
$\omega^{3,1}_{8}$: & $\langle \Omega_{1}+\Omega_{2}+\Omega_{3}+\alpha \Omega_{4}\rangle$ ($\alpha\neq0$)\\
$\omega^{3,1}_{9}$: & $\langle \Omega_{2}\rangle$\\
$\omega^{3,1}_{10}$: & $\langle \Omega_{2}+\Omega_{3}\rangle$\\
$\omega^{3,1}_{11}$: & $\langle \Omega_{2}+\Omega_{4}\rangle$\\
$\omega^{3,1}_{12}$: & $\langle \Omega_{2}+\Omega_{3}+\Omega_{4}\rangle$\\
$\omega^{3,1}_{13}$: & $\langle \Omega_{3}\rangle$\\
$\omega^{3,1}_{14}$: & $\langle \Omega_{3}+\Omega_{4}\rangle$\\
$\omega^{3,1}_{15}$: & $\langle \Omega_{4}\rangle$\\
\end{longtable}
\endgroup

\subsubsection{$2$-dimensional central extensions of $\mu_5^3$}
Let $W=\langle\theta_1,\theta_2\rangle$ be written in the reduced row echelon form, $\theta_1=\Omega_{p_1}+\sum\alpha_j\Omega_j$, $\theta_2=\Omega_{p_2}+\sum\beta_j\Omega_j$. Normalizing the non-zero coefficients by the choice of $a_1$ and $c$ indicated in the table we obtain the following $35$ orbits:
\begingroup\small
\begin{longtable}{lll}
$\omega^{3,2}_{1}$ &  & $\langle \Omega_{1},\ \Omega_{2}\rangle$\\
$\omega^{3,2}_{2}$ & $a_1=\alpha_{3},\ c=1$ & $\langle \Omega_{1}+\Omega_{3},\ \Omega_{2}\rangle$\\
$\omega^{3,2}_{3}$ & $a_1=\sqrt{\frac{1}{\alpha_{4}}},\ c=1$ & $\langle \Omega_{1}+\Omega_{4},\ \Omega_{2}\rangle$\\
$\omega^{3,2}_{4}$ & $a_1=\beta_{3},\ c=1$ & $\langle \Omega_{1},\ \Omega_{2}+\Omega_{3}\rangle$\\
$\omega^{3,2}_{5}$ & $a_1=\sqrt{\frac{1}{\beta_{4}}},\ c=1$ & $\langle \Omega_{1},\ \Omega_{2}+\Omega_{4}\rangle$\\
$\omega^{3,2}_{6}$ & $a_1=\alpha_{3},\ c=\sqrt{\alpha_{3}^{2}\alpha_{4}}$ & $\langle \Omega_{1}+\Omega_{3}+\Omega_{4},\ \Omega_{2}\rangle$\\
$\omega^{3,2}_{7}$ & $a_1=\alpha_{3},\ c=\frac{\alpha_{3}}{\beta_{3}}$ & $\langle \Omega_{1}+\Omega_{3},\ \Omega_{2}+\Omega_{3}\rangle$\\
$\omega^{3,2}_{8}$ & $a_1=\alpha_{3},\ c=\alpha_{3}^{2}\beta_{4}$ & $\langle \Omega_{1}+\Omega_{3},\ \Omega_{2}+\Omega_{4}\rangle$\\
$\omega^{3,2}_{9}$ & $a_1=\beta_{3},\ c=1$ & $\langle \Omega_{1}+\alpha \Omega_{4},\ \Omega_{2}+\Omega_{3}\rangle$ ($\alpha\neq0$)\\
$\omega^{3,2}_{10}$ & $a_1=\sqrt{\frac{\alpha_{4}}{\beta_{4}^{2}}},\ c=\frac{\alpha_{4}}{\beta_{4}}$ & $\langle \Omega_{1}+\Omega_{4},\ \Omega_{2}+\Omega_{4}\rangle$\\
$\omega^{3,2}_{11}$ & $a_1=\frac{1}{\beta_{3}\beta_{4}},\ c=\frac{1}{\beta_{3}^{2}\beta_{4}}$ & $\langle \Omega_{1},\ \Omega_{2}+\Omega_{3}+\Omega_{4}\rangle$\\
$\omega^{3,2}_{12}$ & $a_1=\alpha_{3},\ c=\frac{\alpha_{3}}{\beta_{3}}$ & $\langle \Omega_{1}+\Omega_{3}+\alpha \Omega_{4},\ \Omega_{2}+\Omega_{3}\rangle$ ($\alpha\neq0$)\\
$\omega^{3,2}_{13}$ & $a_1=\alpha_{3},\ c=\alpha_{3}^{2}\beta_{4}$ & $\langle \Omega_{1}+\Omega_{3}+\alpha \Omega_{4},\ \Omega_{2}+\Omega_{4}\rangle$ ($\alpha\neq0$)\\
$\omega^{3,2}_{14}$ & $a_1=\alpha_{3},\ c=\frac{\alpha_{3}}{\beta_{3}}$ & $\langle \Omega_{1}+\Omega_{3},\ \Omega_{2}+\Omega_{3}+\alpha \Omega_{4}\rangle$ ($\alpha\neq0$)\\
$\omega^{3,2}_{15}$ & $a_1=\frac{1}{\beta_{3}\beta_{4}},\ c=\frac{1}{\beta_{3}^{2}\beta_{4}}$ & $\langle \Omega_{1}+\alpha \Omega_{4},\ \Omega_{2}+\Omega_{3}+\Omega_{4}\rangle$ ($\alpha\neq0$)\\
$\omega^{3,2}_{16}$ & $a_1=\alpha_{3},\ c=\frac{\alpha_{3}}{\beta_{3}}$ & $\langle \Omega_{1}+\Omega_{3}+\alpha \Omega_{4},\ \Omega_{2}+\Omega_{3}+\beta \Omega_{4}\rangle$ ($\alpha\neq0,\ \beta\neq0$)\\
$\omega^{3,2}_{17}$ &  & $\langle \Omega_{1},\ \Omega_{3}\rangle$\\
$\omega^{3,2}_{18}$ & $a_1=1,\ c=\alpha_{2}$ & $\langle \Omega_{1}+\Omega_{2},\ \Omega_{3}\rangle$\\
$\omega^{3,2}_{19}$ & $a_1=\sqrt{\frac{1}{\alpha_{4}}},\ c=1$ & $\langle \Omega_{1}+\Omega_{4},\ \Omega_{3}\rangle$\\
$\omega^{3,2}_{20}$ & $a_1=\sqrt[3]{\frac{1}{\beta_{4}}},\ c=1$ & $\langle \Omega_{1},\ \Omega_{3}+\Omega_{4}\rangle$\\
$\omega^{3,2}_{21}$ & $a_1=\sqrt{\frac{\alpha_{2}^{2}}{\alpha_{4}}},\ c=\alpha_{2}$ & $\langle \Omega_{1}+\Omega_{2}+\Omega_{4},\ \Omega_{3}\rangle$\\
$\omega^{3,2}_{22}$ & $a_1=\sqrt[3]{\frac{\alpha_{2}^{2}}{\beta_{4}}},\ c=\alpha_{2}$ & $\langle \Omega_{1}+\Omega_{2},\ \Omega_{3}+\Omega_{4}\rangle$\\
$\omega^{3,2}_{23}$ & $a_1=\frac{\alpha_{4}}{\beta_{4}},\ c=\sqrt{\frac{\alpha_{4}^{3}}{\beta_{4}^{2}}}$ & $\langle \Omega_{1}+\Omega_{4},\ \Omega_{3}+\Omega_{4}\rangle$\\
$\omega^{3,2}_{24}$ & $a_1=\sqrt{\frac{\alpha_{2}^{2}}{\alpha_{4}}},\ c=\alpha_{2}$ & $\langle \Omega_{1}+\Omega_{2}+\Omega_{4},\ \Omega_{3}+\alpha \Omega_{4}\rangle$ ($\alpha\neq0;\ \alpha\sim -\alpha$)\\
$\omega^{3,2}_{25}$ &  & $\langle \Omega_{1},\ \Omega_{4}\rangle$\\
$\omega^{3,2}_{26}$ & $a_1=1,\ c=\alpha_{2}$ & $\langle \Omega_{1}+\Omega_{2},\ \Omega_{4}\rangle$\\
$\omega^{3,2}_{27}$ & $a_1=\alpha_{3},\ c=1$ & $\langle \Omega_{1}+\Omega_{3},\ \Omega_{4}\rangle$\\
$\omega^{3,2}_{28}$ & $a_1=\alpha_{3},\ c=\alpha_{2}$ & $\langle \Omega_{1}+\Omega_{2}+\Omega_{3},\ \Omega_{4}\rangle$\\
$\omega^{3,2}_{29}$ &  & $\langle \Omega_{2},\ \Omega_{3}\rangle$\\
$\omega^{3,2}_{30}$ & $a_1=\sqrt{\frac{1}{\alpha_{4}}},\ c=1$ & $\langle \Omega_{2}+\Omega_{4},\ \Omega_{3}\rangle$\\
$\omega^{3,2}_{31}$ & $a_1=\sqrt[3]{\frac{1}{\beta_{4}}},\ c=1$ & $\langle \Omega_{2},\ \Omega_{3}+\Omega_{4}\rangle$\\
$\omega^{3,2}_{32}$ & $a_1=\frac{\beta_{4}}{\alpha_{4}^{2}},\ c=\frac{\beta_{4}^{2}}{\alpha_{4}^{3}}$ & $\langle \Omega_{2}+\Omega_{4},\ \Omega_{3}+\Omega_{4}\rangle$\\
$\omega^{3,2}_{33}$ &  & $\langle \Omega_{2},\ \Omega_{4}\rangle$\\
$\omega^{3,2}_{34}$ & $a_1=\alpha_{3},\ c=1$ & $\langle \Omega_{2}+\Omega_{3},\ \Omega_{4}\rangle$\\
$\omega^{3,2}_{35}$ &  & $\langle \Omega_{3},\ \Omega_{4}\rangle$\\
\end{longtable}
\endgroup

\subsubsection{$3$-dimensional central extensions of $\mu_5^3$}
Similarly, with $W=\langle\theta_1,\theta_2,\theta_3\rangle$ in the reduced row echelon form ($\theta_3=\Omega_{p_3}+\sum\gamma_j\Omega_j$) we obtain the following $15$ orbits:
\begingroup\small
\begin{longtable}{lll}
$\omega^{3,3}_{1}$ &  & $\langle \Omega_{1},\ \Omega_{2},\ \Omega_{3}\rangle$\\
$\omega^{3,3}_{2}$ & $a_1=\sqrt{\frac{1}{\alpha_{4}}},\ c=1$ & $\langle \Omega_{1}+\Omega_{4},\ \Omega_{2},\ \Omega_{3}\rangle$\\
$\omega^{3,3}_{3}$ & $a_1=\sqrt{\frac{1}{\beta_{4}}},\ c=1$ & $\langle \Omega_{1},\ \Omega_{2}+\Omega_{4},\ \Omega_{3}\rangle$\\
$\omega^{3,3}_{4}$ & $a_1=\sqrt[3]{\frac{1}{\gamma_{4}}},\ c=1$ & $\langle \Omega_{1},\ \Omega_{2},\ \Omega_{3}+\Omega_{4}\rangle$\\
$\omega^{3,3}_{5}$ & $a_1=\sqrt{\frac{\alpha_{4}}{\beta_{4}^{2}}},\ c=\frac{\alpha_{4}}{\beta_{4}}$ & $\langle \Omega_{1}+\Omega_{4},\ \Omega_{2}+\Omega_{4},\ \Omega_{3}\rangle$\\
$\omega^{3,3}_{6}$ & $a_1=\frac{\alpha_{4}}{\gamma_{4}},\ c=\sqrt{\frac{\alpha_{4}^{3}}{\gamma_{4}^{2}}}$ & $\langle \Omega_{1}+\Omega_{4},\ \Omega_{2},\ \Omega_{3}+\Omega_{4}\rangle$\\
$\omega^{3,3}_{7}$ & $a_1=\frac{\gamma_{4}}{\beta_{4}^{2}},\ c=\frac{\gamma_{4}^{2}}{\beta_{4}^{3}}$ & $\langle \Omega_{1},\ \Omega_{2}+\Omega_{4},\ \Omega_{3}+\Omega_{4}\rangle$\\
$\omega^{3,3}_{8}$ & $a_1=\frac{\gamma_{4}}{\beta_{4}^{2}},\ c=\frac{\gamma_{4}^{2}}{\beta_{4}^{3}}$ & $\langle \Omega_{1}+\alpha \Omega_{4},\ \Omega_{2}+\Omega_{4},\ \Omega_{3}+\Omega_{4}\rangle$ ($\alpha\neq0$)\\
$\omega^{3,3}_{9}$ &  & $\langle \Omega_{1},\ \Omega_{2},\ \Omega_{4}\rangle$\\
$\omega^{3,3}_{10}$ & $a_1=\alpha_{3},\ c=1$ & $\langle \Omega_{1}+\Omega_{3},\ \Omega_{2},\ \Omega_{4}\rangle$\\
$\omega^{3,3}_{11}$ & $a_1=\beta_{3},\ c=1$ & $\langle \Omega_{1},\ \Omega_{2}+\Omega_{3},\ \Omega_{4}\rangle$\\
$\omega^{3,3}_{12}$ & $a_1=\alpha_{3},\ c=\frac{\alpha_{3}}{\beta_{3}}$ & $\langle \Omega_{1}+\Omega_{3},\ \Omega_{2}+\Omega_{3},\ \Omega_{4}\rangle$\\
$\omega^{3,3}_{13}$ &  & $\langle \Omega_{1},\ \Omega_{3},\ \Omega_{4}\rangle$\\
$\omega^{3,3}_{14}$ & $a_1=1,\ c=\alpha_{2}$ & $\langle \Omega_{1}+\Omega_{2},\ \Omega_{3},\ \Omega_{4}\rangle$\\
$\omega^{3,3}_{15}$ &  & $\langle \Omega_{2},\ \Omega_{3},\ \Omega_{4}\rangle$\\
\end{longtable}
\endgroup

\subsubsection{$4$-dimensional central extensions of $\mu_5^3$}
There is only one orbit, $W={\rm H}^2(\mu_5^3,\mathbb{C})=\langle\Omega_1,\Omega_2,\Omega_3,\Omega_4\rangle$, which gives the algebra $\omega^{3,4}_{1}$.

\subsection{Central extensions of $\mu_5^4$} Let us use the following notations:
\[\nabla_1=[\Delta_{2,1}],\ \nabla_2=[\Delta_{4,1}], \ \nabla_3=[\Delta_{4,4}],\ \nabla_4=[\Delta_{2,4}+\Delta_{5,1}].\]
Take $\theta=\alpha_1\nabla_1+\alpha_2\nabla_2+\alpha_3\nabla_3+\alpha_4\nabla_4\in {\rm H}^{2}(\mu_5^4,\mathbb{C}).$
Computing $\phi_4^5\theta(x,y)=\theta(\phi_4^5(x),\phi_4^5(y))$ with the automorphisms of Proposition \ref{prop-aut} and reducing modulo ${\rm B}^{2}(\mu_5^4,\mathbb{C})$ we have that the action of ${\rm Aut} (\mu_5^4)$ is given by
\[\alpha_1^\prime=a_1^3\alpha_1,\qquad \alpha_2^\prime=\varepsilon a_1^2\alpha_2,\qquad \alpha_3^\prime=a_1^2\alpha_3,\qquad \alpha_4^\prime=\varepsilon a_1^3\alpha_4,\qquad \varepsilon=\pm1.\]
Thus the action is diagonal, but it depends only on one continuous parameter $a_1$ and the sign $\varepsilon$; in particular, only one coefficient of each subspace can be normalized by $a_1$, and the sign $\varepsilon$ gives the identifications indicated below. By Remark \ref{rem-34}, none of the orbits lies in $T_s(\mu_5^4)$.

\subsubsection{$1$-dimensional central extensions of $\mu_5^4$}
The one-dimensional central extensions of $\mu_5^4$ were classified in \cite{Adashev2023}. In the notation of the present paper the result is as follows: there are $15$ orbits (here, for instance, $\langle\nabla_1+\alpha\nabla_4\rangle$ and $\langle\nabla_1-\alpha\nabla_4\rangle$ are related by the automorphism with $a_1=1$, $\varepsilon=-1$):
\begingroup\small
\begin{longtable}{ll}
$\omega^{4,1}_{1}$: & $\langle \nabla_{1}\rangle$\\
$\omega^{4,1}_{2}$: & $\langle \nabla_{1}+\nabla_{2}\rangle$\\
$\omega^{4,1}_{3}$: & $\langle \nabla_{1}+\nabla_{3}\rangle$\\
$\omega^{4,1}_{4}$: & $\langle \nabla_{1}+\alpha \nabla_{4}\rangle$ ($\alpha\neq0;\ \alpha\sim -\alpha$)\\
$\omega^{4,1}_{5}$: & $\langle \nabla_{1}+\alpha \nabla_{2}+\nabla_{3}\rangle$ ($\alpha\neq0;\ \alpha\sim -\alpha$)\\
$\omega^{4,1}_{6}$: & $\langle \nabla_{1}+\nabla_{2}+\alpha \nabla_{4}\rangle$ ($\alpha\neq0;\ \alpha\sim -\alpha$)\\
$\omega^{4,1}_{7}$: & $\langle \nabla_{1}+\nabla_{3}+\alpha \nabla_{4}\rangle$ ($\alpha\neq0;\ \alpha\sim -\alpha$)\\
$\omega^{4,1}_{8}$: & $\langle \nabla_{1}+\alpha \nabla_{2}+\nabla_{3}+\beta \nabla_{4}\rangle$ ($\alpha\neq0,\ \beta\neq0;\ (\alpha,\beta)\sim(-\alpha,-\beta)$)\\
$\omega^{4,1}_{9}$: & $\langle \nabla_{2}\rangle$\\
$\omega^{4,1}_{10}$: & $\langle \nabla_{2}+\alpha \nabla_{3}\rangle$ ($\alpha\neq0;\ \alpha\sim -\alpha$)\\
$\omega^{4,1}_{11}$: & $\langle \nabla_{2}+\nabla_{4}\rangle$\\
$\omega^{4,1}_{12}$: & $\langle \nabla_{2}+\alpha \nabla_{3}+\nabla_{4}\rangle$ ($\alpha\neq0;\ \alpha\sim -\alpha$)\\
$\omega^{4,1}_{13}$: & $\langle \nabla_{3}\rangle$\\
$\omega^{4,1}_{14}$: & $\langle \nabla_{3}+\nabla_{4}\rangle$\\
$\omega^{4,1}_{15}$: & $\langle \nabla_{4}\rangle$\\
\end{longtable}
\endgroup

\subsubsection{$2$-dimensional central extensions of $\mu_5^4$}
As for $\mu_5^3$, we write $W=\langle\theta_1,\theta_2\rangle$ in the reduced row echelon form, $\theta_1=\nabla_{p_1}+\sum\alpha_j\nabla_j$, $\theta_2=\nabla_{p_2}+\sum\beta_j\nabla_j$, and we indicate the choice of $a_1$ and $\varepsilon$ which normalizes one of the coefficients. We obtain the following $35$ orbits:
\begingroup\small
\begin{longtable}{lll}
$\omega^{4,2}_{1}$ &  & $\langle \nabla_{1},\ \nabla_{2}\rangle$\\
$\omega^{4,2}_{2}$ & $a_1=\alpha_{3},\ \varepsilon=1$ & $\langle \nabla_{1}+\nabla_{3},\ \nabla_{2}\rangle$\\
$\omega^{4,2}_{3}$ &  & $\langle \nabla_{1}+\alpha \nabla_{4},\ \nabla_{2}\rangle$ ($\alpha\neq0;\ \alpha\sim -\alpha$)\\
$\omega^{4,2}_{4}$ &  & $\langle \nabla_{1},\ \nabla_{2}+\alpha \nabla_{3}\rangle$ ($\alpha\neq0;\ \alpha\sim -\alpha$)\\
$\omega^{4,2}_{5}$ & $a_1=\frac{1}{\beta_{4}},\ \varepsilon=1$ & $\langle \nabla_{1},\ \nabla_{2}+\nabla_{4}\rangle$\\
$\omega^{4,2}_{6}$ & $a_1=\alpha_{3},\ \varepsilon=1$ & $\langle \nabla_{1}+\nabla_{3}+\alpha \nabla_{4},\ \nabla_{2}\rangle$ ($\alpha\neq0;\ \alpha\sim -\alpha$)\\
$\omega^{4,2}_{7}$ & $a_1=\alpha_{3},\ \varepsilon=1$ & $\langle \nabla_{1}+\nabla_{3},\ \nabla_{2}+\alpha \nabla_{3}\rangle$ ($\alpha\neq0;\ \alpha\sim -\alpha$)\\
$\omega^{4,2}_{8}$ & $a_1=\alpha_{3},\ \varepsilon=1$ & $\langle \nabla_{1}+\nabla_{3},\ \nabla_{2}+\alpha \nabla_{4}\rangle$ ($\alpha\neq0$)\\
$\omega^{4,2}_{9}$ &  & $\langle \nabla_{1}+\alpha \nabla_{4},\ \nabla_{2}+\beta \nabla_{3}\rangle$ ($\alpha\neq0,\ \beta\neq0;\ (\alpha,\beta)\sim(-\alpha,-\beta)$)\\
$\omega^{4,2}_{10}$ & $a_1=\frac{1}{\beta_{4}},\ \varepsilon=1$ & $\langle \nabla_{1}+\alpha \nabla_{4},\ \nabla_{2}+\nabla_{4}\rangle$ ($\alpha\neq0;\ \alpha\sim -\alpha$)\\
$\omega^{4,2}_{11}$ & $a_1=\frac{1}{\beta_{4}},\ \varepsilon=1$ & $\langle \nabla_{1},\ \nabla_{2}+\alpha \nabla_{3}+\nabla_{4}\rangle$ ($\alpha\neq0;\ \alpha\sim -\alpha$)\\
$\omega^{4,2}_{12}$ & $a_1=\alpha_{3},\ \varepsilon=1$ & $\langle \nabla_{1}+\nabla_{3}+\alpha \nabla_{4},\ \nabla_{2}+\beta \nabla_{3}\rangle$ ($\alpha\neq0,\ \beta\neq0;\ (\alpha,\beta)\sim(-\alpha,-\beta)$)\\
$\omega^{4,2}_{13}$ & $a_1=\alpha_{3},\ \varepsilon=1$ & $\langle \nabla_{1}+\nabla_{3}+\alpha \nabla_{4},\ \nabla_{2}+\beta \nabla_{4}\rangle$ ($\alpha\neq0,\ \beta\neq0;\ (\alpha,\beta)\sim(-\alpha,\beta)$)\\
$\omega^{4,2}_{14}$ & $a_1=\alpha_{3},\ \varepsilon=1$ & $\langle \nabla_{1}+\nabla_{3},\ \nabla_{2}+\alpha \nabla_{3}+\beta \nabla_{4}\rangle$ ($\alpha\neq0,\ \beta\neq0;\ (\alpha,\beta)\sim(-\alpha,\beta)$)\\
$\omega^{4,2}_{15}$ & $a_1=\frac{1}{\beta_{4}},\ \varepsilon=1$ & $\langle \nabla_{1}+\alpha \nabla_{4},\ \nabla_{2}+\beta \nabla_{3}+\nabla_{4}\rangle$ ($\alpha\neq0,\ \beta\neq0;\ (\alpha,\beta)\sim(-\alpha,-\beta)$)\\
$\omega^{4,2}_{16}$ & $a_1=\alpha_{3},\ \varepsilon=1$ & $\langle \nabla_{1}+\nabla_{3}+\alpha \nabla_{4},\ \nabla_{2}+\beta \nabla_{3}+\gamma \nabla_{4}\rangle$ ($\alpha\neq0,\ \beta\neq0,\ \gamma\neq0;\ (\alpha,\beta,\gamma)\sim(-\alpha,-\beta,\gamma)$)\\
$\omega^{4,2}_{17}$ &  & $\langle \nabla_{1},\ \nabla_{3}\rangle$\\
$\omega^{4,2}_{18}$ & $a_1=\alpha_{2},\ \varepsilon=1$ & $\langle \nabla_{1}+\nabla_{2},\ \nabla_{3}\rangle$\\
$\omega^{4,2}_{19}$ &  & $\langle \nabla_{1}+\alpha \nabla_{4},\ \nabla_{3}\rangle$ ($\alpha\neq0;\ \alpha\sim -\alpha$)\\
$\omega^{4,2}_{20}$ & $a_1=\frac{1}{\beta_{4}},\ \varepsilon=1$ & $\langle \nabla_{1},\ \nabla_{3}+\nabla_{4}\rangle$\\
$\omega^{4,2}_{21}$ & $a_1=\alpha_{2},\ \varepsilon=1$ & $\langle \nabla_{1}+\nabla_{2}+\alpha \nabla_{4},\ \nabla_{3}\rangle$ ($\alpha\neq0;\ \alpha\sim -\alpha$)\\
$\omega^{4,2}_{22}$ & $a_1=\alpha_{2},\ \varepsilon=1$ & $\langle \nabla_{1}+\nabla_{2},\ \nabla_{3}+\alpha \nabla_{4}\rangle$ ($\alpha\neq0$)\\
$\omega^{4,2}_{23}$ & $a_1=\frac{1}{\beta_{4}},\ \varepsilon=1$ & $\langle \nabla_{1}+\alpha \nabla_{4},\ \nabla_{3}+\nabla_{4}\rangle$ ($\alpha\neq0;\ \alpha\sim -\alpha$)\\
$\omega^{4,2}_{24}$ & $a_1=\alpha_{2},\ \varepsilon=1$ & $\langle \nabla_{1}+\nabla_{2}+\alpha \nabla_{4},\ \nabla_{3}+\beta \nabla_{4}\rangle$ ($\alpha\neq0,\ \beta\neq0;\ (\alpha,\beta)\sim(-\alpha,\beta)$)\\
$\omega^{4,2}_{25}$ &  & $\langle \nabla_{1},\ \nabla_{4}\rangle$\\
$\omega^{4,2}_{26}$ & $a_1=\alpha_{2},\ \varepsilon=1$ & $\langle \nabla_{1}+\nabla_{2},\ \nabla_{4}\rangle$\\
$\omega^{4,2}_{27}$ & $a_1=\alpha_{3},\ \varepsilon=1$ & $\langle \nabla_{1}+\nabla_{3},\ \nabla_{4}\rangle$\\
$\omega^{4,2}_{28}$ & $a_1=\alpha_{3},\ \varepsilon=1$ & $\langle \nabla_{1}+\alpha \nabla_{2}+\nabla_{3},\ \nabla_{4}\rangle$ ($\alpha\neq0;\ \alpha\sim -\alpha$)\\
$\omega^{4,2}_{29}$ &  & $\langle \nabla_{2},\ \nabla_{3}\rangle$\\
$\omega^{4,2}_{30}$ & $a_1=\frac{1}{\alpha_{4}},\ \varepsilon=1$ & $\langle \nabla_{2}+\nabla_{4},\ \nabla_{3}\rangle$\\
$\omega^{4,2}_{31}$ & $a_1=\frac{1}{\beta_{4}},\ \varepsilon=1$ & $\langle \nabla_{2},\ \nabla_{3}+\nabla_{4}\rangle$\\
$\omega^{4,2}_{32}$ & $a_1=\frac{1}{\alpha_{4}},\ \varepsilon=1$ & $\langle \nabla_{2}+\nabla_{4},\ \nabla_{3}+\alpha \nabla_{4}\rangle$ ($\alpha\neq0;\ \alpha\sim -\alpha$)\\
$\omega^{4,2}_{33}$ &  & $\langle \nabla_{2},\ \nabla_{4}\rangle$\\
$\omega^{4,2}_{34}$ &  & $\langle \nabla_{2}+\alpha \nabla_{3},\ \nabla_{4}\rangle$ ($\alpha\neq0;\ \alpha\sim -\alpha$)\\
$\omega^{4,2}_{35}$ &  & $\langle \nabla_{3},\ \nabla_{4}\rangle$\\
\end{longtable}
\endgroup

\subsubsection{$3$-dimensional central extensions of $\mu_5^4$}
We obtain the following $15$ orbits:
\begingroup\small
\begin{longtable}{lll}
$\omega^{4,3}_{1}$ &  & $\langle \nabla_{1},\ \nabla_{2},\ \nabla_{3}\rangle$\\
$\omega^{4,3}_{2}$ &  & $\langle \nabla_{1}+\alpha \nabla_{4},\ \nabla_{2},\ \nabla_{3}\rangle$ ($\alpha\neq0;\ \alpha\sim -\alpha$)\\
$\omega^{4,3}_{3}$ & $a_1=\frac{1}{\beta_{4}},\ \varepsilon=1$ & $\langle \nabla_{1},\ \nabla_{2}+\nabla_{4},\ \nabla_{3}\rangle$\\
$\omega^{4,3}_{4}$ & $a_1=\frac{1}{\gamma_{4}},\ \varepsilon=1$ & $\langle \nabla_{1},\ \nabla_{2},\ \nabla_{3}+\nabla_{4}\rangle$\\
$\omega^{4,3}_{5}$ & $a_1=\frac{1}{\beta_{4}},\ \varepsilon=1$ & $\langle \nabla_{1}+\alpha \nabla_{4},\ \nabla_{2}+\nabla_{4},\ \nabla_{3}\rangle$ ($\alpha\neq0;\ \alpha\sim -\alpha$)\\
$\omega^{4,3}_{6}$ & $a_1=\frac{1}{\gamma_{4}},\ \varepsilon=1$ & $\langle \nabla_{1}+\alpha \nabla_{4},\ \nabla_{2},\ \nabla_{3}+\nabla_{4}\rangle$ ($\alpha\neq0;\ \alpha\sim -\alpha$)\\
$\omega^{4,3}_{7}$ & $a_1=\frac{1}{\beta_{4}},\ \varepsilon=1$ & $\langle \nabla_{1},\ \nabla_{2}+\nabla_{4},\ \nabla_{3}+\alpha \nabla_{4}\rangle$ ($\alpha\neq0;\ \alpha\sim -\alpha$)\\
$\omega^{4,3}_{8}$ & $a_1=\frac{1}{\beta_{4}},\ \varepsilon=1$ & $\langle \nabla_{1}+\alpha \nabla_{4},\ \nabla_{2}+\nabla_{4},\ \nabla_{3}+\beta \nabla_{4}\rangle$ ($\alpha\neq0,\ \beta\neq0;\ (\alpha,\beta)\sim(-\alpha,-\beta)$)\\
$\omega^{4,3}_{9}$ &  & $\langle \nabla_{1},\ \nabla_{2},\ \nabla_{4}\rangle$\\
$\omega^{4,3}_{10}$ & $a_1=\alpha_{3},\ \varepsilon=1$ & $\langle \nabla_{1}+\nabla_{3},\ \nabla_{2},\ \nabla_{4}\rangle$\\
$\omega^{4,3}_{11}$ &  & $\langle \nabla_{1},\ \nabla_{2}+\alpha \nabla_{3},\ \nabla_{4}\rangle$ ($\alpha\neq0;\ \alpha\sim -\alpha$)\\
$\omega^{4,3}_{12}$ & $a_1=\alpha_{3},\ \varepsilon=1$ & $\langle \nabla_{1}+\nabla_{3},\ \nabla_{2}+\alpha \nabla_{3},\ \nabla_{4}\rangle$ ($\alpha\neq0;\ \alpha\sim -\alpha$)\\
$\omega^{4,3}_{13}$ &  & $\langle \nabla_{1},\ \nabla_{3},\ \nabla_{4}\rangle$\\
$\omega^{4,3}_{14}$ & $a_1=\alpha_{2},\ \varepsilon=1$ & $\langle \nabla_{1}+\nabla_{2},\ \nabla_{3},\ \nabla_{4}\rangle$\\
$\omega^{4,3}_{15}$ &  & $\langle \nabla_{2},\ \nabla_{3},\ \nabla_{4}\rangle$\\
\end{longtable}
\endgroup

\subsubsection{$4$-dimensional central extensions of $\mu_5^4$}
There is only one orbit, $W={\rm H}^2(\mu_5^4,\mathbb{C})$, which gives the algebra $\omega^{4,4}_{1}$.

\section{The main result}

Summarizing the results of Section 3 we obtain the following theorem.

\begin{theorem}\label{main}
Let $s\geq1$. Every $s$-dimensional central extension $(\mu_k^t)_\theta$ without central components of the naturally graded $2$-filiform Leibniz algebras $\mu_n^1,\mu_n^2$ $(n\geq5)$, $\mu_5^3$, $\mu_5^4$ is isomorphic, by an isomorphism preserving the central ideal $\mathbb V$, to one of the algebras $\omega^{t,s}_k$ listed in the Appendix, and these algebras are pairwise non-isomorphic by isomorphisms preserving $\mathbb V$ (with the identifications of the parameters indicated in the lists). The numbers of items in the lists (a family depending on parameters is counted as one item) and the numbers of items lying in $T_s$ are as follows:
\begin{center}
\begin{tabular}{|c|c|c|c|c|c|}
\hline
 & $s=1$ & $s=2$ & $s=3$ & $s=4$ & $s=5$\\ \hline
$\mu_n^1$ & $11$ $(1)$ & $37$ $(18)$ & $37$ $(29)$ & $11$ $(10)$ & $1$ $(1)$\\ \hline
$\mu_n^2$ & $15$ $(2)$ & $52$ $(25)$ & $53$ $(42)$ & $15$ $(14)$ & $1$ $(1)$\\ \hline
$\mu_5^3$ & $15$ $(0)$ & $35$ $(0)$ & $15$ $(0)$ & $1$ $(0)$ & --\\ \hline
$\mu_5^4$ & $15$ $(0)$ & $35$ $(0)$ & $15$ $(0)$ & $1$ $(0)$ & --\\ \hline
\end{tabular}
\end{center}
(the number of items lying in $T_s$ is given in brackets). In particular, an arbitrary non-split $s$-dimensional central extension of a naturally graded $2$-filiform non-Lie Leibniz algebra with $Ann=\mathbb V$ is isomorphic to one of the algebras marked by an asterisk, and these algebras are pairwise non-isomorphic. For $s\geq6$ ($s\geq 6$ for $\mu_n^1,\mu_n^2$ and $s\geq5$ for $\mu_5^3,\mu_5^4$) there are no such extensions, because $\dim {\rm H}^2$ is equal to $5$ (respectively to $4$).
\end{theorem}

\begin{remark} For one-dimensional central extensions we have: the only non-split one-dimensional central extension of $\mu_n^1$ with $Ann=\mathbb V$ is $\omega^{1,1}_1$ ($\langle\nabla_1+\nabla_5\rangle$); for $\mu_n^2$ these are $\omega^{2,1}_8$ ($\langle\nabla_1+\nabla_4\rangle$) and $\omega^{2,1}_{15}$ ($\langle\nabla_1+\nabla_5\rangle$); $\mu_5^3$ and $\mu_5^4$ have no such extensions. All the other one-dimensional extensions have a two- or three-dimensional annihilator and can also be obtained as central extensions of smaller algebras.
\end{remark}

\section{Appendix: The list of the algebras}
In the following tables we list, for every algebra $\omega^{t,s}_k$, only the products which differ from the products of $\mu_n^t$ (respectively $\mu_5^t$) given in Theorems \ref{thm2.4} and \ref{thm2.5}; a product $[e_i,e_j]$ which is non-zero in $\mu_n^t$ is written in full. The algebras whose subspace $W$ lies in $T_s$ are marked by an asterisk. The products are obtained from the basis $\theta_1,\dots,\theta_s$ of the representative $W$ listed in Section 3, taken in the order in which it is written there.

\begingroup\footnotesize
\begin{longtable}{|p{2.0cm}|p{14.0cm}|}
\hline
\multicolumn{2}{|c|}{Central extensions of $\mu_n^1$; dimension $n+1$}\\ \hline
$\omega^{1,1}_{1}$$^{\ast}$ & $[e_{n-2},e_{1}]=e_{n+1}$,\ $[e_{n},e_{n-1}]=e_{n+1}$\\ \hline
$\omega^{1,1}_{2}$ & $[e_{n},e_{n-1}]=e_{n+1}$\\ \hline
$\omega^{1,1}_{3}$ & $[e_{n-1},e_{1}]=e_{n+1}$,\ $[e_{n},e_{n-1}]=e_{n+1}$\\ \hline
$\omega^{1,1}_{4}$ & $[e_{n-2},e_{1}]=e_{n+1}$\\ \hline
$\omega^{1,1}_{5}$ & $[e_{n-1},e_{n-1}]=e_{n+1}$\\ \hline
$\omega^{1,1}_{6}$ & $[e_{n},e_{1}]=e_{n+1}$,\ $[e_{2},e_{n-1}]=e_{n+1}$\\ \hline
$\omega^{1,1}_{7}$ & $[e_{n-2},e_{1}]=e_{n+1}$,\ $[e_{n-1},e_{n-1}]=e_{n+1}$\\ \hline
$\omega^{1,1}_{8}$ & $[e_{n-2},e_{1}]=e_{n+1}$,\ $[e_{n},e_{1}]=e_{n+1}$,\ $[e_{2},e_{n-1}]=e_{n+1}$\\ \hline
$\omega^{1,1}_{9}$ & $[e_{n},e_{1}]=e_{n+1}$,\ $[e_{2},e_{n-1}]=e_{n+1}$,\ $[e_{n-1},e_{n-1}]=e_{n+1}$\\ \hline
$\omega^{1,1}_{10}$ & $[e_{n-2},e_{1}]=e_{n+1}$,\ $[e_{n},e_{1}]=e_{n+1}$,\ $[e_{2},e_{n-1}]=e_{n+1}$,\ $[e_{n-1},e_{n-1}]=e_{n+1}$ ($n\neq 5$)\\ \hline
$\omega^{1,1}_{10}(\alpha)$ & $[e_{n-2},e_{1}]=\alpha e_{n+1}$,\ $[e_{n},e_{1}]=e_{n+1}$,\ $[e_{2},e_{n-1}]=e_{n+1}$,\ $[e_{n-1},e_{n-1}]=e_{n+1}$ ($n=5,\ \alpha\neq0$)\\ \hline
$\omega^{1,1}_{11}$ & $[e_{n-1},e_{1}]=e_{n+1}$\\ \hline
\end{longtable}
\endgroup

\begingroup\footnotesize
\begin{longtable}{|p{2.0cm}|p{14.0cm}|}
\hline
\multicolumn{2}{|c|}{Central extensions of $\mu_n^1$; dimension $n+2$}\\ \hline
$\omega^{1,2}_{1}(\alpha)$$^{\ast}$ & $[e_{n-2},e_{1}]=\alpha e_{n+2}$,\ $[e_{n-1},e_{1}]=e_{n+1}$,\ $[e_{n},e_{1}]=e_{n+2}$,\ $[e_{2},e_{n-1}]=e_{n+2}$,\ $[e_{n-1},e_{n-1}]=e_{n+2}$,\ $[e_{n},e_{n-1}]=e_{n+1}$ ($\alpha\neq0;\ \alpha\sim-\alpha \text{ if } n \text{ is even}$)\\ \hline
$\omega^{1,2}_{2}$$^{\ast}$ & $[e_{n-2},e_{1}]=e_{n+2}$,\ $[e_{n},e_{1}]=e_{n+2}$,\ $[e_{2},e_{n-1}]=e_{n+2}$,\ $[e_{n-1},e_{n-1}]=e_{n+2}$,\ $[e_{n},e_{n-1}]=e_{n+1}$ ($n\neq 5$)\\ \hline
$\omega^{1,2}_{2}(\alpha)$$^{\ast}$ & $[e_{n-2},e_{1}]=\alpha e_{n+2}$,\ $[e_{n},e_{1}]=e_{n+2}$,\ $[e_{2},e_{n-1}]=e_{n+2}$,\ $[e_{n-1},e_{n-1}]=e_{n+2}$,\ $[e_{n},e_{n-1}]=e_{n+1}$ ($n=5,\ \alpha\neq0$)\\ \hline
$\omega^{1,2}_{3}$ & $[e_{n},e_{1}]=e_{n+2}$,\ $[e_{2},e_{n-1}]=e_{n+2}$,\ $[e_{n-1},e_{n-1}]=e_{n+2}$,\ $[e_{n},e_{n-1}]=e_{n+1}$\\ \hline
$\omega^{1,2}_{4}$ & $[e_{n-1},e_{1}]=e_{n+1}$,\ $[e_{n},e_{1}]=e_{n+2}$,\ $[e_{2},e_{n-1}]=e_{n+2}$,\ $[e_{n-1},e_{n-1}]=e_{n+2}$,\ $[e_{n},e_{n-1}]=e_{n+1}$\\ \hline
$\omega^{1,2}_{5}$$^{\ast}$ & $[e_{n-2},e_{1}]=e_{n+1}$,\ $[e_{n},e_{1}]=e_{n+2}$,\ $[e_{2},e_{n-1}]=e_{n+2}$,\ $[e_{n-1},e_{n-1}]=e_{n+2}$,\ $[e_{n},e_{n-1}]=e_{n+1}$ ($n\neq 6$)\\ \hline
$\omega^{1,2}_{5}(\alpha)$$^{\ast}$ & $[e_{n-2},e_{1}]=\alpha e_{n+1}$,\ $[e_{n},e_{1}]=e_{n+2}$,\ $[e_{2},e_{n-1}]=e_{n+2}$,\ $[e_{n-1},e_{n-1}]=e_{n+2}$,\ $[e_{n},e_{n-1}]=e_{n+1}$ ($n=6,\ \alpha\neq0$)\\
\hline
\end{longtable}
\endgroup

\begingroup\footnotesize
\begin{longtable}{|p{2.0cm}|p{14.0cm}|}
\hline
\multicolumn{2}{|c|}{Central extensions of $\mu_n^1$; dimension $n+3$}\\ \hline
$\omega^{1,3}_{1}$$^{\ast}$ & $[e_{n-2},e_{1}]=e_{n+3}$,\ $[e_{n-1},e_{1}]=e_{n+1}$,\ $[e_{n},e_{n-1}]=e_{n+2}$\\ \hline
$\omega^{1,3}_{2}$$^{\ast}$ & $[e_{n-2},e_{1}]=e_{n+3}$,\ $[e_{n-1},e_{1}]=e_{n+1}$,\ $[e_{n-1},e_{n-1}]=e_{n+3}$,\ $[e_{n},e_{n-1}]=e_{n+2}$\\ \hline
$\omega^{1,3}_{3}$$^{\ast}$ & $[e_{n-2},e_{1}]=e_{n+3}$,\ $[e_{n-1},e_{1}]=e_{n+1}$,\ $[e_{n},e_{1}]=e_{n+3}$,\ $[e_{2},e_{n-1}]=e_{n+3}$,\ $[e_{n},e_{n-1}]=e_{n+2}$\\ \hline
$\omega^{1,3}_{4}$$^{\ast}$ & $[e_{n-2},e_{1}]=e_{n+3}$,\ $[e_{n-1},e_{1}]=e_{n+1}$,\ $[e_{n},e_{1}]=e_{n+3}$,\ $[e_{2},e_{n-1}]=e_{n+3}$,\ $[e_{n-1},e_{n-1}]=e_{n+3}$,\ $[e_{n},e_{n-1}]=e_{n+2}$ ($n\neq 5$)\\ \hline
$\omega^{1,3}_{4}(\alpha)$$^{\ast}$ & $[e_{n-2},e_{1}]=e_{n+3}$,\ $[e_{n-1},e_{1}]=e_{n+1}$,\ $[e_{n},e_{1}]=e_{n+3}$,\ $[e_{2},e_{n-1}]=e_{n+3}$,\ $[e_{n-1},e_{n-1}]=\alpha e_{n+3}$,\ $[e_{n},e_{n-1}]=e_{n+2}$ ($n=5,\ \alpha\neq0$)\\ \hline
$\omega^{1,3}_{5}$ & $[e_{n-1},e_{1}]=e_{n+1}$,\ $[e_{n-1},e_{n-1}]=e_{n+2}$,\ $[e_{n},e_{n-1}]=e_{n+3}$\\ \hline
$\omega^{1,3}_{6}$$^{\ast}$ & $[e_{n-2},e_{1}]=e_{n+3}$,\ $[e_{n-1},e_{1}]=e_{n+1}$,\ $[e_{n-1},e_{n-1}]=e_{n+2}$,\ $[e_{n},e_{n-1}]=e_{n+3}$\\
\hline
\end{longtable}
\endgroup

\begingroup\footnotesize
\begin{longtable}{|p{2.0cm}|p{14.0cm}|}
\hline
\multicolumn{2}{|c|}{Central extensions of $\mu_n^1$; dimension $n+4$}\\ \hline
$\omega^{1,4}_{1}$$^{\ast}$ & $[e_{n-2},e_{1}]=e_{n+1}$,\ $[e_{n-1},e_{1}]=e_{n+2}$,\ $[e_{n-1},e_{n-1}]=e_{n+3}$,\ $[e_{n},e_{n-1}]=e_{n+4}$\\ \hline
$\omega^{1,4}_{2}$$^{\ast}$ & $[e_{n-2},e_{1}]=e_{n+1}$,\ $[e_{n-1},e_{1}]=e_{n+2}$,\ $[e_{n},e_{1}]=e_{n+1}$,\ $[e_{2},e_{n-1}]=e_{n+1}$,\ $[e_{n-1},e_{n-1}]=e_{n+3}$,\ $[e_{n},e_{n-1}]=e_{n+4}$\\ \hline
$\omega^{1,4}_{3}$$^{\ast}$ & $[e_{n-2},e_{1}]=e_{n+1}$,\ $[e_{n-1},e_{1}]=e_{n+2}$,\ $[e_{n},e_{1}]=e_{n+3}$,\ $[e_{2},e_{n-1}]=e_{n+3}$,\ $[e_{n-1},e_{n-1}]=e_{n+3}$,\ $[e_{n},e_{n-1}]=e_{n+4}$\\ \hline
$\omega^{1,4}_{4}$$^{\ast}$ & $[e_{n-2},e_{1}]=e_{n+1}$,\ $[e_{n-1},e_{1}]=e_{n+2}$,\ $[e_{n},e_{1}]=e_{n+1}+e_{n+3}$,\ $[e_{2},e_{n-1}]=e_{n+1}+e_{n+3}$,\ $[e_{n-1},e_{n-1}]=e_{n+3}$,\ $[e_{n},e_{n-1}]=e_{n+4}$ ($n\neq 5$)\\ \hline
$\omega^{1,4}_{4}(\alpha)$$^{\ast}$ & $[e_{n-2},e_{1}]=e_{n+1}$,\ $[e_{n-1},e_{1}]=e_{n+2}$,\ $[e_{n},e_{1}]=e_{n+1}+\alpha e_{n+3}$,\ $[e_{2},e_{n-1}]=e_{n+1}+\alpha e_{n+3}$,\ $[e_{n-1},e_{n-1}]=e_{n+3}$,\ $[e_{n},e_{n-1}]=e_{n+4}$ ($n=5,\ \alpha\neq0$)\\ \hline
$\omega^{1,4}_{5}$$^{\ast}$ & $[e_{n-2},e_{1}]=e_{n+1}$,\ $[e_{n-1},e_{1}]=e_{n+2}$,\ $[e_{n},e_{1}]=e_{n+3}$,\ $[e_{2},e_{n-1}]=e_{n+3}$,\ $[e_{n},e_{n-1}]=e_{n+4}$\\ \hline
$\omega^{1,4}_{6}$$^{\ast}$ & $[e_{n-2},e_{1}]=e_{n+1}$,\ $[e_{n-1},e_{1}]=e_{n+2}$,\ $[e_{n},e_{1}]=e_{n+3}$,\ $[e_{2},e_{n-1}]=e_{n+3}$,\ $[e_{n-1},e_{n-1}]=e_{n+1}$,\ $[e_{n},e_{n-1}]=e_{n+4}$\\
\hline
\end{longtable}
\endgroup

\begingroup\footnotesize
\begin{longtable}{|p{2.0cm}|p{14.0cm}|}
\hline
\multicolumn{2}{|c|}{Central extensions of $\mu_n^1$; dimension $n+5$}\\ \hline
$\omega^{1,5}_{1}$$^{\ast}$ & $[e_{n-2},e_{1}]=e_{n+1}$,\ $[e_{n-1},e_{1}]=e_{n+2}$,\ $[e_{n},e_{1}]=e_{n+4}$,\ $[e_{2},e_{n-1}]=e_{n+4}$,\ $[e_{n-1},e_{n-1}]=e_{n+3}$,\ $[e_{n},e_{n-1}]=e_{n+5}$\\ \hline
\end{longtable}
\endgroup

\begingroup\footnotesize
\begin{longtable}{|p{2.0cm}|p{14.0cm}|}
\hline
\multicolumn{2}{|c|}{Central extensions of $\mu_n^2$; dimension $n+1$}\\ \hline
$\omega^{2,1}_{1}$ & $[e_{n-1},e_{1}]=e_{n+1}$\\ \hline
$\omega^{2,1}_{2}$ & $[e_{n-1},e_{n-1}]=e_{n+1}$\\ \hline
$\omega^{2,1}_{3}$ & $[e_{n-1},e_{1}]=e_{n+1}$,\ $[e_{n-1},e_{n-1}]=e_{n+1}$\\ \hline
$\omega^{2,1}_{4}$ & $[e_{n-2},e_{1}]=e_{n+1}$,\ $[e_{n-2},e_{n-1}]=e_{n+1}$\\ \hline
$\omega^{2,1}_{5}$ & $[e_{n-2},e_{1}]=e_{n+1}$,\ $[e_{n-1},e_{1}]=e_{n+1}$,\ $[e_{n-2},e_{n-1}]=e_{n+1}$\\ \hline
$\omega^{2,1}_{6}$ & $[e_{2},e_{1}]=e_3-e_{n+1}$,\ $[e_{n},e_{1}]=e_{n+1}$\\ \hline
$\omega^{2,1}_{7}$ & $[e_{2},e_{1}]=e_3-e_{n+1}$,\ $[e_{n},e_{1}]=e_{n+1}$,\ $[e_{n-1},e_{n-1}]=e_{n+1}$\\ \hline
$\omega^{2,1}_{8}$$^{\ast}$ & $[e_{2},e_{1}]=e_3-e_{n+1}$,\ $[e_{n-2},e_{1}]=e_{n+1}$,\ $[e_{n},e_{1}]=e_{n+1}$,\ $[e_{n-2},e_{n-1}]=e_{n+1}$\\ \hline
$\omega^{2,1}_{9}$ & $[e_{2},e_{1}]=e_3-e_{n+1}$,\ $[e_{n},e_{1}]=e_{n+1}$,\ $[e_{n},e_{n-1}]=e_{n+1}$\\ \hline
$\omega^{2,1}_{10}$ & $[e_{2},e_{1}]=e_3-e_{n+1}$,\ $[e_{n-1},e_{1}]=e_{n+1}$,\ $[e_{n},e_{1}]=e_{n+1}$,\ $[e_{n},e_{n-1}]=e_{n+1}$\\ \hline
$\omega^{2,1}_{11}$ & $[e_{n},e_{n-1}]=e_{n+1}$\\ \hline
$\omega^{2,1}_{12}$ & $[e_{n-1},e_{1}]=e_{n+1}$,\ $[e_{n},e_{n-1}]=e_{n+1}$\\ \hline
$\omega^{2,1}_{13}$ & $[e_{2},e_{1}]=e_3-e_{n+1}$,\ $[e_{n-2},e_{1}]=e_{n+1}$,\ $[e_{n},e_{1}]=e_{n+1}$,\ $[e_{n-2},e_{n-1}]=e_{n+1}$,\ $[e_{n},e_{n-1}]=e_{n+1}$\\ \hline
$\omega^{2,1}_{14}$ & $[e_{2},e_{1}]=e_3-e_{n+1}$,\ $[e_{n-2},e_{1}]=e_{n+1}$,\ $[e_{n-1},e_{1}]=e_{n+1}$,\ $[e_{n},e_{1}]=e_{n+1}$,\ $[e_{n-2},e_{n-1}]=e_{n+1}$,\ $[e_{n},e_{n-1}]=e_{n+1}$\\ \hline
$\omega^{2,1}_{15}$$^{\ast}$ & $[e_{n-2},e_{1}]=e_{n+1}$,\ $[e_{n-2},e_{n-1}]=e_{n+1}$,\ $[e_{n},e_{n-1}]=e_{n+1}$\\ \hline
\end{longtable}
\endgroup

\begingroup\footnotesize
\begin{longtable}{|p{2.0cm}|p{14.0cm}|}
\hline
\multicolumn{2}{|c|}{Central extensions of $\mu_n^2$; dimension $n+2$}\\ \hline
$\omega^{2,2}_{1}$ & $[e_{n-1},e_{1}]=e_{n+1}+e_{n+2}$,\ $[e_{n-1},e_{n-1}]=e_{n+2}$\\ \hline
$\omega^{2,2}_{2}$ & $[e_{2},e_{1}]=e_3-e_{n+2}$,\ $[e_{n-1},e_{1}]=e_{n+1}$,\ $[e_{n},e_{1}]=e_{n+2}$,\ $[e_{n-1},e_{n-1}]=e_{n+1}$\\ \hline
$\omega^{2,2}_{3}$$^{\ast}$ & $[e_{2},e_{1}]=e_3-e_{n+2}$,\ $[e_{n-2},e_{1}]=e_{n+2}$,\ $[e_{n-1},e_{1}]=e_{n+1}$,\ $[e_{n},e_{1}]=e_{n+2}$,\ $[e_{n-2},e_{n-1}]=e_{n+2}$,\ $[e_{n-1},e_{n-1}]=e_{n+1}$\\ \hline
$\omega^{2,2}_{4}$ & $[e_{2},e_{1}]=e_3-2e_{n+2}$,\ $[e_{n-1},e_{1}]=e_{n+1}$,\ $[e_{n},e_{1}]=2e_{n+2}$,\ $[e_{n-1},e_{n-1}]=e_{n+1}$,\ $[e_{n},e_{n-1}]=e_{n+2}$\\ \hline
$\omega^{2,2}_{5}$$^{\ast}$ & $[e_{2},e_{1}]=e_3-2e_{n+2}$,\ $[e_{n-2},e_{1}]=e_{n+2}$,\ $[e_{n-1},e_{1}]=e_{n+1}$,\ $[e_{n},e_{1}]=2e_{n+2}$,\ $[e_{n-2},e_{n-1}]=e_{n+2}$,\ $[e_{n-1},e_{n-1}]=e_{n+1}$,\ $[e_{n},e_{n-1}]=e_{n+2}$\\ \hline
$\omega^{2,2}_{6}$ & $[e_{n-2},e_{1}]=e_{n+2}$,\ $[e_{n-1},e_{1}]=e_{n+1}$,\ $[e_{n-2},e_{n-1}]=e_{n+2}$,\ $[e_{n-1},e_{n-1}]=e_{n+1}$\\ \hline
$\omega^{2,2}_{7}$ & $[e_{n-2},e_{1}]=e_{n+2}$,\ $[e_{n-1},e_{1}]=e_{n+1}+e_{n+2}$,\ $[e_{n-2},e_{n-1}]=e_{n+2}$,\ $[e_{n-1},e_{n-1}]=e_{n+1}$\\
\hline
\end{longtable}
\endgroup

\begingroup\footnotesize
\begin{longtable}{|p{2.0cm}|p{14.0cm}|}
\hline
\multicolumn{2}{|c|}{Central extensions of $\mu_n^2$; dimension $n+3$}\\ \hline
$\omega^{2,3}_{1}$ & $[e_{n-2},e_{1}]=e_{n+3}$,\ $[e_{n-1},e_{1}]=e_{n+1}+e_{n+2}$,\ $[e_{n-2},e_{n-1}]=e_{n+3}$,\ $[e_{n-1},e_{n-1}]=e_{n+2}$\\ \hline
$\omega^{2,3}_{2}$ & $[e_{2},e_{1}]=e_3-e_{n+3}$,\ $[e_{n-1},e_{1}]=e_{n+1}+e_{n+2}$,\ $[e_{n},e_{1}]=e_{n+3}$,\ $[e_{n-1},e_{n-1}]=e_{n+2}$\\ \hline
$\omega^{2,3}_{3}$ & $[e_{2},e_{1}]=e_3-e_{n+3}$,\ $[e_{n-1},e_{1}]=e_{n+1}+e_{n+2}$,\ $[e_{n},e_{1}]=e_{n+3}$,\ $[e_{n-1},e_{n-1}]=e_{n+2}$,\ $[e_{n},e_{n-1}]=e_{n+3}$\\ \hline
$\omega^{2,3}_{4}$$^{\ast}$ & $[e_{2},e_{1}]=e_3-e_{n+3}$,\ $[e_{n-2},e_{1}]=e_{n+3}$,\ $[e_{n-1},e_{1}]=e_{n+1}+e_{n+2}$,\ $[e_{n},e_{1}]=e_{n+3}$,\ $[e_{n-2},e_{n-1}]=e_{n+3}$,\ $[e_{n-1},e_{n-1}]=e_{n+2}$\\ \hline
$\omega^{2,3}_{5}$ & $[e_{2},e_{1}]=e_3-e_{n+3}$,\ $[e_{n-2},e_{1}]=e_{n+3}$,\ $[e_{n-1},e_{1}]=e_{n+1}+e_{n+2}$,\ $[e_{n},e_{1}]=e_{n+3}$,\ $[e_{n-2},e_{n-1}]=e_{n+3}$,\ $[e_{n-1},e_{n-1}]=e_{n+2}$,\ $[e_{n},e_{n-1}]=e_{n+3}$\\ \hline
$\omega^{2,3}_{6}$ & $[e_{2},e_{1}]=e_3-2e_{n+3}$,\ $[e_{n-1},e_{1}]=e_{n+1}+e_{n+2}$,\ $[e_{n},e_{1}]=2e_{n+3}$,\ $[e_{n-1},e_{n-1}]=e_{n+2}$,\ $[e_{n},e_{n-1}]=e_{n+3}$\\ \hline
$\omega^{2,3}_{7}$$^{\ast}$ & $[e_{2},e_{1}]=e_3-2e_{n+3}$,\ $[e_{n-2},e_{1}]=e_{n+3}$,\ $[e_{n-1},e_{1}]=e_{n+1}+e_{n+2}$,\ $[e_{n},e_{1}]=2e_{n+3}$,\ $[e_{n-2},e_{n-1}]=e_{n+3}$,\ $[e_{n-1},e_{n-1}]=e_{n+2}$,\ $[e_{n},e_{n-1}]=e_{n+3}$\\
\hline
\end{longtable}
\endgroup

\begingroup\footnotesize
\begin{longtable}{|p{2.0cm}|p{14.0cm}|}
\hline
\multicolumn{2}{|c|}{Central extensions of $\mu_n^2$; dimension $n+4$}\\ \hline
$\omega^{2,4}_{1}$ & $[e_{2},e_{1}]=e_3-e_{n+3}-e_{n+4}$,\ $[e_{n-1},e_{1}]=e_{n+1}+e_{n+2}$,\ $[e_{n},e_{1}]=e_{n+3}+e_{n+4}$,\ $[e_{n-1},e_{n-1}]=e_{n+2}$,\ $[e_{n},e_{n-1}]=e_{n+4}$\\ \hline
$\omega^{2,4}_{2}$$^{\ast}$ & $[e_{2},e_{1}]=e_3-e_{n+4}$,\ $[e_{n-2},e_{1}]=e_{n+3}$,\ $[e_{n-1},e_{1}]=e_{n+1}+e_{n+2}$,\ $[e_{n},e_{1}]=e_{n+4}$,\ $[e_{n-2},e_{n-1}]=e_{n+3}$,\ $[e_{n-1},e_{n-1}]=e_{n+2}$,\ $[e_{n},e_{n-1}]=e_{n+4}$\\ \hline
$\omega^{2,4}_{3}$$^{\ast}$ & $[e_{2},e_{1}]=e_3-e_{n+4}$,\ $[e_{n-2},e_{1}]=e_{n+3}$,\ $[e_{n-1},e_{1}]=e_{n+1}+e_{n+2}$,\ $[e_{n},e_{1}]=e_{n+4}$,\ $[e_{n-2},e_{n-1}]=e_{n+3}$,\ $[e_{n-1},e_{n-1}]=e_{n+2}$\\ \hline
$\omega^{2,4}_{4}$$^{\ast}$ & $[e_{2},e_{1}]=e_3-e_{n+3}-e_{n+4}$,\ $[e_{n-2},e_{1}]=e_{n+3}$,\ $[e_{n-1},e_{1}]=e_{n+1}+e_{n+2}$,\ $[e_{n},e_{1}]=e_{n+3}+e_{n+4}$,\ $[e_{n-2},e_{n-1}]=e_{n+3}$,\ $[e_{n-1},e_{n-1}]=e_{n+2}$,\ $[e_{n},e_{n-1}]=e_{n+4}$\\ \hline
$\omega^{2,4}_{5}$$^{\ast}$ & $[e_{2},e_{1}]=e_3-e_{n+3}-e_{n+4}$,\ $[e_{n-2},e_{1}]=e_{n+3}$,\ $[e_{n-1},e_{1}]=e_{n+1}+e_{n+2}$,\ $[e_{n},e_{1}]=e_{n+3}+e_{n+4}$,\ $[e_{n-2},e_{n-1}]=e_{n+3}$,\ $[e_{n-1},e_{n-1}]=e_{n+2}$,\ $[e_{n},e_{n-1}]=e_{n+3}$\\ \hline
$\omega^{2,4}_{6}$$^{\ast}$ & $[e_{2},e_{1}]=e_3-2e_{n+4}$,\ $[e_{n-2},e_{1}]=e_{n+3}$,\ $[e_{n-1},e_{1}]=e_{n+1}+e_{n+2}$,\ $[e_{n},e_{1}]=2e_{n+4}$,\ $[e_{n-2},e_{n-1}]=e_{n+3}$,\ $[e_{n-1},e_{n-1}]=e_{n+2}$,\ $[e_{n},e_{n-1}]=e_{n+4}$\\ \hline
$\omega^{2,4}_{7}$$^{\ast}$ & $[e_{2},e_{1}]=e_3-e_{n+3}-2e_{n+4}$,\ $[e_{n-2},e_{1}]=e_{n+3}$,\ $[e_{n-1},e_{1}]=e_{n+1}+e_{n+2}$,\ $[e_{n},e_{1}]=e_{n+3}+2e_{n+4}$,\ $[e_{n-2},e_{n-1}]=e_{n+3}$,\ $[e_{n-1},e_{n-1}]=e_{n+2}$,\ $[e_{n},e_{n-1}]=e_{n+3}+e_{n+4}$\\
\hline
\end{longtable}
\endgroup

\begingroup\footnotesize
\begin{longtable}{|p{2.0cm}|p{14.0cm}|}
\hline
\multicolumn{2}{|c|}{Central extensions of $\mu_n^2$; dimension $n+5$}\\ \hline
$\omega^{2,5}_{1}$$^{\ast}$ & $[e_{2},e_{1}]=e_3-e_{n+4}-e_{n+5}$,\ $[e_{n-2},e_{1}]=e_{n+1}$,\ $[e_{n-1},e_{1}]=e_{n+2}+e_{n+3}$,\ $[e_{n},e_{1}]=e_{n+4}+e_{n+5}$,\ $[e_{n-2},e_{n-1}]=e_{n+1}$,\ $[e_{n-1},e_{n-1}]=e_{n+3}$,\ $[e_{n},e_{n-1}]=e_{n+5}$\\ \hline
\end{longtable}
\endgroup

\begingroup\footnotesize
\begin{longtable}{|p{2.0cm}|p{14.0cm}|}
\hline
\multicolumn{2}{|c|}{Central extensions of $\mu_5^3$; dimension $6$}\\ \hline
$\omega^{3,1}_{1}$ & $[e_{5},e_{1}]=e_{6}$,\ $[e_{2},e_{4}]=ie_3+e_{6}$,\ $[e_{5},e_{4}]=e_3+3i e_{6}$\\ \hline
$\omega^{3,1}_{2}$ & $[e_{4},e_{1}]=e_{6}$,\ $[e_{5},e_{1}]=e_{6}$,\ $[e_{2},e_{4}]=ie_3+e_{6}$,\ $[e_{5},e_{4}]=e_3+3i e_{6}$\\ \hline
$\omega^{3,1}_{3}$ & $[e_{4},e_{1}]=-i e_{6}$,\ $[e_{5},e_{1}]=e_{6}$,\ $[e_{2},e_{4}]=ie_3+e_{6}$,\ $[e_{4},e_{4}]=2e_{6}$,\ $[e_{5},e_{4}]=e_3+3i e_{6}$\\ \hline
$\omega^{3,1}_{4}$ & $[e_{2},e_{1}]=e_3+i e_{6}$,\ $[e_{5},e_{1}]=2e_{6}$,\ $[e_{2},e_{4}]=ie_3+e_{6}$,\ $[e_{5},e_{4}]=e_3+3i e_{6}$\\ \hline
$\omega^{3,1}_{5}$ & $[e_{4},e_{1}]=(1-i)e_{6}$,\ $[e_{5},e_{1}]=e_{6}$,\ $[e_{2},e_{4}]=ie_3+e_{6}$,\ $[e_{4},e_{4}]=2e_{6}$,\ $[e_{5},e_{4}]=e_3+3i e_{6}$\\ \hline
$\omega^{3,1}_{6}$ & $[e_{2},e_{1}]=e_3+i e_{6}$,\ $[e_{4},e_{1}]=e_{6}$,\ $[e_{5},e_{1}]=2e_{6}$,\ $[e_{2},e_{4}]=ie_3+e_{6}$,\ $[e_{5},e_{4}]=e_3+3i e_{6}$\\ \hline
$\omega^{3,1}_{7}$ & $[e_{2},e_{1}]=e_3+i e_{6}$,\ $[e_{4},e_{1}]=-i e_{6}$,\ $[e_{5},e_{1}]=2e_{6}$,\ $[e_{2},e_{4}]=ie_3+e_{6}$,\ $[e_{4},e_{4}]=2e_{6}$,\ $[e_{5},e_{4}]=e_3+3i e_{6}$\\
\hline
\end{longtable}
\endgroup

\begingroup\footnotesize
\begin{longtable}{|p{2.0cm}|p{14.0cm}|}
\hline
\multicolumn{2}{|c|}{Central extensions of $\mu_5^3$; dimension $7$}\\ \hline
$\omega^{3,2}_{1}$ & $[e_{4},e_{1}]=e_{7}$,\ $[e_{5},e_{1}]=e_{6}$,\ $[e_{2},e_{4}]=ie_3+e_{6}$,\ $[e_{5},e_{4}]=e_3+3i e_{6}$\\ \hline
$\omega^{3,2}_{2}$ & $[e_{4},e_{1}]=-i e_{6}+e_{7}$,\ $[e_{5},e_{1}]=e_{6}$,\ $[e_{2},e_{4}]=ie_3+e_{6}$,\ $[e_{4},e_{4}]=2e_{6}$,\ $[e_{5},e_{4}]=e_3+3i e_{6}$\\ \hline
$\omega^{3,2}_{3}$ & $[e_{2},e_{1}]=e_3+i e_{6}$,\ $[e_{4},e_{1}]=e_{7}$,\ $[e_{5},e_{1}]=2e_{6}$,\ $[e_{2},e_{4}]=ie_3+e_{6}$,\ $[e_{5},e_{4}]=e_3+3i e_{6}$\\ \hline
$\omega^{3,2}_{4}$ & $[e_{4},e_{1}]=(1-i)e_{7}$,\ $[e_{5},e_{1}]=e_{6}$,\ $[e_{2},e_{4}]=ie_3+e_{6}$,\ $[e_{4},e_{4}]=2e_{7}$,\ $[e_{5},e_{4}]=e_3+3i e_{6}$\\ \hline
$\omega^{3,2}_{5}$ & $[e_{2},e_{1}]=e_3+i e_{7}$,\ $[e_{4},e_{1}]=e_{7}$,\ $[e_{5},e_{1}]=e_{6}+e_{7}$,\ $[e_{2},e_{4}]=ie_3+e_{6}$,\ $[e_{5},e_{4}]=e_3+3i e_{6}$\\ \hline
$\omega^{3,2}_{6}$ & $[e_{2},e_{1}]=e_3+i e_{6}$,\ $[e_{4},e_{1}]=-i e_{6}+e_{7}$,\ $[e_{5},e_{1}]=2e_{6}$,\ $[e_{2},e_{4}]=ie_3+e_{6}$,\ $[e_{4},e_{4}]=2e_{6}$,\ $[e_{5},e_{4}]=e_3+3i e_{6}$\\ \hline
$\omega^{3,2}_{7}$ & $[e_{4},e_{1}]=-i e_{6}+(1-i)e_{7}$,\ $[e_{5},e_{1}]=e_{6}$,\ $[e_{2},e_{4}]=ie_3+e_{6}$,\ $[e_{4},e_{4}]=2e_{6}+2e_{7}$,\ $[e_{5},e_{4}]=e_3+3i e_{6}$\\
\hline
\end{longtable}
\endgroup

\begingroup\footnotesize
\begin{longtable}{|p{2.0cm}|p{14.0cm}|}
\hline
\multicolumn{2}{|c|}{Central extensions of $\mu_5^3$; dimension $8$}\\ \hline
$\omega^{3,3}_{1}$ & $[e_{4},e_{1}]=e_{7}-i e_{8}$,\ $[e_{5},e_{1}]=e_{6}$,\ $[e_{2},e_{4}]=ie_3+e_{6}$,\ $[e_{4},e_{4}]=2e_{8}$,\ $[e_{5},e_{4}]=e_3+3i e_{6}$\\ \hline
$\omega^{3,3}_{2}$ & $[e_{2},e_{1}]=e_3+i e_{6}$,\ $[e_{4},e_{1}]=e_{7}-i e_{8}$,\ $[e_{5},e_{1}]=2e_{6}$,\ $[e_{2},e_{4}]=ie_3+e_{6}$,\ $[e_{4},e_{4}]=2e_{8}$,\ $[e_{5},e_{4}]=e_3+3i e_{6}$\\ \hline
$\omega^{3,3}_{3}$ & $[e_{2},e_{1}]=e_3+i e_{7}$,\ $[e_{4},e_{1}]=e_{7}-i e_{8}$,\ $[e_{5},e_{1}]=e_{6}+e_{7}$,\ $[e_{2},e_{4}]=ie_3+e_{6}$,\ $[e_{4},e_{4}]=2e_{8}$,\ $[e_{5},e_{4}]=e_3+3i e_{6}$\\ \hline
$\omega^{3,3}_{4}$ & $[e_{2},e_{1}]=e_3+i e_{8}$,\ $[e_{4},e_{1}]=e_{7}-i e_{8}$,\ $[e_{5},e_{1}]=e_{6}+e_{8}$,\ $[e_{2},e_{4}]=ie_3+e_{6}$,\ $[e_{4},e_{4}]=2e_{8}$,\ $[e_{5},e_{4}]=e_3+3i e_{6}$\\ \hline
$\omega^{3,3}_{5}$ & $[e_{2},e_{1}]=e_3+i e_{6}+i e_{7}$,\ $[e_{4},e_{1}]=e_{7}-i e_{8}$,\ $[e_{5},e_{1}]=2e_{6}+e_{7}$,\ $[e_{2},e_{4}]=ie_3+e_{6}$,\ $[e_{4},e_{4}]=2e_{8}$,\ $[e_{5},e_{4}]=e_3+3i e_{6}$\\ \hline
$\omega^{3,3}_{6}$ & $[e_{2},e_{1}]=e_3+i e_{6}+i e_{8}$,\ $[e_{4},e_{1}]=e_{7}-i e_{8}$,\ $[e_{5},e_{1}]=2e_{6}+e_{8}$,\ $[e_{2},e_{4}]=ie_3+e_{6}$,\ $[e_{4},e_{4}]=2e_{8}$,\ $[e_{5},e_{4}]=e_3+3i e_{6}$\\ \hline
$\omega^{3,3}_{7}$ & $[e_{2},e_{1}]=e_3+i e_{7}+i e_{8}$,\ $[e_{4},e_{1}]=e_{7}-i e_{8}$,\ $[e_{5},e_{1}]=e_{6}+e_{7}+e_{8}$,\ $[e_{2},e_{4}]=ie_3+e_{6}$,\ $[e_{4},e_{4}]=2e_{8}$,\ $[e_{5},e_{4}]=e_3+3i e_{6}$\\
\hline
\end{longtable}
\endgroup

\begingroup\footnotesize
\begin{longtable}{|p{2.0cm}|p{14.0cm}|}
\hline
\multicolumn{2}{|c|}{Central extensions of $\mu_5^3$; dimension $9$}\\ \hline
$\omega^{3,4}_{1}$ & $[e_{2},e_{1}]=e_3+i e_{9}$,\ $[e_{4},e_{1}]=e_{7}-i e_{8}$,\ $[e_{5},e_{1}]=e_{6}+e_{9}$,\ $[e_{2},e_{4}]=ie_3+e_{6}$,\ $[e_{4},e_{4}]=2e_{8}$,\ $[e_{5},e_{4}]=e_3+3i e_{6}$\\ \hline
\end{longtable}
\endgroup

\begingroup\footnotesize
\begin{longtable}{|p{2.0cm}|p{14.0cm}|}
\hline
\multicolumn{2}{|c|}{Central extensions of $\mu_5^4$; dimension $6$}\\ \hline
$\omega^{4,1}_{1}$ & $[e_{2},e_{1}]=e_3+e_{6}$\\ \hline
$\omega^{4,1}_{2}$ & $[e_{2},e_{1}]=e_3+e_{6}$,\ $[e_{4},e_{1}]=e_{6}$\\ \hline
$\omega^{4,1}_{3}$ & $[e_{2},e_{1}]=e_3+e_{6}$,\ $[e_{4},e_{4}]=e_{6}$\\ \hline
$\omega^{4,1}_{4}(\alpha)$ & $[e_{2},e_{1}]=e_3+e_{6}$,\ $[e_{5},e_{1}]=\alpha e_{6}$,\ $[e_{2},e_{4}]=\alpha e_{6}$ ($\alpha\neq0;\ \alpha\sim -\alpha$)\\ \hline
$\omega^{4,1}_{5}(\alpha)$ & $[e_{2},e_{1}]=e_3+e_{6}$,\ $[e_{4},e_{1}]=\alpha e_{6}$,\ $[e_{4},e_{4}]=e_{6}$ ($\alpha\neq0;\ \alpha\sim -\alpha$)\\ \hline
$\omega^{4,1}_{6}(\alpha)$ & $[e_{2},e_{1}]=e_3+e_{6}$,\ $[e_{4},e_{1}]=e_{6}$,\ $[e_{5},e_{1}]=\alpha e_{6}$,\ $[e_{2},e_{4}]=\alpha e_{6}$ ($\alpha\neq0;\ \alpha\sim -\alpha$)\\ \hline
$\omega^{4,1}_{7}(\alpha)$ & $[e_{2},e_{1}]=e_3+e_{6}$,\ $[e_{5},e_{1}]=\alpha e_{6}$,\ $[e_{2},e_{4}]=\alpha e_{6}$,\ $[e_{4},e_{4}]=e_{6}$ ($\alpha\neq0;\ \alpha\sim -\alpha$)\\ \hline
$\omega^{4,1}_{8}(\alpha,\beta)$ & $[e_{2},e_{1}]=e_3+e_{6}$,\ $[e_{4},e_{1}]=\alpha e_{6}$,\ $[e_{5},e_{1}]=\beta e_{6}$,\ $[e_{2},e_{4}]=\beta e_{6}$,\ $[e_{4},e_{4}]=e_{6}$ ($\alpha\neq0,\ \beta\neq0;\ (\alpha,\beta)\sim(-\alpha,-\beta)$)\\ \hline
$\omega^{4,1}_{9}$ & $[e_{4},e_{1}]=e_{6}$\\ \hline
$\omega^{4,1}_{10}(\alpha)$ & $[e_{4},e_{1}]=e_{6}$,\ $[e_{4},e_{4}]=\alpha e_{6}$ ($\alpha\neq0;\ \alpha\sim -\alpha$)\\ \hline
$\omega^{4,1}_{11}$ & $[e_{4},e_{1}]=e_{6}$,\ $[e_{5},e_{1}]=e_{6}$,\ $[e_{2},e_{4}]=e_{6}$\\ \hline
$\omega^{4,1}_{12}(\alpha)$ & $[e_{4},e_{1}]=e_{6}$,\ $[e_{5},e_{1}]=e_{6}$,\ $[e_{2},e_{4}]=e_{6}$,\ $[e_{4},e_{4}]=\alpha e_{6}$ ($\alpha\neq0;\ \alpha\sim -\alpha$)\\ \hline
$\omega^{4,1}_{13}$ & $[e_{4},e_{4}]=e_{6}$\\ \hline
$\omega^{4,1}_{14}$ & $[e_{5},e_{1}]=e_{6}$,\ $[e_{2},e_{4}]=e_{6}$,\ $[e_{4},e_{4}]=e_{6}$\\ \hline
$\omega^{4,1}_{15}$ & $[e_{5},e_{1}]=e_{6}$,\ $[e_{2},e_{4}]=e_{6}$\\ \hline
\end{longtable}
\endgroup

\begingroup\footnotesize
\begin{longtable}{|p{2.0cm}|p{14.0cm}|}
\hline
\multicolumn{2}{|c|}{Central extensions of $\mu_5^4$; dimension $7$}\\ \hline
$\omega^{4,2}_{1}$ & $[e_{2},e_{1}]=e_3+e_{6}$,\ $[e_{4},e_{1}]=e_{7}$\\ \hline
$\omega^{4,2}_{2}$ & $[e_{2},e_{1}]=e_3+e_{6}$,\ $[e_{4},e_{1}]=e_{7}$,\ $[e_{4},e_{4}]=e_{6}$\\ \hline
$\omega^{4,2}_{3}(\alpha)$ & $[e_{2},e_{1}]=e_3+e_{6}$,\ $[e_{4},e_{1}]=e_{7}$,\ $[e_{5},e_{1}]=\alpha e_{6}$,\ $[e_{2},e_{4}]=\alpha e_{6}$ ($\alpha\neq0;\ \alpha\sim -\alpha$)\\ \hline
$\omega^{4,2}_{4}(\alpha)$ & $[e_{2},e_{1}]=e_3+e_{6}$,\ $[e_{4},e_{1}]=e_{7}$,\ $[e_{4},e_{4}]=\alpha e_{7}$ ($\alpha\neq0;\ \alpha\sim -\alpha$)\\ \hline
$\omega^{4,2}_{5}$ & $[e_{2},e_{1}]=e_3+e_{6}$,\ $[e_{4},e_{1}]=e_{7}$,\ $[e_{5},e_{1}]=e_{7}$,\ $[e_{2},e_{4}]=e_{7}$\\ \hline
$\omega^{4,2}_{6}(\alpha)$ & $[e_{2},e_{1}]=e_3+e_{6}$,\ $[e_{4},e_{1}]=e_{7}$,\ $[e_{5},e_{1}]=\alpha e_{6}$,\ $[e_{2},e_{4}]=\alpha e_{6}$,\ $[e_{4},e_{4}]=e_{6}$ ($\alpha\neq0;\ \alpha\sim -\alpha$)\\ \hline
$\omega^{4,2}_{7}(\alpha)$ & $[e_{2},e_{1}]=e_3+e_{6}$,\ $[e_{4},e_{1}]=e_{7}$,\ $[e_{4},e_{4}]=e_{6}+\alpha e_{7}$ ($\alpha\neq0;\ \alpha\sim -\alpha$)\\
\hline
\end{longtable}
\endgroup

\begingroup\footnotesize
\begin{longtable}{|p{2.0cm}|p{14.0cm}|}
\hline
\multicolumn{2}{|c|}{Central extensions of $\mu_5^4$; dimension $8$}\\ \hline
$\omega^{4,3}_{1}$ & $[e_{2},e_{1}]=e_3+e_{6}$,\ $[e_{4},e_{1}]=e_{7}$,\ $[e_{4},e_{4}]=e_{8}$\\ \hline
$\omega^{4,3}_{2}(\alpha)$ & $[e_{2},e_{1}]=e_3+e_{6}$,\ $[e_{4},e_{1}]=e_{7}$,\ $[e_{5},e_{1}]=\alpha e_{6}$,\ $[e_{2},e_{4}]=\alpha e_{6}$,\ $[e_{4},e_{4}]=e_{8}$ ($\alpha\neq0;\ \alpha\sim -\alpha$)\\ \hline
$\omega^{4,3}_{3}$ & $[e_{2},e_{1}]=e_3+e_{6}$,\ $[e_{4},e_{1}]=e_{7}$,\ $[e_{5},e_{1}]=e_{7}$,\ $[e_{2},e_{4}]=e_{7}$,\ $[e_{4},e_{4}]=e_{8}$\\ \hline
$\omega^{4,3}_{4}$ & $[e_{2},e_{1}]=e_3+e_{6}$,\ $[e_{4},e_{1}]=e_{7}$,\ $[e_{5},e_{1}]=e_{8}$,\ $[e_{2},e_{4}]=e_{8}$,\ $[e_{4},e_{4}]=e_{8}$\\ \hline
$\omega^{4,3}_{5}(\alpha)$ & $[e_{2},e_{1}]=e_3+e_{6}$,\ $[e_{4},e_{1}]=e_{7}$,\ $[e_{5},e_{1}]=\alpha e_{6}+e_{7}$,\ $[e_{2},e_{4}]=\alpha e_{6}+e_{7}$,\ $[e_{4},e_{4}]=e_{8}$ ($\alpha\neq0;\ \alpha\sim -\alpha$)\\ \hline
$\omega^{4,3}_{6}(\alpha)$ & $[e_{2},e_{1}]=e_3+e_{6}$,\ $[e_{4},e_{1}]=e_{7}$,\ $[e_{5},e_{1}]=\alpha e_{6}+e_{8}$,\ $[e_{2},e_{4}]=\alpha e_{6}+e_{8}$,\ $[e_{4},e_{4}]=e_{8}$ ($\alpha\neq0;\ \alpha\sim -\alpha$)\\ \hline
$\omega^{4,3}_{7}(\alpha)$ & $[e_{2},e_{1}]=e_3+e_{6}$,\ $[e_{4},e_{1}]=e_{7}$,\ $[e_{5},e_{1}]=e_{7}+\alpha e_{8}$,\ $[e_{2},e_{4}]=e_{7}+\alpha e_{8}$,\ $[e_{4},e_{4}]=e_{8}$ ($\alpha\neq0;\ \alpha\sim -\alpha$)\\
\hline
\end{longtable}
\endgroup

\begingroup\footnotesize
\begin{longtable}{|p{2.0cm}|p{14.0cm}|}
\hline
\multicolumn{2}{|c|}{Central extensions of $\mu_5^4$; dimension $9$}\\ \hline
$\omega^{4,4}_{1}$ & $[e_{2},e_{1}]=e_3+e_{6}$,\ $[e_{4},e_{1}]=e_{7}$,\ $[e_{5},e_{1}]=e_{9}$,\ $[e_{2},e_{4}]=e_{9}$,\ $[e_{4},e_{4}]=e_{8}$\\ \hline
\end{longtable}
\endgroup

\end{document}